\documentclass[11pt]{article}
\usepackage{amssymb}
\usepackage{amsbsy}
\usepackage{amsthm}
\usepackage{graphicx} 
\usepackage{subfigure}
\usepackage{pst-eucl}
\usepackage[latin1]{inputenc}
\usepackage[english]{babel}
\usepackage{amsmath,amssymb,graphics,mathrsfs}
\usepackage{amsmath,amssymb,latexsym}
\usepackage{float}
\usepackage{graphicx,xcolor}
\usepackage[T1]{fontenc}
\usepackage[active]{srcltx}
\usepackage{multicol}
\usepackage{pst-all}
\usepackage{enumerate}
\usepackage{pstricks}
\usepackage{pstricks-add}
\usepackage{setspace}
\usepackage{soul}
\usepackage{cancel}
\usepackage{epsfig}
\usepackage{bm}
\usepackage{comment}
\usepackage{nonfloat}
\usepackage{caption}[=v1]

\usepackage[left=3cm,top=3cm,right=2.4cm,bottom=3.2cm]{geometry}
\usepackage[colorlinks=true,citecolor=red,linkcolor=blue,urlcolor=red,pdfpagetransition=Blinds,pdftoolbar=false,pdfmenubar=false]{hyperref}
\usepackage{cleveref}

\usepackage[
backend=bibtex,
style=numeric,
sorting=nty,
maxbibnames = 99
]{biblatex}
\newcommand{\p} {\hbox{\rm I \kern -5pt P}}

\def\L       {{\boldsymbol L}}
\def\R       {\mathbb{R}}
\def\L  {L}
\def\Om       {\Omega}

\def\eps        {\varepsilon}
\def \hueco{\noalign{\medskip}}

\def \beq{\begin{equation}}
\def \eeq{\end{equation}}
\def \ba{\begin{array}}
\def \ea{\end{array}}
\def \dis{\displaystyle}

\newcommand{\Q}{\textbf{Q}}
\newcommand{\trace}{\texttt{tr}}

\newcommand{\I}{\mathbf{I}}

\def \dt{\Delta t}

\newtheorem{obs}{Remark}
\newtheorem{lem}{Lemma}
\newtheorem{thm}{Theorem}

\newtheorem{defn}{Definition}

\begin{document}

\title{Linear Numerical Schemes for a $\Q$-Tensor System for Nematic Liquid Crystals} 

\author{Justin Swain \&
 Giordano Tierra\thanks{Department of Mathematics, University of North Texas, Denton TX (USA). Emails:  JustinSwain2@my.unt.edu, gtierra@unt.edu}}

\maketitle

\begin{abstract}
In this work, we present three linear numerical schemes to {model} nematic liquid crystals using the Landau-de Gennes $\Q$-tensor theory. 
The first scheme is based on using a truncation procedure of the energy, which allows for an unconditionally energy stable first order accurate decoupled scheme. 
The second scheme uses a modified second order accurate optimal dissipation algorithm, which gives a second order accurate coupled scheme. 
Finally, the third scheme uses a new idea to decouple the unknowns from the second scheme which allows us to obtain accurate dynamics while improving computational efficiency. 
We present several numerical experiments to offer a comparative study of the accuracy, efficiency and the ability of the numerical schemes to represent realistic dynamics.

\end{abstract}

\section{Introduction}
Liquid crystals (LCs) are important materials that are used in several technological applications \cite{Lagerwall12}.  The most common usage is in the omnipresent liquid crystal displays which uses the birefringence property of the material to create images on a screen \cite{Chen18}. However, LCs also respond to other external stimuli, e.g. magnetic, mechanical, chemical, which can be used to induce complex shape changes in the material used for applications in biomedical devices, robotics, optics, textiles, and sensors \cite{Chen23,Pilz,Li17,Li22,Roach,Woltman}. \\
Liquid crystals exhibit intermediary phases between solid and liquid. One such phase is the nematic phase which possess the microscopic orientational order of a crystalline solid, however, the molecules have no positional order but flow freely past each other and thus display macroscopic properties of a liquid. Models of LCs usually represent molecules as rods or disks and use some parameter to describe the orientation of the molecules. Popular models of LCs are the Oseen-Frank director theory \cite{Frank,Oseen}, the Ericksen-Leslie formulation \cite{Ericksen62,Ericksen76,Leslie66,Leslie68}, and the hydrodynamic Beris-Edwards model \cite{BerisEdwards}. \\
In this work, we focus on the Landau-de Gennes $\Q$-tensor theory \cite{deGennes,MottramNewton,SonnetVirga} which can be seen as a subsystem of the Beris-Edwards model without hydrodynamic effects. In the $\Q$-tensor model the orientation of the nematic is associated with a second-order tensor field which measures the local average deviation of the material from the isotropic state. This model has the advantage of being able to simulate complex defect dynamics due to its ability to describe more general defect geometries, and variable degrees of order \cite{Ball}. \\
Several analytic results exist concerning the well-posedness of the PDE associated to the $\Q$-tensor problem coupled with the Navier-Stokes equations \cite{Abels,Abels14,ClimentGuillen,Guillen15,Paicu12,Xu}. The development of numerical schemes for this problem is relatively new, and faces certain challenges including: (1) the complex dynamics relating the physical properties of these materials cause the systems to be very large with several equations and unknowns; and (2) the equations are highly nonlinear with terms that couple the unknowns from different equations. Since the model involves a large system of coupled nonlinear PDEs we expect the discretization of this system to involve large and computationally expensive systems to solve. \\
Recently, we have seen new numerical methods proposed for modeling LC dynamics using the $\Q$-tensor theory. For example, in \cite{CaiShenXu} the authors present a first order accurate energy stable scheme for the 2D $\Q$-tensor model with hydrodynamic effects. A stabilized convex splitting idea is used to develop a first order accurate decoupled energy stable scheme in \cite{Zhao16}. In \cite{Zhao17}, a second order accurate scheme is presented for the hydrodynamic $\Q$-tensor model using energy quadratization (IEQ) \cite{GongZhaoWangIEQ}. Error estimates for the IEQ approach has been recently presented to the $\Q$-tensor model \cite{Gudibanda22}, and the Beris-Edwards model \cite{Webber23}. Another idea presented in \cite{ShenXuYang} uses a scalar auxiliary variable (SAV) technique \cite{GongZhaoWangSAV} to construct a second order accurate coupled scheme for a $\Q$-tensor problem. Another recent and interesting paper explored numerical schemes for a hydrodynamic $\Q$-tensor model in thin films \cite{Bouck21}. In addition, work has been done on energy minimization techniques to compute equilibrium solutions for the $\Q$-tensor problem such as \cite{BorthagarayWalker,SchimmingVinalsWalker,SurowiecWalker}. \\
In this paper, we propose three new efficient linear numerical schemes for simulating nematic liquid crystals using a $\Q$-tensor model. The first scheme adapts an energy truncation procedure \cite{Cabrales15} which allows for an unconditionally energy stable first order accurate decoupled scheme. The second scheme uses a modified second order accurate optimal dissipation algorithm \cite{Tierra14} which gives a second order accurate coupled scheme. Finally, the third scheme uses a technique to decouple the unknowns from the second scheme which allows us to recover accurate dynamics while improving computational efficiency.  \\
The rest of the paper is organized as follows. In \Cref{sec:model} we describe the $\Q$-tensor model and some analytical results that are known about the problem. Next, we introduce the numerical schemes and their properties in \Cref{sec:numericalSchemes}. In \Cref{sec:numericalResults} we provide computational results to justify the applicability and efficiency our numerical schemes. Finally, we state the conclusions of our work in \Cref{sec:conclusion} .

\section{Landau-deGennes $\Q$-Tensor Model}\label{sec:model}
In this section we provide an overview of the Landau-de Gennes $\Q$-Tensor model pertaining to the nematic phase of liquid crystals \cite{deGennes}. An alternative to the Oseen-Frank director theory \cite{Frank}, is to use the second order symmetric and traceless tensor $\Q$ representing the statistical second moment of the nematic system from the isotropic phase. In the isotropic phase $\Q=\mathbf{0}$, and nonzero in the nematic phase. The eigenvalues of $\Q$ describe further phase transitions with the uniaxial phase corresponding to the case where $\Q$ has two equal eigenvalues, and otherwise the system is described as biaxial \cite{MottramNewton}. 
Let $\Om\subset\R^d$, $d=2,3$, be a bounded domain with polyhedral boundary $\partial\Om$. We will use the following tensor spaces
$$\ba{rcl}\dis
\Lambda
&:=&
\{ \Q\in\R^{d\times d} : \Q^T = \Q \}\,, \\
\Lambda_0
&:=&
\{ \Q\in\R^{d\times d} : \Q^T = \Q, \trace(\Q)=0 \}\,.
\ea
$$
Each $\Q\in\Lambda$ is uniquely determined by six variables $\{ q_i \}_{i=1}^6$ thus can be written as
$$
\Q 
\,=\, 
\begin{pmatrix}
Q_{11} & Q_{12} & Q_{13} \\
Q_{21} & Q_{22} & Q_{23} \\
Q_{31} & Q_{32} & Q_{33} 
\end{pmatrix},
$$
and each $\Q\in\Lambda_0$ can be uniquely determined by five unknowns (with $Q_{33} = -(Q_{11} + Q_{22})$). 
Here we use spatial coordinates $\mathbf{x}\in\R^d$. It is useful to write $\Q$ in its eigenframe
$$
\Q = s_1(\mathbf{v}_1\otimes\mathbf{v}_1) + s_2(\mathbf{v}_2\otimes\mathbf{v}_2) - \frac13(s_1 + s_2)\I \,,
$$
where $\I$ is the identity matrix, $\mathbf{v}_i$, are the orthonormal eigenvectors of $\Q$, and $s_i$ are linear combinations of the corresponding eigenvalues $\lambda_i$ given by
$$
\lambda_1
\,=\,
\frac13(2s_1 - s_2) \,, \quad
\lambda_2
\,=\,
-\frac13(s_1 + s_2) \,, \quad
\lambda_3
\,=\,
\frac13(2s_2 - s_1)\,. 
$$
Since $\trace(\Q)=0$, we have that $\sum_i \lambda_i = 0$. Furthermore, the eigenvalues are restricted to the range $-\frac13 \leq \lambda_i \leq \frac23$ \cite{SonnetVirga}. 

The Landau-deGennes free energy function is defined as:
\beq\label{eq:energyLdG}\dis
E(\Q)
\,:=\,
\int_\Omega \left( \mathcal{W}(\Q,\nabla \Q) + \frac1\varepsilon\Psi(\Q) \right) d\mathbf{x} \,,
\eeq
with $\mathcal{W}(\Q,\nabla \Q)$ being the elastic energy density and $\Psi(\Q)$ representing the thermotropic energy density, respectively. Moreover, $\varepsilon>0$ is a dimensionless parameter balancing the energetic contributions of the elastic and thermotropic energies.

\subsection{Elastic Energy Density}
The elastic energy is given by
\beq\label{eq:elasticEnergy}\dis
\mathcal{W}(\Q,\nabla \Q)
\,:=\,
\frac12 \left(
 L_1 \left| \nabla \Q \right|^2 + L_2 \left| \nabla \cdot \Q\right|^2 + L_3 \left(\nabla\Q\right)^T \,\vdots\, \nabla\Q
+ L_4 \nabla\Q\,\vdots\,(\bm{\epsilon}\cdot\Q) + \L_5 \nabla\Q \,\vdots\, ( \Q \cdot \nabla)\Q ) \right) ,
\eeq
where $L_i$, $i=1,\dots, 5$ are the elastic material parameters and $\vdots$ represents the scalar product between third order tensors. Then, using the convention of summation over repeated indices, the elastic energy terms can be written as
$$
\dis
\left| \nabla \Q \right|^2 
:=
\frac{\partial}{\partial x_k} Q_{ij} \frac{\partial}{\partial x_k} Q_{ij} \,, \quad \quad
\left| \nabla \cdot \Q\right|^2
:=
\frac{\partial Q_{ij}}{\partial x_j}\frac{\partial Q_{ik}}{\partial x_k} \,, \quad \quad
\left(\nabla\Q\right)^T \, \vdots \, \nabla\Q
:=
\frac{\partial Q_{ik}}{\partial x_j}\frac{\partial Q_{ij}}{\partial x_k} \,,
$$
$$
\nabla\Q\,\vdots\,(\bm{\epsilon}\cdot\Q)
:=
\bm{\epsilon}_{ijk} Q_{lj}\frac{\partial Q_{ij}}{\partial x_k} \,, \quad \quad
\nabla\Q \,\vdots\, ( \Q \cdot \nabla)\Q )
:=
Q_{lk}\frac{\partial Q_{ij}}{\partial x_l}\frac{\partial Q_{ij}}{\partial x_k} \,,
$$
where $\bm{\epsilon}$ is the Levi-Civita permutation symbol. The five elastic constants $L_i$ can be related to the Oseen-Frank model constants $K_i$, for $i=1,\dots,4$, and chirality $q_0$ \cite{MottramNewton}. In this work we consider the one constant Landau-deGennes model with $L_i=0$ for $i=2,\dots,5$.

\subsection{Landau-deGennes Potential Function}
The second term of the energy corresponds to the thermotropic contribution. The Landau-de Gennes potential functional $\Psi(\Q)$ describes the various states of the nematic system, i.e. uniaxial nematic, biaxial nematic, or isotropic liquid. The most general form of this function is written as 
\beq\label{eq:psi}
\Psi(\Q)
\,=\,
\frac{A}2\trace(\Q^2) - \frac{B}3\trace(\Q^3)+\frac{C}4\trace(\Q^2)^2\,.
\eeq
The parameters $A,B$, and $C$ depend on the material being studied. We consider $B,C>0$ and independent of temperature $\theta$, while $A(\theta)=\widehat{\gamma}(\theta - \theta^*)$, with $\widehat{\gamma}>0$ and $\theta^*$ being the isotropic-nematic phase transition temperature \cite{Majumdar,MottramNewton}. Here, $\trace(\Q):=Q_{ii}$, and the higher order terms can be written as
$$
\dis
\trace(\Q^2)
\,:=\,
Q_{ij}Q_{ij} \,, \quad
\trace(\Q^3)
\,:=\,
Q_{ij}Q_{jk}Q_{ki} \,.
$$
It is known that when $B\neq0$ the symmetric and traceless tensors which minimize \Cref{eq:psi} will have two equal eigenvalues, i.e. the uniaxial state \cite{SonnetVirga}. Minimizing the total free energy involves a balance of competing energy terms from the elastic and thermotropic components. In this way, it is important to choose parameters $L_i$, and $A,B,C$ properly to ensure the energy functional $E(\Q)$ is bounded from below \cite{Majumdar}.

\subsection{System Dynamics}
The dynamics of the system are given by using an $L^2$-gradient flow
$$
\Q_t + \gamma\frac{\delta E(\Q)}{\delta \Q}\,=\,\textbf{0}\,,
$$ 
with $\gamma$ being a relaxation parameter and the variational derivative of the energy given by
$$
\frac{\delta E(\Q)}{\delta \Q}
\,=\,
- \Delta \Q 
+ \frac{1}{\varepsilon}\nabla \Psi (\Q) \,,
$$
which results in the following formulation of the problem: Given $\Q_0(\mathbf{x}):\Om\rightarrow\Lambda_0$, find $\Q(\mathbf{x},t):\Om\times(0,T) \rightarrow\Lambda_0$ such that
\beq\label{eq:SF}
\left\{\ba{rclr}\dis
\Q_t  +  \gamma \left( -\Delta \Q 
+ \frac{1}{\varepsilon}\bm{\psi} (\Q) \right)
&=&
 \mathbf{0}
&\quad \mbox{ for } (\mathbf{x},t) \in \Omega \times (0,T)\,, \\
\partial_\mathbf{n}\Q 
&=& 
\mathbf0 
&\quad \mbox{ for } (\mathbf{x},t) \in \partial\Omega \times (0,T)\,,\\
\Q(\mathbf{x},0)
&=&
\Q_0(\mathbf{x})
&\quad \mbox{ for } \mathbf{x} \in \Omega \,,
\ea \right.
\eeq
where $\mathbf{n}$ denotes the unit outward normal vector on $\partial\Omega$, $\partial_\mathbf{n}\Q=\nabla\Q\cdot\mathbf{n}$, and $T>0$ is the final time. Here, the gradient of $\Psi$ with respect to $\Q$ is a second order tensor, $\bm{\psi}$, given by
\beq\label{eq:gradPsi}\dis
\bm{\psi} (\Q) := \nabla \Psi(\Q) = A\Q - B\Q^2 + C\trace(\Q^2)\Q\,.
\eeq
Additionally, we write the second derivative of $\Psi$, which is a fourth order tensor, $\nabla \bm{\psi} (\Q)$, as
\beq\label{eq:hessianPsi}\dis
\left[ \nabla \bm{\psi}(\Q) \right]_{ijkl} = A\delta_{ik}\delta_{jl} 
- B\left( Q_{lj}\delta_{ik} + Q_{ik}\delta_{jl} \right)
+ C\left(\trace(\Q^2)\delta_{ik}\delta_{jl} + 2Q_{ij}Q_{kl} \right) \,,
\eeq
with $\delta_{ij}$ denoting the Kronecker delta function.
\begin{obs}
See appendix for details on the computations in \eqref{eq:hessianPsi}.
\end{obs}
\subsection{Traceless Formulation}
Following the ideas in \cite{Guillen15}, we provide a reformulation of the Landau-de Gennes model to facilitate the development of numerical schemes. First, to maintain the traceless condition on $\Q$ we replace the variational derivative with the following expression:
$$
\Q_t + \gamma \left(\frac{\delta E(\Q)}{\delta \Q} - \frac13\trace\left(\frac{\delta E(\Q)}{\delta \Q} \right) \right)
\,=\,
\mathbf0 \,.
$$
Hence, the problem reads:
\beq\label{eq:SFmod}
\left\{\ba{rclr}\dis
\Q_t  + \gamma\left[ - \Delta \Q 
+ \frac{1}{\eps}\left( A\Q - B \left(\Q^2\right) + C\trace(\Q^2)\Q \right)
+ \frac{1}{\eps} \bm{p}(\Q) \right]
&=&
 \mathbf{0}
&\quad \mbox{ for } (\mathbf{x},t) \in \Omega \times (0,T) \,, \\
\partial_\mathbf{n}\Q 
&=& 
\mathbf0 
&\quad \mbox{ for } (\mathbf{x},t) \in \partial\Omega \times (0,T) \,,\\
\Q(\mathbf{x},0)
&=&
\Q_0(\mathbf{x})
&\quad \mbox{ for } \mathbf{x} \in \Omega\,,
\ea \right. 
\eeq
where the function $\bm{p}(\Q)$ represents the non-zero trace part of the problem: 
$$
\bm{p}(\Q) = -\frac{B}{3}\trace(\Q^2)\I\,.
$$

\begin{obs}
Another possible way to enforce the traceless condition on $\Q$ is to derive the dynamics by taking the variational derivative with respect to a symmetric and traceless tensor which will intrinsically impose the condition \cite{SonnetVirga,SurowiecWalker}.
\end{obs}

Now we summarize various results presented in \cite{Guillen15} about the problem \eqref{eq:SF} and the modified problem \eqref{eq:SFmod}, both coupled with hydrodynamics effects.
\begin{thm}[Well-Posedness]\label{thm:wellPosedness}
Let $\Q_0\in H^1(\Om)$. Then there exists a weak solution $\Q$ of the problem \eqref{eq:SFmod} such that,
$$
\Q\in L^\infty(0,\infty;H^1(\Om))\cap L^2(0,T;H^2(\Om)), \quad \forall T>0\,.
$$
Furthermore, if $\Q$ is a solution to \eqref{eq:SFmod} in $(0,T)$ for some fixed $T>0$, and 
$$
\Delta \Q \in L^{\frac{2s}{2s-3}}(0,T;L^s(\Om)) \mbox{ for some } 2\leq s\leq 3 \,,
$$
then the solution is unique.
\end{thm}

\begin{lem}[Traceless Property]\label{lem:traceless}
    If $\trace(\Q_0) = 0$, then $\trace(\Q)=0$ for all $t>0$.
\end{lem}
Additionally, this continuous problem possesses the following maximum principle on the norm of $\Q$. Here, we use the Frobenius norm $|\Q|=\sqrt{\Q:\Q}$ where $\Q:\Q=Q_{ij}Q_{ij}$.

\begin{lem}[Maximum Principle]\label{lem:maximumPrinciple}
Let $\alpha>0$ satisfy
\beq\label{eq:alpha}
\alpha^2 \geq \frac{B^2}{C^2} - \frac{2A}{C} \,,
\eeq
where $A,B$, and $C$ are the coefficients in the potential function $\Psi(\Q)$. 
Let $\Q_0\in H^1(\Omega)$ with $|\Q_0|\leq \alpha$ a.e. in $\Omega$. If $\Q\in L^\infty(0,T;H^1(\Omega))\cap L^2(0,T;H^2(\Omega))$ is any point-wise solution for the problem \eqref{eq:SFmod} in $(0,T)\times\Omega$ then
$$
|\Q(\mathbf{x},t)|\leq \alpha \mbox{ a.e. } (\mathbf{x},t)\in (0,T)\times\Omega\,.
$$
\end{lem}

\begin{lem}[Energy Law]\label{lem:continuousEnergyLaw}
The problem satisfies the following dissipative energy law
\beq\label{eq:continuousEnergyLaw}
    \frac{d}{dt} E(\Q) + \frac1\gamma \| \Q_t \|^2_{L^2} = 0 \,.
\eeq
\end{lem}

\section{Numerical Schemes}\label{sec:numericalSchemes}
The main challenge for developing numerical schemes to approximate solutions to \eqref{eq:SF} is the nonlinearity of the unknowns in the potential function $\Psi(\Q)$. There are two main objectives we pursue to design numerical schemes. The first goal is to design schemes that obey a dissipative energy law for the discrete problem similar to the one associated with the continuous problem (equation \eqref{eq:continuousEnergyLaw}), so that we can guarantee decreasing energy in time. The second will be to stay as close as possible to the energy law at the discrete level so that we can recover the dynamics of the system more accurately. \\
Within this study, we intend to introduce a linear, unconditionally energy-stable Finite Element numerical method for the simulation of nematic liquid crystals. Subsequently, we explore a series of numerical techniques that enhance computational efficiency but come at the expense of compromising the structure-preserving attributes of the initial method. Notably, substantial reductions in computational cost can be achieved by reconfiguring the problem formulation to decouple the six unknown variables. This modification allows us to tackle multiple smaller problems sequentially, as opposed to dealing with a single, large linear system. Indeed, we will demonstrate the feasibility of implementing this concept for three-dimensional simulations with a relatively fine spatial mesh.

\subsection{A Finite Element Space-Discrete Scheme}
In this section we develop numerical schemes for the weak formulation of the modified problem \eqref{eq:SFmod}. The $L^2$ inner product is denoted by $(\cdot , \cdot )$. We consider a Finite Element method for space discretization as follows. Let $\bm{K}_h \subset H^1(\Omega)$ be a conformed finite element space associated to a regular and quasi-uniform triangulation $\mathcal{T}_h$ of the domain $\Omega$ whose polyhedric boundary is denoted by $\partial\Omega$. For the sake of notation, we skip the use of the subscript $h$ to denote functions that are discrete in space. Then the Finite Element approximation of \eqref{eq:SFmod} is: Find $\Q(\mathbf{x},t)\in\bm{K}_h$ such that
$$
\Q(\mathbf{x},0) = P_{\bm{K}_h}\Q_0\, ,
$$
where $P_{\bm{K}_h}$ denotes the projection operator into the space $\bm{K}_h$, and 
$$
\left( \Q_t, \overline \Q \right)
+ \gamma \left[\left( \nabla \Q, \nabla \overline\Q \right) 
+ \frac1\eps \left( \bm{\psi}(\Q), \overline\Q \right)
+ \frac1\eps \left(\bm{p}(\Q), \overline\Q \right) \right]
\, = \, 
0 \,,
$$
for any $\overline\Q\in\bm{K}_h$.

\subsection{Fully Discrete Schemes}
We consider a Finite Difference method in time using a regular partition of the interval $[0,T]$ into steps of length $\dt>0$. We denote $t^n:=n\dt$, $\Q^n$ represents the approximation of $\Q(\mathbf{x},t^n)$ with $\Q^0=P_{\bm{K}_h}\Q_0$,
$$
\dis
\delta_t u^{n+1} 
\,:=\,
\frac{u^{n+1} - u^n}{\dt} \quad\quad \mbox{ and } \quad\quad
u^{n+\frac12} 
\,:=\,
\frac{u^{n+1} + u^n}{2}\,.
$$
Recall that the nonlinearity of the model exists in the potential function $\bm{\psi}(\Q)$. The main contribution of our work relies in proposing numerical schemes with linear approximations to $\bm{\psi}(\Q)$ for which we write $\bm{\psi}^{\dt}(\Q^{n+1},\Q^n)$. \\
We propose the following generic numerical scheme.
\begin{itemize}
\item Initialization: Let $\Q^0\in \bm{K}_h$ satisfying $\Q^0\in\Lambda_0$, and $|\Q^0|\leq\alpha$ a.e. in $\Om$. \\
\item Step $n+1$: given $\Q^n\in\bm{K}_h$, find $\Q^{n+1}\in\bm{K}_h$ such that
\beq
\ba{rcl}\label{eq:numericalScheme}\dis
\left( \delta_t \Q^{n+1} , \overline{\Q} \right) 
+ \gamma \left[\left( \nabla \Q^{n+\frac{1}{2}} , \nabla \overline{\Q} \right) 
+ \frac{1}{\eps}\left( \bm{\psi}^{\Delta t} (\Q^{n+1}, \Q^{n}), \overline{\Q} \right)
+  \frac1\eps \left( \bm{p}(\Q^{n+1},\Q^n) , \overline{\Q} \right) \right]
&=&
0\,,
\ea
\eeq
for all $\overline\Q \in \bm{K}_h$.
\end{itemize}

\begin{lem}[Discrete Energy Law]\label{lem:discreteEnergyLaw}
    The generic numerical scheme \eqref{eq:numericalScheme} satisfies a discrete version of the energy law \eqref{eq:continuousEnergyLaw}:
        \beq\label{eq:discreteEnergyLaw}\dis
        \delta_t E(\Q^{n+1}) + \frac1\gamma \frac{1}{(\Delta t)^2} || \Q^{n+1} - \Q^n ||_{L^2}^2 + \mathbf{ND}\left(\Q^{n+1}, \Q^n\right) = 0 \, ,
    \eeq
    where $\mathbf{ND}\left(\Q^{n+1}, \Q^n\right)$ denotes the numerical dissipation
    \beq\label{eq:numericalDissipation}\dis
    \ba{rcl}
        \mathbf{ND}\left(\Q^{n+1}, \Q^n\right) 
        &:=& \dis
        \frac1{\eps\dt} \int_\Omega \left( \bm{\psi}^{\Delta t} (\Q^{n+1},\Q^n) : (\Q^{n+1} - \Q^n) \right) - \left( \Psi(\Q^{n+1}) - \Psi(\Q^n) \right) d\mathbf{x} \\
        \hueco
        & &\dis
        + \frac1{\eps\dt}\int_\Om \bm{p}(\Q^{n+1},\Q^n)\left( \trace(\Q^{n+1}) - \trace(\Q^n) \right).
    \ea
    \eeq
\end{lem}
\begin{proof}
    Take $\overline{\Q}=\frac1\gamma \delta_t\Q^{n+1}$ in \eqref{eq:numericalScheme}. We obtain
$$
        \ba{rcl}\dis
        \frac{1}{\dt}\left( \frac{1}{2} || \nabla \Q^{n+1} ||_{L^2} -  \frac{1}{2} || \nabla \Q^{n} ||_{L^2} \right)         
        + \frac1\gamma\frac{1}{(\dt)^2} || \Q^{n+1} - \Q^n ||_{L^2}^2 
        & & \\
        \hueco \dis
        + \frac{1}{\eps\dt } \left( \bm{\psi}^{\Delta t} (\Q^{n+1}, \Q^{n}) , (\Q^{n+1} - \Q^n ) \right) 
        + \frac1{\eps\dt} \left( \bm{p}(\Q^{n+1},\Q^n) ,(\Q^{n+1} - \Q^n ) \right)
        &=&
         0\,.
    \ea
$$
Therefore, by adding and subtracting $\eps^{-1}\left( \Psi(\Q^{n+1}) - \Psi(\Q^n) \right)$,
$$
    \ba{rcl}\dis
       \frac{E(\Q^{n+1}) - E(\Q^{n})}{\dt}
        + \frac1\gamma \frac{1}{(\dt)^2} || \Q^{n+1} - \Q^n ||_{L^2}^2 
        & & \\
        \hueco \dis
        + \frac1{\eps\dt} \int_\Omega \left( \bm{\psi}^{\dt} (\Q^{n+1},\Q^n) : (\Q^{n+1} - \Q^n) \right) - \left( \Psi(\Q^{n+1}) - \Psi(\Q) \right) d\mathbf{x} 
        & & \\
        \hueco \dis
        +  \frac1{\eps\dt} \int_\Om \bm{p}(\Q^{n+1},\Q^n)\left(\trace(\Q^{n+1}) - \trace(\Q^n) \right) d\mathbf{x} 
        &=& 0 \,.
    \ea
$$
\end{proof}

\begin{defn}\label{def:energyStable}
A numerical scheme is {\bf energy stable} if 
\beq\label{eq:energyStable}
\dis \frac{E(\Q^{n+1}) - E(\Q^n)}{\dt} + \frac1\gamma || \delta_t\Q^{n+1} ||_{L^2}^2 
\,\leq\,
 0 \quad\quad \forall\, n\geq 0\,.
\eeq
In particular, energy stability implies the decreasing energy property $E(\Q^{n+1}) \leq E(\Q^n)$ for all $n\geq 1$.
Moreover, if the scheme can be shown to satisfy \eqref{eq:energyStable} for any choice of $\dt>0$, then the scheme is called {\bf unconditionally energy stable}.
\end{defn}
\begin{obs}
From \eqref{eq:discreteEnergyLaw}, the numerical scheme \eqref{eq:numericalScheme} is energy stable if 
\beq\label{eq:stabilityCondition}\dis
\mathbf{ND}(\Q^{n+1},\Q^n) \geq 0\,.
\eeq
The main contribution of this paper is to present different choices of $\bm{\psi}^{\dt}(\Q^{n+1}, \Q^n)$ so that we may control the sign and/or size of the numerical dissipation \eqref{eq:numericalDissipation}. The key point of our approach is to derive numerical schemes whose solutions are as close to the solution of the original problem as possible. In fact, one of the goals of this work is to emphasize that even schemes that are not unconditionally energy stable can be reliable if the size of the numerical dissipation is small. 
\end{obs}
In the next sections we develop three choices for $\bm{\psi}^{\dt}(\Q^{n+1}, \Q^n)$ along with their properties.

\subsubsection{First Order Unconditionally Energy Stable Decoupled Scheme (UES1D)}
Here, we present a linear, first order accurate in time, unconditionally energy stable numerical scheme. 
Let $\alpha>0$ satisfy \eqref{eq:alpha} from the maximum principle on $|\Q|$. Consider splitting $\Psi(\Q)$ into three terms:
\beq\label{eq:splitPsi}
\ba{rcl}\dis
\Psi(\Q)
&=&
\Psi_1 (\Q) + \Psi_2(\Q) + \Psi_3(\Q) \,, 
\ea
\eeq
such that
\beq\label{eq:psiParts}
\ba{rcl}\dis
\Psi_1(\Q)
&=& \dis
\frac{C}4 \left( \trace(\Q^2) - \alpha^2 \right)^2 \,, \\
\hueco \dis
\Psi_2(\Q)
&=& \dis
\left(\frac{A}2 + \frac{C\alpha^2}2 \right)\trace(\Q^2) - \frac{C}4 \alpha^4 \,, \\
\hueco \dis
\Psi_3(\Q)
&=& \dis
-\frac{B}3 \trace(\Q^3)\,.
\ea
\eeq

The first derivative is given by
\beq
\ba{rcl}
\bm{\psi}(\Q)
&=& \dis
\bm{\psi}_1 (\Q) + \bm{\psi}_2 (\Q) + \bm{\psi}_3 (\Q) \,, 
\ea
\eeq
with 
\beq\label{eq:gradPsiParts}
\ba{rcl}
\hueco
\bm{\psi}_1 (\Q) 
&=& \dis
C\left( \trace(\Q^2) - \alpha^2 \right) \Q \,, \\
\hueco
\hueco 
\bm{\psi}_2 (\Q)
&=& \dis
\left( A + C \alpha^2 \right) \Q \,, \\
\hueco
\bm{\psi}_3 (\Q)
&=& \dis
- B \Q^2 \,.
\ea
\eeq
We will use different ideas to approximate each of the terms in $\bm{\psi}(\Q)$. First, since $\bm{\psi}_2(\Q)$ is linear in $\Q$, in all the numerical schemes moving forward we will use the Crank-Nicolson approximation
\beq\label{eq:psiTwoNS}\dis
\bm{\psi}_2^{\dt} (\Q)
\,=\,
\bm{\psi}_2 \left( \Q^{n+\frac12} \right)
\,=\,
\left(A + C\alpha^2 \right) \frac{\Q^{n+1} + \Q^n}{2} \,.
\eeq
Next, by using the symmetry of $\Psi_1$ we consider a $\mathcal{C}^2$ truncation as follows:
\beq\label{eq:psiOneTrunc}
\widehat\Psi_1 (\Q)
\,=\,
\left\{ \ba{ll}\dis
\frac{C}4 \left(\trace(\Q^2) - \alpha^2\right)^2 & \mbox{ if } |\Q|\leq\alpha \,, \\
\hueco
C\alpha^2\left( |\Q| - \alpha \right)^2 & \mbox{ if } |\Q|>\alpha\,. \\
\ea \right.
\eeq
Note that $\trace(\Q^2)=|\Q|^2$ so $\widehat\Psi_1$ is continuous. Furthermore, the first derivative
\beq\label{eq:gradPsiOne}
\bm{\widehat\psi}_1 (\Q)
\,=\,
\left\{ \ba{ll}\dis
C \left(\trace(\Q^2) - \alpha^2\right)\Q & \mbox{ if } |\Q|\leq\alpha \,, \\
\hueco
2C\alpha^2\left( |\Q| - \alpha \right)\frac{\Q}{|\Q|} & \mbox{ if } |\Q|>\alpha\,, \\
\ea \right.
\eeq
and second derivative
\beq\label{eq:hessianPsiOne}
\left[\nabla \bm{\widehat\psi}_1 (\Q)\right]_{ijkl}
\,=\,
\left\{ \ba{ll}\dis
2C Q_{ij}Q_{kl} + C(\trace(\Q^2) - \alpha^2)\delta_{ik}\delta_{jl} & \mbox{ if } |\Q|\leq\alpha \,,\\
\hueco
2C\alpha^2\left( \alpha\frac{q_j q_l}{|\Q|^3} + \frac{|\Q| - \alpha}{|\Q|}\delta_{ik}\delta_{jl}\right) & \mbox{ if } |\Q|>\alpha \,,\\
\ea \right.
\eeq
are both continuous. Moreover, we have the following bounds on $\nabla \bm{\widehat\psi}_1 (\Q)$ 
\beq\label{eq:hessianPsiOneBounds}
\left|\left| \left[ \nabla \bm{\widehat\psi}_1 (\Q) \right]_{ijkl} \right|\right|_{L^\infty} 
\,\leq\,
 2C\alpha^2 (1 + \delta_{ik}\delta_{jl}) \quad \mbox{ and } \quad
\left|\left| \nabla \bm{\widehat\psi}_1 (\Q)\right|\right|_F 
\,\leq\, 
12\sqrt{3}C\alpha^2 \,,
\eeq
where the fourth order tensor Frobenius norm is $||\bm{A}||_F:=\sqrt{A_{ijkl}A_{ijkl}}$.

Finally, we would like to develop a $\mathcal{C}^2$ truncation for $\Psi_3$ but it lacks symmetry as in $\Psi_1$. Thus we follow the ideas presented in \cite{SurowiecWalker} and consider the following construction which smoothly transitions $\Psi_3$ into a function with bounded second derivative. Let $\rho:[0,\infty)\rightarrow\R$ be the following monotone function
\beq\label{eq:rho}
\rho(r)
\,:=\,
\left\{
\ba{ll}\dis
1
& \mbox{ if }
r\leq \sqrt{\alpha_1} \,,
\\ \hueco \dis
\left(2\frac{r - \alpha_1}{\alpha_2 - \alpha_1} + 1\right)\left(1 - \frac{r - \alpha_1}{\alpha_2 - \alpha_1}\right)^2
& \mbox{ if }
\sqrt{\alpha_1} < r < \sqrt{\alpha_2} \,,
\\ \hueco
0
& \mbox{ if }
\sqrt{\alpha_2} \leq r\,.
\ea
\right.
\eeq
Here $\alpha<\alpha_1<\alpha_2$ are fixed constants where $\alpha$ is the value from the maximum principle \eqref{eq:alpha}. 
Then we modify $\Psi_3$ as follows
\beq\label{eq:barPsiThree}
\widehat\Psi_3(\Q) = -\frac{B}{3}\trace\Q^3 \rho(\trace(\Q^2)) + \trace(\Q^2) \left[ 1 - \rho(\trace(\Q^2)) \right]\,.
\eeq
In this way, we do not change the function in the region $|\Q|\leq\alpha$, and when $|\Q|>\alpha$ we smoothly transition to a function of quadratic growth so that the Hessian tensor is bounded. There exist computable constants $ s_{3,\infty }$, and $ s_{3,F}$ (see appendix) such that
\beq\label{eq:hessianPsiThreeBounds}
\left|\left| \left[ \nabla \bm{\widehat\psi}_3(\Q) \right]_{ijkl} \right|\right|_{L^\infty} 
\,\leq\,
 s_{3,\infty } 
 \quad 
 \mbox{ and }
 \quad
\left|\left|\nabla \bm{\widehat\psi}_3(\Q) \right|\right|_F 
\,\leq\, 
 s_{3,F} .
\eeq
Using these ideas, we write the truncated potential function $\widehat\Psi(\Q)$ as
\beq\label{eq:psiHat}
\widehat\Psi (\Q)
\, = \,
\widehat\Psi_1(\Q) + \Psi_2(\Q) + \widehat\Psi_3(\Q)\,.
\eeq
Then the modified energy $\widehat E(\Q)$ is
\beq\label{eq:modifiedEnergy}
\widehat E(\Q) 
\, = \,
\int_\Om \left( \frac12|\nabla \Q |^2 + \frac1\eps \widehat\Psi(\Q) \right)d \mathbf{x}\,.
\eeq
We write $\bm{\psi}^{\dt}$ as follows
\beq\label{eq:scheme1Psi}
\ba{rcl}
\bm{\psi}^{\dt} (\Q^{n+1},\Q^n)
&=& \dis
\bm{\psi}_1^{\dt}(\Q^{n+1},\Q^n) + \bm{\psi}_2^{\dt}(\Q^{n+1},\Q^n) + \bm{\psi}_3^{\dt}(\Q^{n+1},\Q^n)\,,
\ea
\eeq
with
\beq\label{eq:scheme1PsiParts}
\ba{rcl}\dis
\bm{\psi}_1^{\dt} (\Q^{n+1},\Q^n)
&=& \dis
\bm{\widehat\psi}_1( \Q^n ) + \frac{1}{2} S_1\left( \Q^{n+1} - \Q^n \right)\,, \\
\hueco \dis
\bm{\psi}_2^{\dt} (\Q^{n+1},\Q^n)
&=& \dis
\bm{\psi}_2 \left( \Q^{n+\frac{1}{2}}\right)\,, \\
\hueco \dis
\bm{\psi}_3^{\dt} (\Q^{n+1},\Q^n)
&=& \dis
\bm{\widehat\psi}_3 (\Q^n) + \frac{1}{2} S_3\left( \Q^{n+1} - \Q^n \right)\,.
\ea
\eeq

Then the scheme can be written as follows:
\beq\label{eq:scheme1}
\dis
    \left( \delta_t \Q^{n+1} , \overline{\Q} \right) + \gamma\left[\left( \nabla\Q^{n+\frac12} , \nabla \overline{\Q} \right)
+ \frac{1}{\eps}\left( \bm{\psi}^{\dt}(\Q^{n+1},\Q^n), \overline{\Q} \right)
+ \frac{1}{\eps}\left( \bm{p}(\Q^n), \overline\Q  \right) \right]
\, = \,
0\,,
\eeq
for all $\overline\Q \in \bm{K}_h$. 
Here, we choose the trace part of the problem to be approximated as
\beq\label{eq:tracePartScheme1}
p(\Q^n) = - \frac13 \trace\left( \bm{\widehat\psi}_3(\Q^n) \right) \I = \frac{B}{3} \trace\left( (\Q^n)^2\right) \I\,.
\eeq

\begin{lem}\label{lem:tracelessScheme1}
Scheme UES1D satisfies a discrete version of \Cref{lem:traceless}. For any positive constants $S_1$, and $S_3$, if $\trace(\Q^n)=0$, then $\trace(\Q^{n+1})=0$.
\end{lem}
\begin{proof}
In \eqref{eq:scheme1} let $\overline\Q = \frac1\gamma\trace(\Q^{n+1})\I$. Then we obtain
$$
\ba{rcl} \dis
\frac1\gamma \frac1\dt\left(\trace(\Q^{n+1}), \trace(\Q^{n+1}) \right)
- \frac1\gamma\frac1\dt \left( \trace(\Q^{n}), \trace(\Q^{n+1}) \right) 
& & \\
\hueco \dis
+ \frac12 \left( \nabla \trace(\Q^{n+1}), \trace(\Q^{n+1}) \right) 
+ \frac12 \left( \nabla \trace(\Q^{n}), \trace(\Q^{n+1}) \right)
& & \\
\hueco \dis
+ \frac1\eps \left(\trace(\bm{\widehat\psi}_1 (\Q^n) ) + \frac12 S_1 (\trace(\Q^{n+1}) - \trace(\Q^n)), \trace(\Q^{n+1}) \right)
& & \\
\hueco \dis
+ \frac1\eps \left( \trace(\bm{\psi}_2 (\Q^n) ) , \trace(\Q^{n+1}) \right)
& & \\
\hueco \dis
+ \frac1\eps \left(\trace(\bm{\widehat\psi}_3 (\Q^n) ) + \frac12 S_3 (\trace(\Q^{n+1}) - \trace(\Q^n)), \trace(\Q^{n+1}) \right)
& & \\
\hueco \dis
+  \frac1\eps \left( \trace(\bm{p}(\Q^n)), \trace(\Q^{n+1})\right)
&=&
0\,.
\ea
$$
Referring to \eqref{eq:gradPsiParts} we see that
$$
\trace(\bm{\widehat\psi}_1 (\Q^n) ) = 0 \quad \quad \mbox{ and } \quad \quad \trace(\bm{\psi}_2 (\Q^n) ) = 0\,.
$$
Moreover, the trace penalty term has been chosen so that
$$
\trace(\nabla \widehat\psi_3 (\Q^n) ) + \trace(\bm{p}(\Q^n)) = 0 \, .
$$
Then by rearranging the terms we arrive at
$$\dis
\left( \frac1\gamma \frac1\dt + \frac{S_1 + S_3}{2\eps} \right) \| \trace(\Q^{n+1}) \|_{L^2}^2 
+ \frac12 \| \nabla \trace(\Q^{n+1}) \|_{L^2}^2 = 0\,.
$$
Hence $\trace(\Q^{n+1}) = 0$.
\end{proof}

\begin{obs}
The unknowns in $\Q$ are decoupled in Scheme UES1D.
\end{obs}

\begin{obs}
Since Scheme UES1D preserves the traceless property we only need to solve for five unknowns since $Q^{n+1}_{33} = -(Q^{n+1}_{11} + Q^{n+1}_{22})$.
\end{obs}

\begin{obs}
Since $\trace(\Q^{n+1})=0$, when $\trace(\Q^n)=0$, then referring to \eqref{eq:numericalDissipation} the numerical dissipation introduced by Scheme UES1D will be 
$$
\mathbf{ND}\left(\Q^{n+1}, \Q^n\right) 
        \,=\, \dis
        \frac1{\eps\dt}\int_\Omega \left( \bm{\psi}^{\Delta t} (\Q^{n+1},\Q^n) : (\Q^{n+1} - \Q^n) \right) - \left( \widehat\Psi(\Q^{n+1}) - \widehat\Psi(\Q^n) \right) d\mathbf{x}\,.
$$
\end{obs}

Following the definition of energy stability we can show that this choice of $\bm{\psi}^{\dt}$ and $\bm{p}(\Q^n)$ will guarantee energy stability of the scheme.
\begin{lem}\label{lem:scheme1NDPositive}
If 
$$
S_1 \geq \| \nabla \bm{\widehat\psi}_1 (\Q) \|_F \quad \quad \mbox{ and } \quad \quad S_3 \geq \| \nabla \bm{\widehat\psi}_3 (\Q) \|_F\,,
$$
then $\mathbf{ND} (\Q^{n+1},\Q^n) \geq 0$ for all $n\geq 0$. Therefore, Scheme UES1D is linear and unconditionally energy stable with respect to the modified energy $\widehat E(\Q)$.
\end{lem}
\begin{proof}
    We will make use of the following inequality: for a fourth-order tensor $A$ and a symmetric second order tensor $P$
    \beq\label{eq:tensorInequality}
        \mathbf{P}\!:\!\bm{A}\!:\!\mathbf{P} \leq ||\bm{A}||_F ||\mathbf{P}||_F^2\,.
    \eeq
    For $i=1,3$, we can write the Taylor expansion of $\widehat\Psi_i$ about $\Q^n$
    \beq\label{eq:taylorPsiOne}
        \widehat\Psi_i (\Q^{n+1}) = \widehat\Psi_i (\Q^n) + \bm{\widehat\psi}_i (\Q^n) : (\Q^{n+1} - \Q^n) 
        + \frac{1}{2} (\Q^{n+1} - \Q^n) : \nabla \bm{\widehat\psi}_i(\Q^n) : (\Q^{n+1} - \Q^n)\,,
    \eeq
    and it can be reorganized in the following way
    $$
        \bm{\widehat\psi}_i (\Q^n) : (\Q^{n+1} - \Q^n)
        + \frac{1}{2} (\Q^{n+1} - \Q^n) : \bm{\widehat\psi}_i(\Q^n) : (\Q^{n+1} - \Q^n) =
        \widehat\Psi_i(\Q^{n+1}) - \widehat\Psi_i (\Q^n)\,. \\
    $$
    Using \eqref{eq:tensorInequality} we obtain
    $$
        \left(\bm{\widehat\psi}_i (\Q^n) + S_i(\Q^{n+1} - \Q^n)\right):(\Q^{n+1} - \Q^n)
        \geq \widehat\Psi_i(\Q^{n+1}) - \widehat\Psi_i (\Q^n) \\
    $$
    and finally integrating over $\Om$ and referring to equation \eqref{eq:numericalDissipation} we deduce 
    $$
        \mathbf{ND}(\Q^{n+1},\Q^n) \geq 0\,,
    $$
    for all $n\geq0$.
\end{proof}

\begin{obs}
The numerical dissipation introduced in this scheme will depend on $\alpha$ from the maximum principle, and the choices of $\alpha_1$, and $\alpha_2$ from the truncation of $\Psi_3(\Q)$. We note that the parameter $ s_{3,F}$ grows immensely when $\alpha_2$ is large, and when $(\alpha_2 - \alpha_1)$ is small (see appendix). Therefore the dissipation will be relatively large for any choices of $\alpha_1$ and $\alpha_2$.
\end{obs}

\begin{lem}\label{lem:solvableScheme1}
Scheme UES1D is unconditionally uniquely solvable.
\end{lem}
\begin{proof}
Since \eqref{eq:scheme1} is a squared algebraic linear system we just need to prove uniqueness which will imply existence. Suppose $\Q_1$ and $\Q_2$ are solutions to \eqref{eq:scheme1}. Let $\widehat{\Q}=\Q_1 - \Q_2$. Then $\widehat\Q$ satisfies
$$
\ba{rcl} \dis
\frac1\gamma\frac1\dt \left( \widehat\Q , \overline{\Q} \right) + \frac12 \left( \nabla \widehat \Q , \nabla \overline{\Q} \right) 
+ \frac{S_1 + A + C\alpha^2 + S_3}{2\eps}\left( \widehat\Q, \overline{\Q} \right) 
&=&
0 \quad \quad \forall\, \overline\Q \in \mathbf{K}_h\,.
\ea
$$
Hence, setting $\overline\Q = \widehat\Q$ we get
$$ \dis
\left(\frac1\gamma\frac1\dt +  \frac{S_1 + A + C\alpha^2 + S_3}{2\eps} \right) \| \widehat \Q \|^2_{L^2} 
+ \frac12 \| \widehat \Q \|^2_{H^1}
\, = \,
0\,.
$$
Because
$$
\frac1\gamma\frac1\dt +  \frac{S_1 + A + C\alpha^2 + S_3}{2\eps} \geq 0\, ,
$$
we deduce that $\widehat\Q=\bm{0}$ for any $\dt>0$.
\end{proof}

\subsubsection{Second Order Optimal Dissipation Coupled Scheme (OD2C)}
In this section we adapt a second order optimal dissipation algorithm (OD2) from \cite{Tierra14} to this problem. Given a second order tensor valued function $\bm{f}(\Q)$, the OD2 approximation is defined as
\beq\label{eq:OD2}
\bm{f}^{\dt}(\Q^{n+1},\Q^n) = \bm{f}(\Q^n) + \frac12(\Q^{n+1} - \Q^n):\nabla\bm{f}(\Q^n)\,,
\eeq
such that the numerical dissipation associated to this approximation is
$$
\int_\Om (\Q^{n+1} - \Q^n): \nabla\bm{f}(\Q^n):(\Q^{n+1} - \Q^n) \sim \mathcal{O}(\left(\dt \right)^2)\,.
$$
The idea for OD2C is to apply the OD2 approximation to each term of \eqref{eq:gradPsiParts}, and the trace penalty term $\bm{p}(\Q)$. This will create a linear, second order accurate numerical scheme but the unknowns will be coupled and therefore more computationally expensive than UES1D. Since we will not be truncating $\Psi_1(\Q)$ or $\Psi_3(\Q)$ we will not have a bound on the Hessian tensor, so energy stability is not guaranteed. However, the advantage of this scheme compared to the previous one will be: (1) the method is second order accurate in time and (2) the numerical dissipation is higher order, which allows us to more accurately capture the dynamics of the PDE. \\
Our linear approximation to $\bm{\psi}$ is as follows
\beq\label{eq:od2c}
\ba{rcl}
\bm{\psi}^{\dt} (\Q^{n+1},\Q^n )
&=& \dis
\bm{\psi}_1^{\dt} (\Q^{n+1},\Q^n ) + \bm{\psi}_2^{\dt} (\Q^{n+1},\Q^n ) + \bm{\psi}_3^{\dt} (\Q^{n+1},\Q^n ) \,, \\
\hueco
\bm{\psi}_1^{\dt} (\Q^{n+1},\Q^n )
&=& \dis
\bm{\psi}_1 (\Q^{n+1},\Q^n) + \frac12 (\Q^{n+1} - \Q^n) : \nabla\bm{\psi}_1 (\Q^n) \,, \\
\hueco
\bm{\psi}^{\dt}_2 (\Q^{n+1},\Q^n)
&=& \dis
\bm{\psi}_2(\Q^{n+\frac12} ) \,, \\
\hueco
\bm{\psi}_3^{\dt} (\Q^{n+1},\Q^n )
&=& \dis
\bm{\psi}_3 (\Q^{n+1},\Q^n) + \frac12 (\Q^{n+1} - \Q^n) : \nabla\bm{\psi}_3 (\Q^n)  \,, \\
\bm{p}^{\dt}(\Q^{n+1},\Q^n)
&=& \dis
\bm{p}(\Q^n) + \frac12(\Q^{n+1} - \Q^n): \nabla \bm{p} (\Q^n)  \\
\ea
\eeq
and the scheme can be written as
\beq\label{eq:scheme2}
\dis
    \left( \delta_t \Q^{n+1} , \overline{\Q} \right) + \gamma \left[ \left( \nabla\Q^{n+\frac12} , \nabla \overline{\Q} \right) 
+ \frac{1}{\eps}\left( \bm{\psi}^{\dt}(\Q^{n}), \overline{\Q} \right) 
+ \frac1\eps \left( \bm{p}^{\dt}(\Q^{n+1},\Q^n) , \overline\Q  \right)  \right]
\, = \,
0 \,,
\eeq
for all $\overline\Q \in \bm{K}_h$.
Here, 
\beq\label{eq:tracePartScheme2}
\bm{p}(\Q^n) = \frac{B}3 \trace\left( (\Q^n)^2 \right)\I
\eeq
and
\beq\label{eq:gradTracePartScheme2}
\nabla \bm{p}(\Q^n) = \frac{2B}{3} \Q^n \otimes \I \, .
\eeq

\begin{lem}\label{lem:tracelessScheme2}
Scheme OD2C satisfies a discrete version of \Cref{lem:tracelessScheme1}. If $\trace(\Q^n)=0$, then $\trace(\Q^{n+1})=0$.
\end{lem}
\begin{proof}
Using the symmetry of $\Q$ we have the following fact:
$$
\trace\left(  \left( \Q^{n+1} - \Q^n \right) : \nabla\bm{\psi}_3 (\Q^n) \right) = 2  \left( \Q^{n+1} - \Q^n \right) : \Q^n = \trace\left(  \left( \Q^{n+1} - \Q^n \right) : \nabla \bm{p}(\Q^n) \right).
$$
Hence, 
$$
 \trace\left( \bm{\psi}(\Q^{n}) + \frac{1}{2} (\Q^{n+1} - \Q^n):\nabla\bm{\psi}(\Q^n) \right) + \trace\left(\bm{p}(\Q^n) + \frac12(\Q^{n+1} - \Q^n):\nabla \bm{p}(\Q^n) \right) = 0\,.
$$
Then taking $\overline\Q = \trace(\Q^{n+1})\I$ and following the same ideas in \Cref{lem:traceless} we obtain $\trace(\Q^{n+1}) = 0$.
\end{proof}
\begin{obs}
Since Scheme OD2C preserves the traceless property we only need to solve for five unknowns since $Q^{n+1}_{33} = -(Q^{n+1}_{11} + Q^{n+1}_{22})$.
\end{obs}
\begin{lem}\label{lem:scheme2EnergyLaw}
Scheme OD2C satisfies a discrete version of \eqref{eq:continuousEnergyLaw}, with 
$$
\mathbf{ND}\left(\Q^{n+1}, \Q^n\right) 
\,=\, \dis
\frac1{\eps\dt}\int_\Omega \left( \bm{\psi}^{\Delta t} (\Q^{n+1},\Q^n) : (\Q^{n+1} - \Q^n) \right) - \left( \Psi(\Q^{n+1}) - \Psi(\Q^n) \right) d\mathbf{x} .
$$
\end{lem}

\begin{obs}\label{rem:scheme2ND}
The numerical dissipation introduced is second order in time. 
Indeed, using Taylor
$$
\ba{rcl}
\Psi(\Q^{n+1})
&=&\dis
\Psi(\Q^{n})
+ \bm{\psi} (\Q^n)(\Q^{n+1} - \Q^n)
+ \frac12 \nabla\bm{\psi}(\Q^n)(\Q^{n+1} - \Q^n)^2
+ \frac16 \bm{H}_{\bm{\psi}} (\Q^\chi)(\Q^{n+1} - \Q^n)^3\,,
\\ \hueco
\Psi(\Q^{n+1})
&=&\dis
\Psi(\Q^{n})
+ \bm{\psi}(\Q^n)(\Q^{n+1} - \Q^n)
+ \frac12\nabla\bm{\psi}(\Q^\xi)(\Q^{n+1} - \Q^n)^2\,,
\ea
$$
subtracting and multiplying by $\dfrac1{\Delta t}$we obtain
$$
\ba{rcl}\dis
\frac1{2\Delta t}\big(\nabla\bm{\psi}(\Q^n) - \nabla\bm{\psi}(\Q^\xi)\big)(\Q^{n+1} - \Q^n)^2
&=&\dis
-\frac1{6\Delta t} \nabla \bm{H}_{\bm{\psi}} (\Q^\chi)(\Q^{n+1} - \Q^n)^3 \\
\hueco 
&=&\dis
-\frac{(\Delta t)^2}6 \nabla \bm{H}_{\bm{\psi}}(\Q^\chi)(\delta_t(\Q^{n+1}))^3 
\,\sim\,
\mathcal{O}((\Delta t)^2)\,.
\ea
$$
Then, the total numerical dissipation introduced by the system is
$$
\ba{rcl}
\mathbf{ND}\left(\Q^{n+1}, \Q^n\right) 
&=& \dis
\frac1{\eps\dt}\int_\Omega \left( \bm{\psi}^{\Delta t} (\Q^{n+1},\Q^n) : (\Q^{n+1} - \Q^n) \right) - \left( \Psi(\Q^{n+1}) - \Psi(\Q^n) \right) d\mathbf{x} \\
\hueco
&=& \dis
\int_\Omega \frac1{2\Delta t}\big(\nabla \bm{\psi}(\Q^n) - \nabla\bm{\psi} (\Q^\xi)\big)(\Q^{n+1} - \Q^n)^2
\,\sim\,
\mathcal{O}\left( (\dt)^2 \right)\,.
\ea
$$
\end{obs}

\begin{obs}
The numerical dissipation has no sign so we cannot guarantee energy stability.
\end{obs}

\begin{lem}\label{lem:solvableScheme2}
Scheme OD2C is uniquely solvable if
\beq\label{eq:solvableScheme2}
\dt \leq \frac{2\eps}{\gamma \left(\|\nabla\bm{\psi}_1(\Q^n) \|_F + \| \nabla \bm{\psi}_3(\Q^n) \|_F + \| \nabla \bm{p}(\Q^n) \|_F\right)}\,.
\eeq
\end{lem}
\begin{proof}
It is enough to prove uniqueness since this implies existence. Let $\Q_1$ and $\Q_2$ be solutions to \eqref{eq:scheme2}, and let $\widehat\Q=\Q_1 - \Q_2$. Then $\widehat \Q$ satisfies
$$
\ba{rcl} \dis
\left( \frac1\gamma\frac1\dt + \frac{A+C\alpha^2}{2\eps} \right) \left(\widehat\Q, \overline\Q \right)
+ \frac12 \left( \nabla \widehat\Q, \nabla\overline\Q \right)
& & \\
\hueco \dis
+ \frac1{2\eps}\left(\left(\nabla\bm{\psi}_1(\Q^n) + \nabla\bm{\psi}_3(\Q^n) + \nabla \bm{p}(\Q^n) \right):\widehat\Q, \overline\Q \right) 
&=& 
0 \quad \quad \forall\, \overline\Q \in \bm{K}_h\,.
\ea
$$
Then by taking $\overline\Q =   \widehat\Q$ we obtain
$$
\left( \frac1\gamma \frac1\dt + \frac{A+C\alpha^2}{2\eps} \right) \| \widehat\Q \|^2_{L^2}
+ \frac12 \| \nabla \widehat\Q \|^2_{L^2}
\, \leq \,
\left( \|\nabla\bm{\psi}_1(\Q^n) \|_F + \| \nabla\bm{\psi}_3 (\Q^n) \|_F + \| \nabla \bm{p}(\Q^n) \|_F \right) \| \widehat\Q \|^2_{L^2}\,.
$$
Hence, $\widehat\Q=0$ when
$$
\frac1\gamma \frac1\dt + \frac{A + C\alpha}{2\eps} - \frac{\| \nabla\bm{\psi}_1(\Q^n) \|_F + \|\nabla\bm{\psi}_3(\Q^n) \|_F + \| \nabla \bm{p}(\Q^n) \|_F }{2\eps} \geq 0\,,
$$
which is true if \eqref{eq:solvableScheme2} holds.
\end{proof}

\subsubsection{First Order Optimal Dissipation Decoupled Scheme (OD1D)}
In this case, we introduce an idea to modify the OD2 approximation to be able to decouple the computation of the unknowns. It is the second order term in the OD2 approximation which will couple the unknowns in $\Q$. To remedy this, we use the symmetric property of $\Q$ to create a lower triangular approximation $\bm{A}^{LT}$ to a fourth order tensor $\bm{A}$ defined by
$$
\bm{A}^{LT}
\, = \,
\begin{bmatrix}
\bm{A}^{LT}_{11} & \bm{A}^{LT}_{12} & \bm{A}^{LT}_{13} \\
\bm{A}^{LT}_{21} & \bm{A}^{LT}_{22} & \bm{A}^{LT}_{23} \\
\bm{A}^{LT}_{31} & \bm{A}^{LT}_{32} & \bm{A}^{LT}_{33} 
\end{bmatrix}\,,
$$
where
$$
\bm{A}^{LT}_{11}
\,=\,
\begin{bmatrix}
A_{1111}   &   0   &   0   \\
0   &   0   &   0   \\
0   &   0   &   0   \\
\end{bmatrix} \,,
\quad \quad 
\bm{A}^{LT}_{12}
\, = \,
\begin{bmatrix}
A_{1211}+A_{1112}   & 0  & 0   \\
A_{1212}+A_{1221}   &   0   &   0   \\
0   &   0   &   0   \\
\end{bmatrix} \,,
\quad \quad
\bm{A}^{LT}_{21}
\, = \,
\begin{bmatrix}
A_{2111} + A_{1121}   & 0  & 0   \\
A_{2121}+A_{2112}   &   0   &   0   \\
0   &   0   &   0   \\
\end{bmatrix} \,,
$$
$$
\bm{A}^{LT}_{13}
\, = \,
\begin{bmatrix}
A_{1311} + A_{1113}   &   0   &   0   \\
A_{1312}+A_{1321} + A_{1213} + A_{1231}  &   0   &   0   \\
A_{1313} + A_{1331}   &   0   &   0   \\
\end{bmatrix} \,,
\quad \quad 
\bm{A}^{LT}_{31}
\, = \,
\begin{bmatrix}
A_{3111}+A_{1131}   &   0   &   0   \\
A_{3112}+A_{3121}+A_{2113}+A_{2131}  &   0   &   0   \\
A_{3113}+A_{3131}  &   0   &   0   \\
\end{bmatrix} \,,
$$
$$
\bm{A}^{LT}_{22}
\, = \,
\begin{bmatrix}
A_{2211}+A_{1122}   &   0   &   0   \\
A_{2212}+A_{2221}+A_{1222}+A_{2122}   &   A_{2222}  &   0   \\
A_{2213}+A_{2231}+A_{1322}+A_{3122}   &   0   &   0   \\
\end{bmatrix} \,, 
$$
$$
\bm{A}^{LT}_{23}
\, = \,
\begin{bmatrix}
A_{2311}+A_{1123}   &   0   &   0   \\
A_{2312}+A_{2321}+A_{1223}+A_{1232}   &   A_{2322}+A_{2223}  &   0  \\
A_{2313}+A_{2331}+A_{1323}+A_{1332}  &    A_{2323}+A_{2332}   &   0   \\
\end{bmatrix}\,,
$$
$$
\bm{A}^{LT}_{32}
\, = \,
\begin{bmatrix}
A_{3211}+A_{1132}   &   0   &   0   \\
A_{3212}+A_{3221}+A_{2123}+A_{2132}   &   A_{3222}+A_{2232}   &   0  \\
A_{3213}+A_{3231}+A_{3123}+A_{3132}  &    A_{3223}+A_{3232}   &   0  \\
\end{bmatrix} \,,
$$
$$
\bm{A}^{LT}_{33}
\, = \,
\begin{bmatrix}
A_{3311}+A_{1133}   &   0   &   0   \\
A_{3312}+A_{3321}+A_{1233}+A_{2133}   &   A_{3322}+A_{2233}  &   0   \\
A_{3313}+A_{3331}+A_{1333}+A_{3133}   &   A_{3323}+A_{3332}+A_{2333}+A_{3233}   &   A_{3333}   \\
\end{bmatrix}\,,
$$
which satisfies
\beq\label{eq:identityLT}
\Q:\bm{A}:\Q = \Q:\bm{A}^{LT}:\Q \mbox{ for all } \Q\in\Lambda
\eeq
and therefore the modified OD2 approximation that we propose for a second order tensor valued function $\bm{f}(\Q)$ is
\beq\label{eq:OD2LT}
\bm{f}^{LT}(\Q^{n+1},\Q^n) = \bm{f}(\Q^n) + \frac12(\Q^{n+1} - \Q^n):\nabla\bm{f}^{LT}(\Q^n) \, ,
\eeq
which also has numerical dissipation of $\mathcal{O}(\left(\dt \right)^2)$ (see remark~\ref{rem:OD1D_orderND}). \\
We now apply the idea of a lower triangular approximation to the OD2C approximation \eqref{eq:od2c} of $\Psi_1(\Q)$, $\Psi_3(\Q)$, and $\bm{p}(\Q)$. In this way we will decouple the unknowns as in scheme UES1D, but only introduce $\mathcal{O}\left((\dt)^2\right)$ numerical dissipation. However, because the second derivative of $\Psi(\Q)$ is not bounded this scheme is not unconditionally energy stable. \\
The scheme can be written as follows.
\beq\label{eq:od1d}
\ba{rcl}
\bm{\psi}^{\dt} (\Q^{n+1},\Q^n )
&=& \dis
\bm{\psi}_1^{\dt} (\Q^{n+1},\Q^n ) + \bm{\psi}_2^{\dt} (\Q^{n+1},\Q^n ) + \bm{\psi}_3^{\dt} (\Q^{n+1},\Q^n ) \,, \\
\hueco
\bm{\psi}_1^{\dt} (\Q^{n+1},\Q^n )
&=& \dis
\bm{\psi}_1 (\Q^{n+1},\Q^n) + \frac12 (\Q^{n+1} - \Q^n) : \nabla\bm{\psi}^{LT}_1 (\Q^n) \,, \\
\hueco
\bm{\psi}^{\dt}_2 (\Q^{n+1},\Q^n)
&=& \dis
\bm{\psi}_2(\Q^{n+\frac12} ) \,, \\
\hueco
\bm{\psi}_3^{\dt} (\Q^{n+1},\Q^n )
&=& \dis
\bm{\psi}_3 (\Q^{n+1},\Q^n) + \frac12 (\Q^{n+1} - \Q^n) : \nabla\bm{\psi}^{LT}_3 (\Q^n)  \,, \\
\bm{p}^{\dt}(\Q^{n+1},\Q^n)
&=& \dis
\bm{p}(\Q^n) + \frac12(\Q^{n+1} - \Q^n): \nabla \bm{p}^{LT} (\Q^n) \,, \\
\ea
\eeq
and the scheme can be written as
\beq\label{eq:scheme3}
\dis
    \left( \delta_t \Q^{n+1} , \overline{\Q} \right) + \gamma \left[ \left( \nabla\Q^{n+\frac12} , \nabla \overline{\Q} \right) 
+ \frac{1}{\eps}\left( \bm{\psi}^{\dt}(\Q^{n}), \overline{\Q} \right) 
+ \frac1\eps \left( \bm{p}^{\dt}(\Q^{n+1},\Q^n) , \overline\Q  \right)  \right]
\, = \,
0 \,,
\eeq
for all $\overline\Q \in \bm{K}_h$.
Here, 
\beq\label{eq:tracePartScheme3}
\bm{p}(\Q^n) = \frac{B}3 \trace\left( (\Q^n)^2 \right)\I 
\eeq
and
\beq\label{eq:gradTracePartScheme3}
\nabla \bm{p}(\Q^n) = \frac{2B}{3} \Q^n \otimes \I \, .
\eeq
for all $\overline\Q \in \bm{K}_h$.
\begin{obs}
With this idea, the computation of each term $Q_{ij}$ of $\Q$ only depends on the previous terms $Q_{kl}$ for $k\leq i$, and $l\leq j$. Therefore Scheme OD1D is decoupled although the computation of the unknowns needs to be done sequentially.
\end{obs}
\begin{lem}\label{lem:tracelessScheme3}
Scheme OD1D satisfies a discrete version of \Cref{lem:traceless}. If $\trace(\Q^n)=0$, then $\trace(\Q^{n+1})=0$.
\end{lem}
\begin{proof}
The same logic from the proof of \Cref{lem:tracelessScheme2} will work here using the fact that
$$
\trace\left(  \left( \Q^{n+1} - \Q^n \right) : \nabla\bm{\psi}^{LT}_3(\Q^n) \right) = 2  \left( \Q^{n+1} - \Q^n \right) : \Q^n = \trace\left(  \left( \Q^{n+1} - \Q^n \right) : \nabla \bm{p}^{LT}(\Q^n) \right).
$$
\end{proof}
\begin{obs}
Since Scheme OD1D preserves the traceless property we only need to solve for five unknowns since $Q^{n+1}_{33} = -(Q^{n+1}_{11} + Q^{n+1}_{22})$.
\end{obs}

\begin{lem}\label{lem:scheme3EnergyLaw}
Scheme OD1D satisfies a discrete version of \eqref{eq:continuousEnergyLaw}, with 
$$
\mathbf{ND}\left(\Q^{n+1}, \Q^n\right) 
\,=\, \dis
\frac1{\eps\dt}\int_\Omega \left( \bm{\psi}^{\Delta t} (\Q^{n+1},\Q^n) : (\Q^{n+1} - \Q^n) \right) - \left( \Psi(\Q^{n+1}) - \Psi(\Q^n) \right) d\mathbf{x} \,.
$$
\end{lem}
\begin{obs}\label{rem:OD1D_orderND}
The numerical dissipation introduced by scheme OD1D is second order in time. 
Using Taylor and \eqref{eq:identityLT} gives the same numerical dissipation as computed in \Cref{rem:scheme2ND}.
\end{obs}

\begin{lem}\label{lem:solvableScheme3}
Scheme OD1D is uniquely solvable if \eqref{eq:solvableScheme2} holds.
\end{lem}
\begin{proof}
Same as in \Cref{lem:solvableScheme2}.
\end{proof}

\section{Numerical Results}\label{sec:numericalResults}
In this section we present the results of numerical simulations to showcase the accuracy and efficiency of the three schemes presented in this paper. \\
The following simulations have been done in 2D and 3D domains using the \textit{FreeFEM++} software \cite{freeFEM}, and the data was post-processed using MATLAB \cite{MATLAB}. Visualization of the data was done using Paraview \cite{Paraview}. All simulations were completed on a desktop computer using a ten core cpu with base clock 3.7 GHz, and 64GB of RAM. \\
In all of the simulations, the discrete space considered is $\bm{K}_h=\mathbb{P}^{3\times 3}_1$. The first numerical experiment will compare the convergence rates of the three schemes UES1D, OD2C, and OD1D. Next, we will show the lower numerical dissipation of schemes OD2C and OD1D as compared to UES1D, and moreover we show a comparison of the computational cost of the schemes. In the third experiment we will show dynamics of defects in 2D with various boundary conditions. Finally, we perform a computational stress test of scheme OD1D by simulating a nematic in a 3D domain with a relatively fine mesh.\\
Unless otherwise stated, the parameters chosen for the simulations are
\beq\label{eq:parameters}
A = -0.2, \quad \quad B = 1.0, \quad \quad C = 1.0 \quad \quad \eps = 10^{-2}, \quad \quad \gamma = 1.0.
\eeq
Therefore, the value $\alpha$ from the maximum principle \eqref{eq:alpha} is
\beq\label{eq:alphaSimulations}
\alpha = \sqrt{\frac{B^2}{C^2} - \frac{2A}{C} } = \sqrt{1.4}\,,
\eeq
and so we choose $\alpha_1 = 1.19$, and $\alpha_2 = 1.2$ in the truncation \eqref{eq:rho} for scheme UES1D.
With these values, we see from \Cref{lem:scheme1NDPositive} we can take
$$
\ba{rcl}\dis
\| \nabla\bm{\widehat\psi}_1 \|_F &\leq& 16.8\sqrt{3} = S_1\, , \\
\hueco \dis
\| \nabla\bm{\widehat\psi}_3 \|_F &\leq& 32571 = S_3 \, .
\ea
$$
\begin{obs}\label{rem:stabilizationValues}
In practice, the parameters $S_1$ and $S_3$ can be taken much lower and still provide numerically the decreasing energy property although the energy stability property will not be guaranteed. For all of our simulations, we use the bound on $\| \nabla\bm{\psi}_1(\Q)\|_F$, and $\| \nabla\bm{\psi}_3(\Q)\|_F$ obtained by assuming $|\Q|\leq \alpha$. This will result in the same value for $S_1$ but the bound on $\nabla \bm{\psi}_3(\Q)$ can be reduced to
\beq\label{eq:valueS3}
S_3 = 208 \geq \| \nabla \bm{\psi}_3 (\Q) \|_F \quad \quad \mbox{ for }\quad  |\Q| \leq \alpha.
\eeq
\end{obs}
\subsection{Experimental Order of Convergence}
In this first example we study the numerical error in time. The domain considered for this experiment is the square $\Om = [0,2]^2$, and final time $T=$1e-4. We impose Neumann boundary conditions $\partial_\mathbf{n} \Q = \bm{0}$. The initial configuration $\Q_0$ is given component wise by
\beq\label{eq:simulation1}
\left[Q_0\right]_{kl} 
\, = \,
 \frac12 \sin(k \pi x) \cos\left(\pi(l y - \frac12)\right)  \quad \quad \mbox{ and } \quad \quad
\left[Q_0\right]_{33} 
\, = \, - \left( \left[Q_0\right]_{11} + \left[Q_0\right]_{22} \right) .
\eeq
We compute an experimental order of convergence (EOC) using a reference solution $\Q^{\mbox{exact}}$ obtained by solving the system using a $100\times 100$ triangular mesh, and $\dt = 10^{-7}$. We then solve the system using a sequence of time steps and same spatial mesh. After that we compute the absolute errors of the associated solutions at the final time $T$ using the $L^2$ and $H^1$ norms as follows
$$\dis
e_2 (Q_{ij}) := \| Q^{\mbox{exact}}_{ij}(T) - Q_{ij}(T) \|_{L^2 }\,, \quad \quad \mbox{ and } \quad \quad 
e_1 (Q_{ij}) := \| Q^{\mbox{exact}}_{ij}(T) - Q_{ij}(T) \|_{H^1 } \,.
$$
The experimental order of convergence is computed using adjacent time steps $\dt$, and $\widetilde{\dt}$ by
$$\dis
r_k(\Q_{ij}) := \left. \log\left( \frac{e_k (\Q_{ij})}{\widetilde{e_k}(\Q_{ij})} \right) \right/ \left. \log\left( \frac{\dt}{\widetilde{\dt}} \right) \right..
$$
The final configuration of $\Q^{\mbox{exact}}$can be seen in \Cref{fig:convRate} computed with scheme OD1D. It should be noted that the lines in \Cref{fig:convRate} are normalized to the same length, however, the eigenvectors are in $\R^3$ and thus appear different lengths when they are not parallel to the domain. \\
\Crefrange{tab:ues1dConvergenceL2}{tab:od1dConvergenceH1} show the errors and convergence rates for a sequence of time steps $\dt = \left. 10^{-5} \right/ \left.\kappa \right.$, with $\kappa=1,2,3,4,5$, for each of the three schemes. We see that for scheme UES1D the convergence rate of all of the unknowns in $\Q$ is first order in both norms. Scheme OD2C has second order convergence in all of the unknowns for both norms. Finally, scheme OD1D displays at least first order convergence for all of the unknowns with respect to both norms. Moreover, we note that the errors associated to scheme OD1D are higher than for OD2C although both are lower than the errors associated with energy stable scheme UES1D.
%UES1D L2
\begin{table}
\begin{center}
\resizebox{\textwidth}{!}{%
\begin{tabular}{l l l l l l l l l l l}
\hline
$\dt$	 	&	$e_2(Q_{11})$			&	$r_2(Q_{11})$			&	$e_2(Q_{12})$ 	&	$r_2(Q_{12})$	&	$e_2(Q_{13})$		&	$r_2(Q_{13})$ &	$e_2(Q_{22})$		&	$r_2(Q_{22})$	&	$e_2(Q_{23})$ 	&	$r_2(Q_{23})$\\
\hline
1.00e-3	&	1.3034e-3	&	-		&	1.4001e-3	&	-		&	1.9434e-3	&	-	&	2.0045e-3	&	-		&	2.4432e-3	&	-		\\
5.00e-4	&	7.7166e-4	&	0.7562	&	8.3027e-4	&	0.7539	&	1.1525e-4	&	0.7537&	1.1909e-3	&	0.7512	&	1.4499e-3	&	0.7528	\\
3.33e-4	&	5.4489e-4	&	0.8582	&	5.8674e-4	&	0.8562	&	8.1455e-4	&	0.8560&	8.4220e-4	&	0.8543	&	1.0250e-3	&	0.8554	\\
2.50e-4	&	4.1916e-4	&	0.9119	&	4.5157e-4	&	0.9102	&	6.2691e-4	&	0.9101&	6.4844e-4	&	0.9088	&	7.8897e-4	&	0.9097	\\
2.00e-4	&	3.3924e-4	&	0.9840	&	3.6559e-4	&	0.9466	&	5.0755e-4	&	0.9465&	5.2510e-4	&	0.9455	&	6.3880e-4	&	0.9462	\\
\end{tabular}}
\caption{\label{tab:ues1dConvergenceL2} EOC for Scheme UES1D with respect to the discrete $L^2$ norm.}
\end{center}
\end{table}

%UES1D H1
\begin{table}
\begin{center}
\resizebox{\textwidth}{!}{%
\begin{tabular}{l l l l l l l l l l l}
\hline
$\dt$	 	&	$e_1(Q_{11})$			&	$r_1(Q_{11})$			&	$e_1(Q_{12})$ 	&	$r_1(Q_{12})$	&	$e_1(Q_{13})$		&	$r_1(Q_{13})$ &	$e_1(Q_{22})$		&	$r_1(Q_{22})$	&	$e_1(Q_{23})$ 	&	$r_1(Q_{23})$\\
\hline
1.00e-3	&	4.5355e-2	&	-		&	4.8490e-2	&	-		&	6.8667e-2	&	-	&	6.4894e-2	&	-		&	8.2724e-2	&	-		\\
5.00e-4	&	2.6190e-2	&	0.7922	&	2.7391e-2	&	0.8240	&	3.8799e-2	&	0.8236&	3.6829e-2	&	0.8172	&	4.6932e-2	&	0.8177	\\
3.33e-4	&	1.8330e-2	&	0.8801	&	1.8997e-2	&	0.9025	&	2.6908e-2	&	0.9205&	2.5591e-2	&	0.8978	&	3.2604e-2	&	0.8984	\\
2.50e-4	&	1.4035e-2	&	0.9281	&	1.4475e-2	&	0.9450	&	2.0502e-2	&	0.9451&	1.9520e-2	&	0.9414	&	2.4865e-2	&	0.9419	\\
2.00e-4	&	1.1325e-2	&	0.9614	&	1.1646e-2	&	0.9747	&	1.6495e-2	&	0.9747&	1.5715e-2	&	0.9718	&	2.0016e-2	&	0.9721	\\
\end{tabular}}
\caption{\label{tab:ues1dConvergenceH1} EOC for Scheme UES1D with respect to the discrete $H^1$ norm.}
\end{center}
\end{table}

%OD2CL2
\begin{table}
\begin{center}
\resizebox{\textwidth}{!}{%
\begin{tabular}{l l l l l l l l l l l}
\hline
$\dt$	 	&	$e_2(Q_{11})$			&	$r_2(Q_{11})$			&	$e_2(Q_{12})$ 	&	$r_2(Q_{12})$	&	$e_2(Q_{13})$		&	$r_2(Q_{13})$ &	$e_2(Q_{22})$		&	$r_2(Q_{22})$	&	$e_2(Q_{23})$ 	&	$r_2(Q_{23})$ \\
\hline
1.00e-3	&	7.9174e-7	&	-		&	1.2530e-6	&	-		&	1.7774e-6	&	-	&	1.5891e-6	&	-		&	2.0323e-6	&	-		\\
5.00e-4	&	1.9778e-7	&	2.0011	&	3.1300e-7	&	2.0011	&	4.4400e-7	&	2.0011&	3.9696e-7	&	2.0011	&	5.0768e-7	&	2.0011	\\
3.33e-4	&	8.7859e-8	&	2.0015	&	1.3903e-7	&	2.0015	&	1.9722e-7	&	2.0015&	1.7632e-7	&	2.0015	&	2.2550e-7	&	2.0015	\\
2.50e-4	&	4.9379e-8	&	2.0025	&	7.8147e-8	&	2.0025	&	1.1085e-7	&	2.0025&	9.9109e-8	&	2.0025	&	1.2675e-7	&	2.0025	\\
2.00e-4	&	3.1574e-8	&	2.0041	&	4.9968e-8	&	2.0041	&	7.0881e-8	&	2.0041&	6.3372e-8	&	2.0041	&	8.1046e-8	&	2.0041	
\end{tabular}}
\caption{\label{tab:od2cConvergenceL2} EOC for Scheme OD2C with respect to the discrete $L^2$ norm.}
\end{center}
\end{table}

%OD2C H1
\begin{table}
\begin{center}
\resizebox{\textwidth}{!}{%
\begin{tabular}{l l l l l l l l l l l}
\hline
$\dt$	 	&	$e_1(Q_{11})$			&	$r_1(Q_{11})$			&	$e_1(Q_{12})$ 	&	$r_1(Q_{12})$	&	$e_1(Q_{13})$		&	$r_1(Q_{13})$&	$e_1(Q_{22})$		&	$r_1(Q_{22})$	&	$e_1(Q_{23})$ 	&	$r_1(Q_{23})$ \\
\hline
1.00-3	&	1.2022e-4	&	-		&	1.9040e-4	&	-		&	2.6974e-4	&	-	&	2.4119e-4	&	-		&	3.0823e-4	&	-	\\
5.00e-4	&	3.0072e-5	&	1.9992	&	4.7602e-5	&	1.9992	&	6.7472e-5	&	1.9992&	6.0332e-5	&	1.9992	&	7.7010e-5	&	1.9992	\\
3.33e-4	&	1.3361e-5	&	2.0008	&	2.1149e-5	&	2.0008	&	2.9977e-5	&	2.0008&	2.6804e-5	&	2.0008	&	3.4255e-5	&	2.0009	\\
2.50e-4	&	7.5107e-6	&	2.0022	&	1.1888e-5	&	2.0022	&	1.6851e-5	&	2.0022& 	1.5068e-5	&	2.0022	&	1.9256e-5	&	2.0022	\\
2.00e-4	&	4.8026e-6	&	2.0039	&	7.6022e-6	&	2.0039	&	1.0776e-5	&	2.0039&	9.6349e-6	&	2.0039	&	1.2313e-5	&	2.0039	\\
\end{tabular}}
\caption{\label{tab:od2cConvergenceH1} EOC for Scheme OD2C with respect to the discrete $H^1$ norm.}
\end{center}
\end{table}

% OD1D L2
\begin{table}
\begin{center}
\resizebox{\textwidth}{!}{%
\begin{tabular}{l l l l l l l l l l l}
\hline
$\dt$	 	&	$e_2(Q_{11})$			&	$r_2(Q_{11})$			&	$e_2(Q_{12})$ 	&	$r_2(Q_{12})$	&	$e_2(Q_{13})$		&	$r_2(Q_{13})$ &	$e_2(Q_{22})$		&	$r_2(Q_{22})$	&	$e_2(Q_{23})$ 	&	$r_2(Q_{23})$\\
\hline
1.00e-3	&	5.1966e-6	&	-		&	3.8651e-6	&	-		&	3.1535e-6	&	-	&	4.1587e-6	&	-		&	3.7933e-6	&	-		\\
5.00e-4	&	2.5468e-6	&	1.0289	&	1.8352e-6	&	1.0746	&	1.3610e-6	&	1.2122&	1.9408e-6	&	1.0995	&	1.6627e-6	&	1.1899	\\
3.33e-4	&	1.6771e-6	&	1.0303	&	1.2010e-6	&	1.0457	&	8.7098e-7	&	1.1009&	1.2654e-6	&	1.0549	&	1.0684e-6	&	1.0908	\\
2.50e-4	&	1.2439e-6	&	1.0387	&	8.8878e-7	&	1.0464	&	6.3917e-7	&	1.0757&	9.3516e-7	&	1.0512	&	7.8530e-7	&	1.0701	\\
2.00e-4	&	9.8437e-7	&	1.0487	&	7.0261e-7	&	1.0533	&	5.0325e-7	&	1.0714&	7.3880e-7	&	1.0563	&	6.1880e-7	&	1.0679	\\
\end{tabular} }
\caption{\label{tab:od1dConvergenceL2} EOC for Scheme OD1D with respect to the discrete $L^2$ norm.}
\end{center}
\end{table}
%OD1D H1
\begin{table}
\begin{center}
\resizebox{\textwidth}{!}{%
\begin{tabular}{l l l l l l l l l l l } 
\hline
$\dt$	 	&	$e_1(Q_{11})$			&	$r_1(Q_{11})$			&	$e_1(Q_{12})$ 	&	$r_1(Q_{12})$	&	$e_1(Q_{13})$		&	$r_1(Q_{13})$ 	&	$e_1(Q_{22})$		&	$r_1(Q_{22})$	&	$e_1(Q_{23})$ 	&	$r_1(Q_{23})$ \\
\hline
1.00e-3	&	1.3190e-4	&	-		&	1.9599e-4	&	-		&	2.7262e-4	&	-	&	2.4470e-4	&	-		&	3.1185e-4	&	-		\\
5.00e-4	&	4.0126e-5	&	1.7168	&	5.2966e-5	&	1.8876	&	7.0106e-5	&	1.9593&	6.3621e-5	&	1.9434	&	8.0445e-5	&	1.9548	\\
3.33e-4	&	2.1984e-5	&	1.4840	&	2.6120e-5	&	1.7436	&	3.2452e-5	&	1.8997&	2.9908e-5	&	1.8616	&	3.7411e-5	&	1.8882	\\
2.50e-4	&	1.4954e-5	&	1.3395	&	1.6456e-5	&	1.6060	&	1.9187e-5	&	1.8268&	1.7988e-5	&	1.7674	&	2.2237e-5	&	1.8082	\\
2.00e-4	&	1.1296e-5	&	1.2574	&	1.1787e-5	&	1.4955	&	1.2981e-5	&	1.7512&	1.2375e-5	&	1.6760	&	1.5126e-5	&	1.7269	\\
\end{tabular} }
\caption{\label{tab:od1dConvergenceH1} EOC for Scheme OD1D with respect to the discrete $H^1$ norm.}
\end{center}
\end{table}

\begin{figure}[h]
\includegraphics[height=0.23\textwidth]{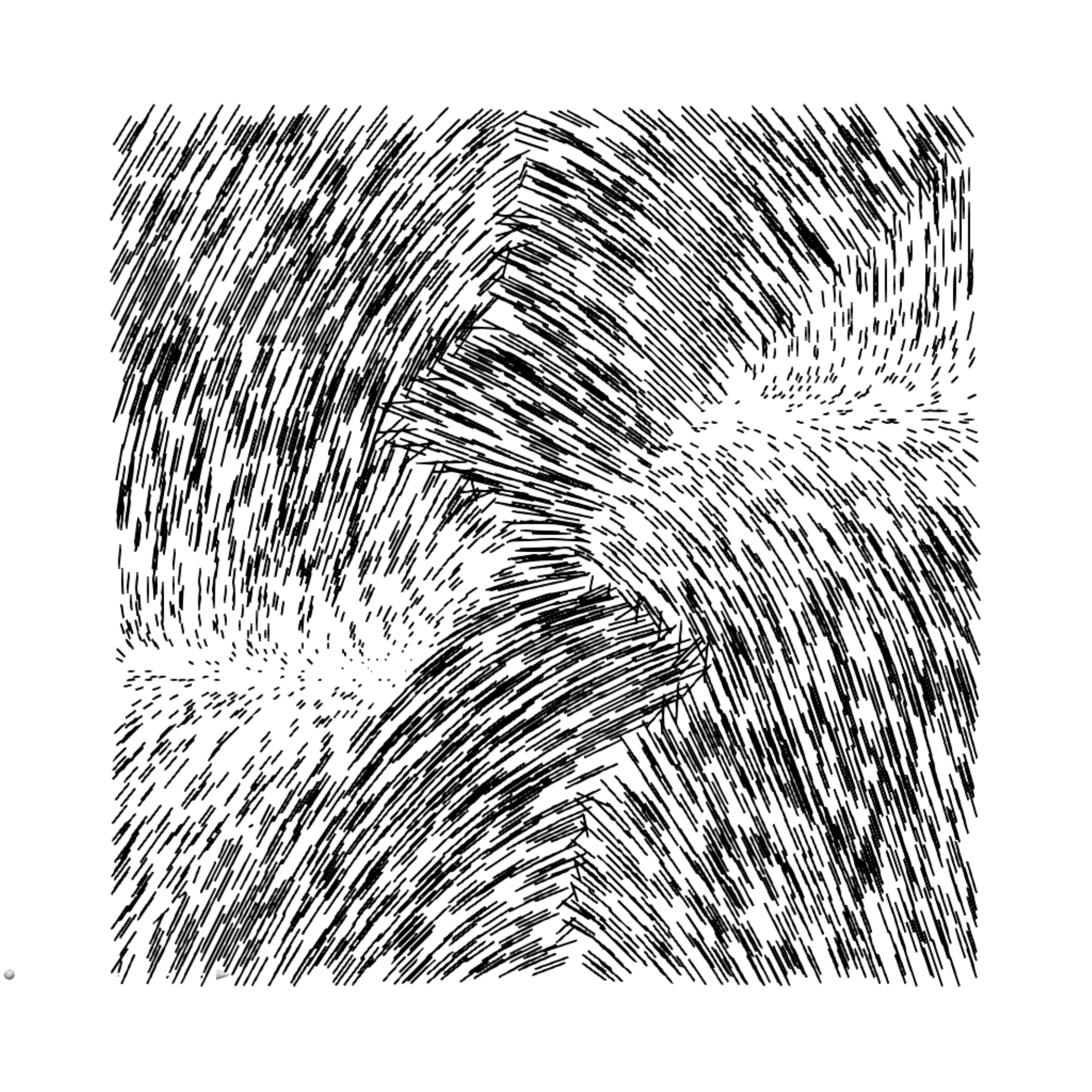}
\includegraphics[height=0.23\textwidth]{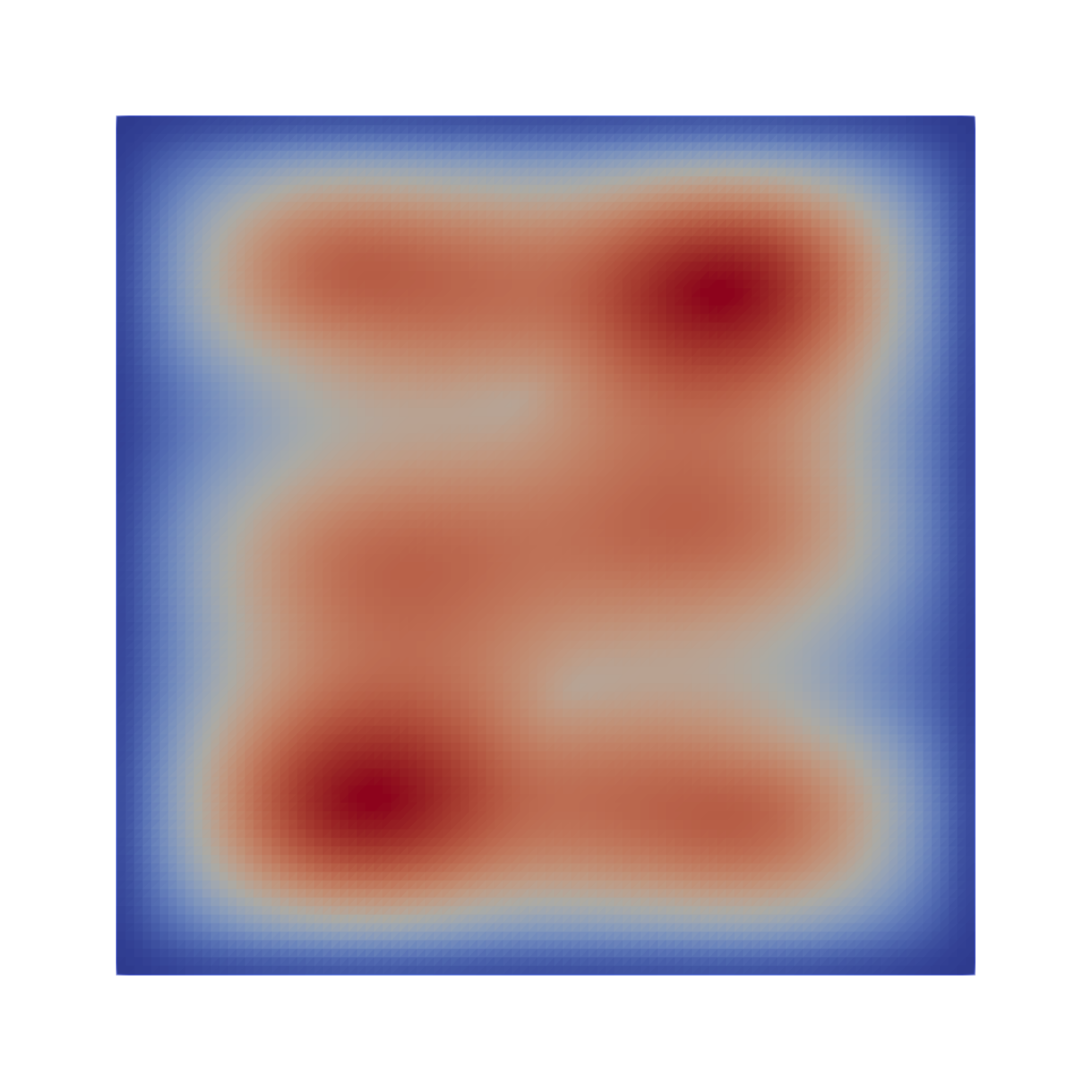}
\includegraphics[height = 0.23\textwidth]{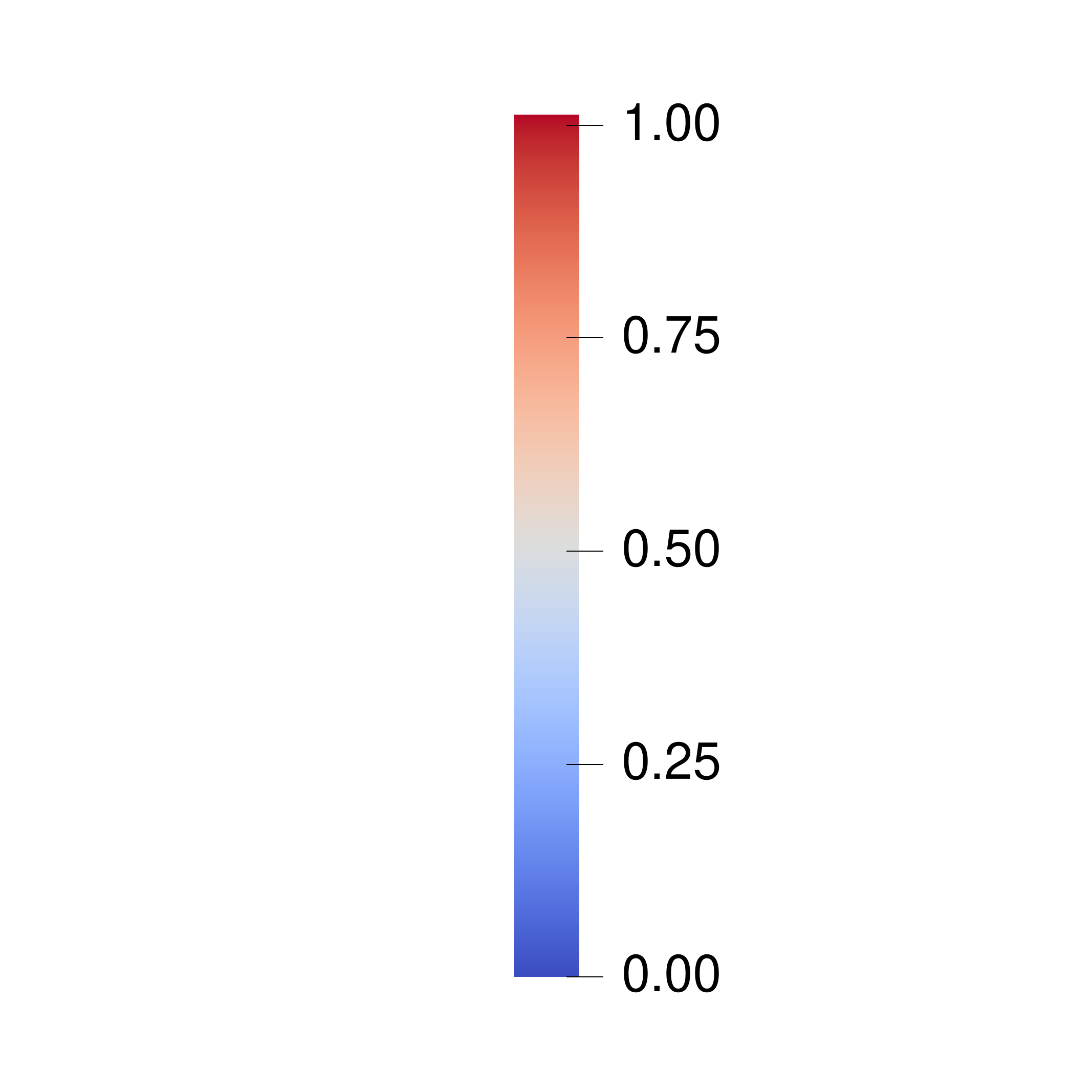}
\caption{\label{fig:convRate} Convergence rate reference solution $\Q^{\mbox{exact}}$ obtained using scheme OD1D. Colors indicate alignment with the dominant eigenvector. Dominant eigenvector is shown with black lines. The color bar for the remaining figures is given on the right.}
\end{figure}

\subsection{Numerical Dissipation}
In this example we study the evolution of the energy and numerical dissipation introduced by each scheme. We consider the domain $\Om = [0,4]^2$ discretized into $50\times 50$ triangular elements, and final time $T=1.0$. For these simulations we use $\eps = 10^{-3}$. The initial condition is 
\beq\label{eq:simulation2}
\ba{rcl} \dis
\mathbf{d}_0 
&=& \dis
\left( \cos\left( 4\mbox{atan}2(x - 2, y - 2 ) \right), \sin\left( 4\mbox{atan}2(x - 2, y - 2 ) \right), 0  \right)^T \,, \\
\hueco \dis
\Q_0
&=&\dis
\frac{\mathbf{d}_0 \mathbf{d}_0^T}{| \mathbf{d}_0 |^2} - \frac13\I \,.
\ea
\eeq
Here, $\mbox{atan2}(x,y)$ is the two argument tangent function. We will impose Neumann boundary conditions $\partial_\mathbf{n} \Q = \bm{0}$. \\
For each numerical scheme, a sequence of time steps $\dt = \left\{\mbox{4e-4,2e-4,1e-4}\right\}$ is considered and the discrete energy and numerical dissipation is computed at each time step. The dynamics computed with the finest time step using scheme OD1D can be seen in \Cref{fig:defectsNeumann}. \\
\Crefrange{fig:dynamicsUES1D}{fig:dynamicsOD1D} show the time evolution of the energy (top row) and the numerical dissipation (bottom row) for each numerical scheme. We can observe the energy is decreasing over the whole time interval, and in \Cref{fig:dynamicsOD2C} and \Cref{fig:dynamicsOD1D} a sharp decrease happens in the energy following the annihilation of all defects around time $t=0.35$ (see \Cref{fig:defectsNeumann}). The numerical dissipation (bottom row) is always positive for the energy stable scheme UES1D, however, schemes OD2C and OD1D introduce much less dissipation. In fact, the dissipation introduced by scheme UES1D is so large, that the dynamics have been slowed down. We note that the time axis in \Cref{fig:dynamicsUES1D} is 20 times longer than in the case of the other two schemes. It is only the case of the finest time step $\Delta t = 1e-4$ that the solution gets close to annihilating the defects which corresponds with the increased numerical dissipation after time $t=15$.\\
In \Cref{fig:dynamicsUES1D_S} we compare the dynamics obtained using scheme UES1D with different values of the stabilization constants $S_1$ and $S_3$. Specifically, we try the values given in \Cref{rem:stabilizationValues} where $S_1 = 848$, and $S_3=208$, and compare to $S_1 = 100,$ $S_3 = 20$, and $S_1 = 10$, $S_3 = 2$. In the figure we can see that larger values of the stabilization parameters will slow the dynamics, and it is only in the third case that we observe the same sharp decrease in the energy corresponding to defect annihilation as was present in \Cref{fig:dynamicsOD2C,fig:dynamicsOD1D}. The numerical dissipation introduced at the same time can bee seen to rise to over 600 which is much larger that what was introduce by scheme OD2C ($\leq 30$) and scheme OD1D ( $\leq 150$). In \Cref{fig:dynamicsCompare} we show the results obtained with each scheme on the same plot. The values shown for scheme UES1D in \Cref{fig:dynamicsCompare} correspond to the case of the smallest stabilization constants in the previous figure. \\
In \Cref{tab:compCost} we compare the computational time needed for each scheme to complete $10,000$ iterations using the same standard desktop computer. In this case, the decoupling of the unknowns in scheme OD1D saves nearly two thirds of the computational time from scheme OD2C. Furthermore, the truncation procedure for scheme UES1D is computationally inefficient resulting in the longest time to compute 10,000 iterations. It also should be noted that the slowed dynamics caused by the high numerical dissipation of scheme UES1D necessitated running the simulation over a larger time interval which increases the required number of iterations, hence the computational cost of the energy stable scheme is even greater.
\begin{table}[h]
\begin{center}
\begin{tabular}{r|c|c|c}
Scheme			&	UES1D	&	OD2C	&	OD1D	\\
\hline
Run Time (seconds) 	&	209,959	&	34,846	&	13,490	
\end{tabular}
\caption{\label{tab:compCost} The time for each scheme to compute 10,000 iterations for the experiment given by \eqref{eq:simulation2}}
\end{center}
\end{table}

\begin{figure}[h]
\begin{center}
\includegraphics[width = 0.23\textwidth]{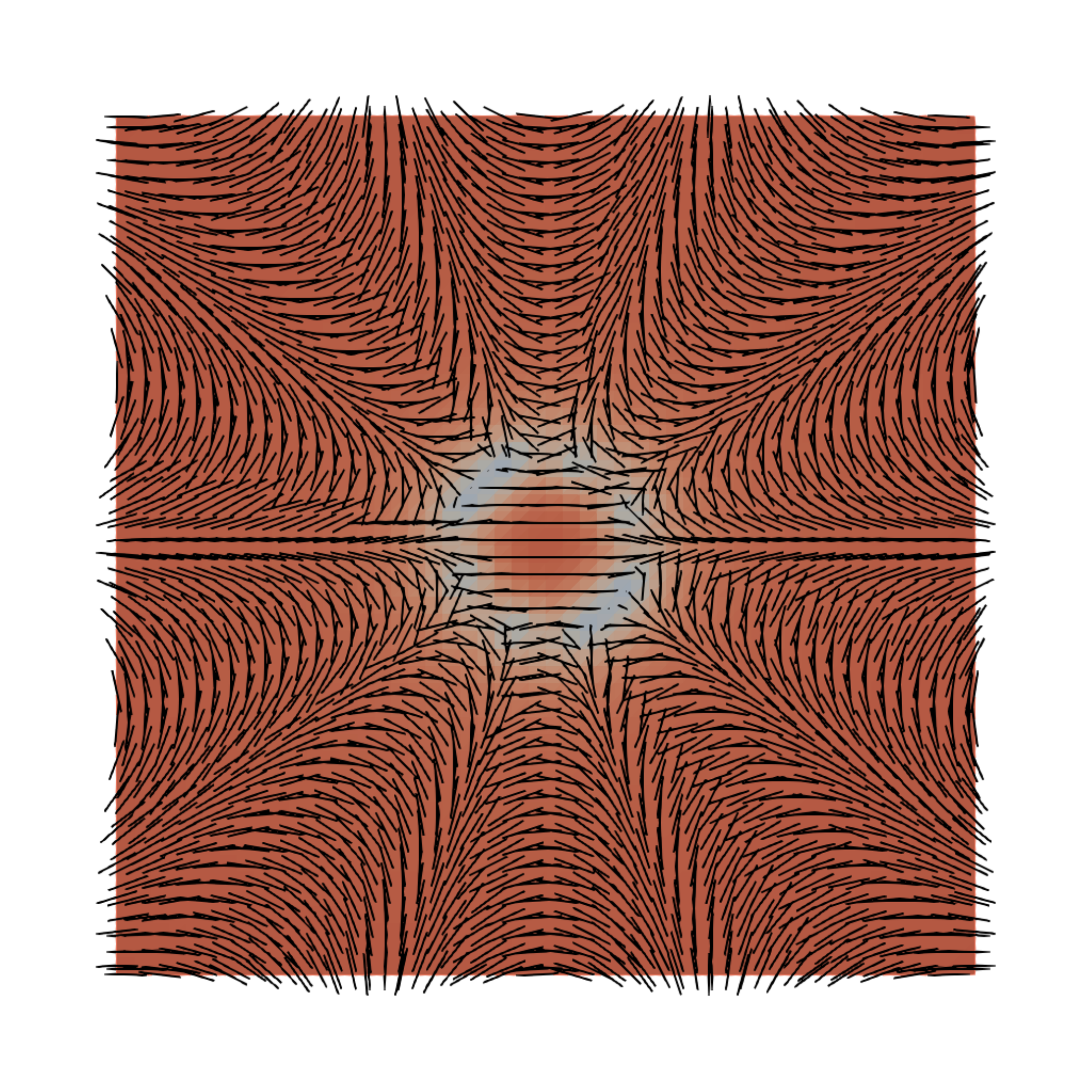}
\includegraphics[width = 0.23\textwidth]{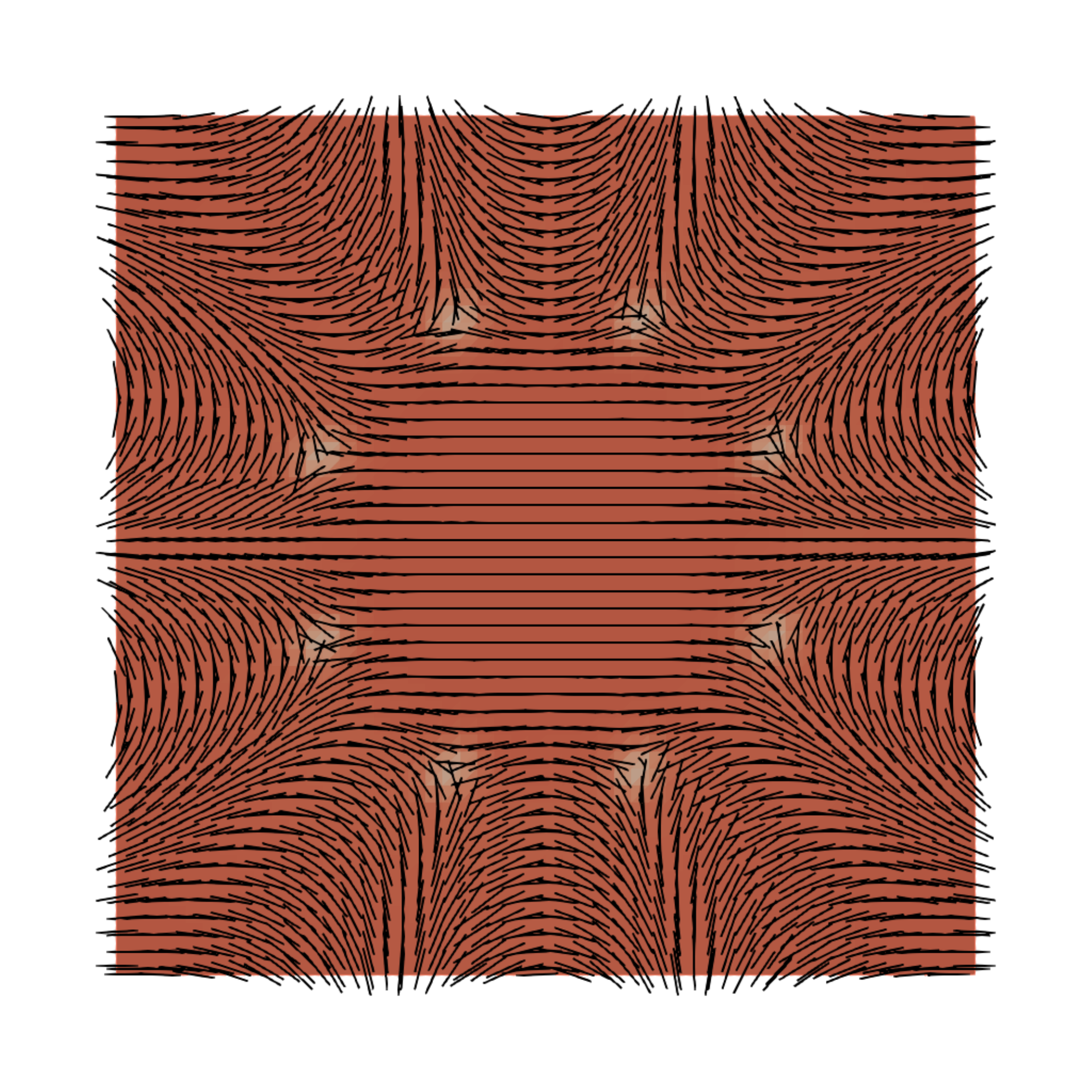}
\includegraphics[width = 0.23\textwidth]{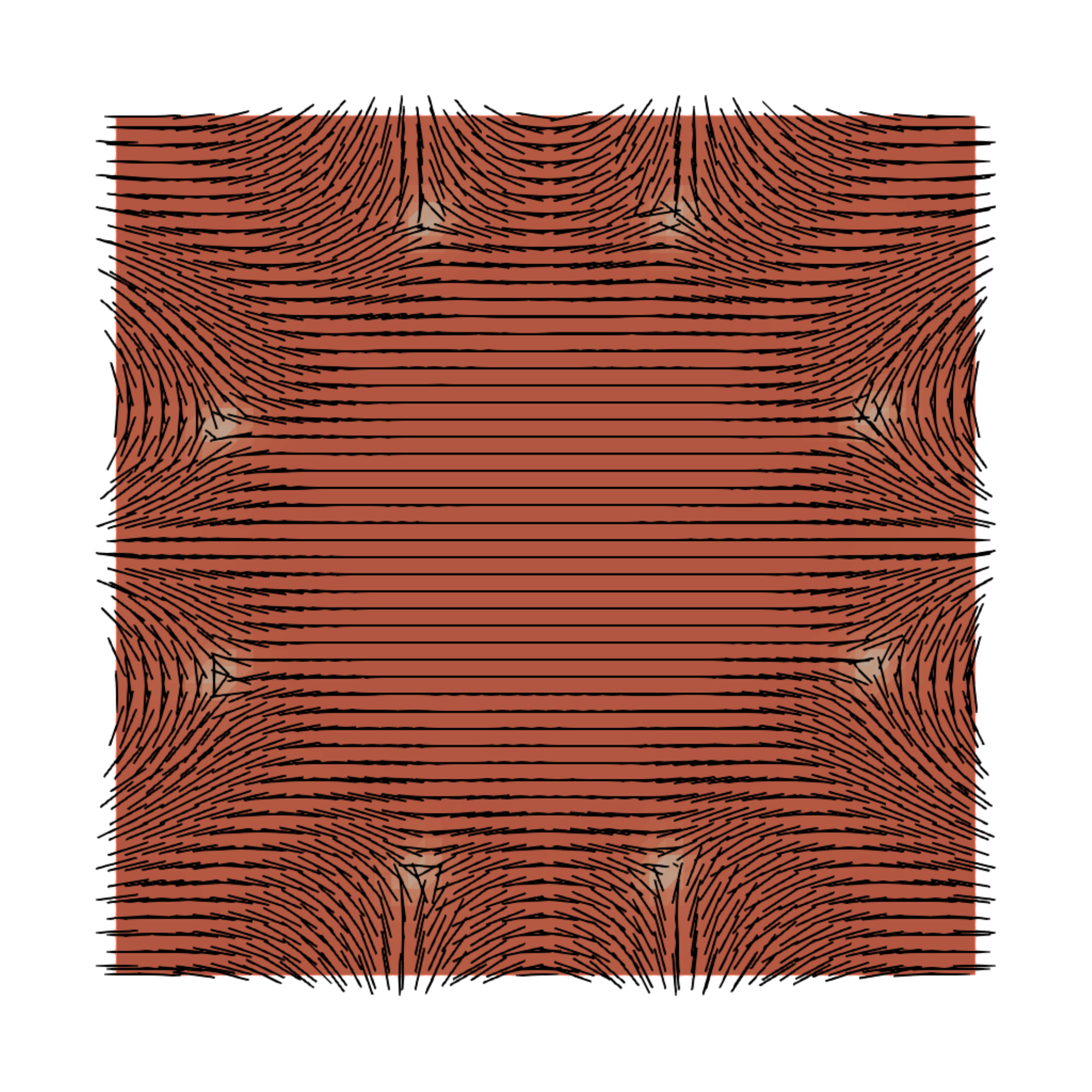} \\
\includegraphics[width = 0.23\textwidth]{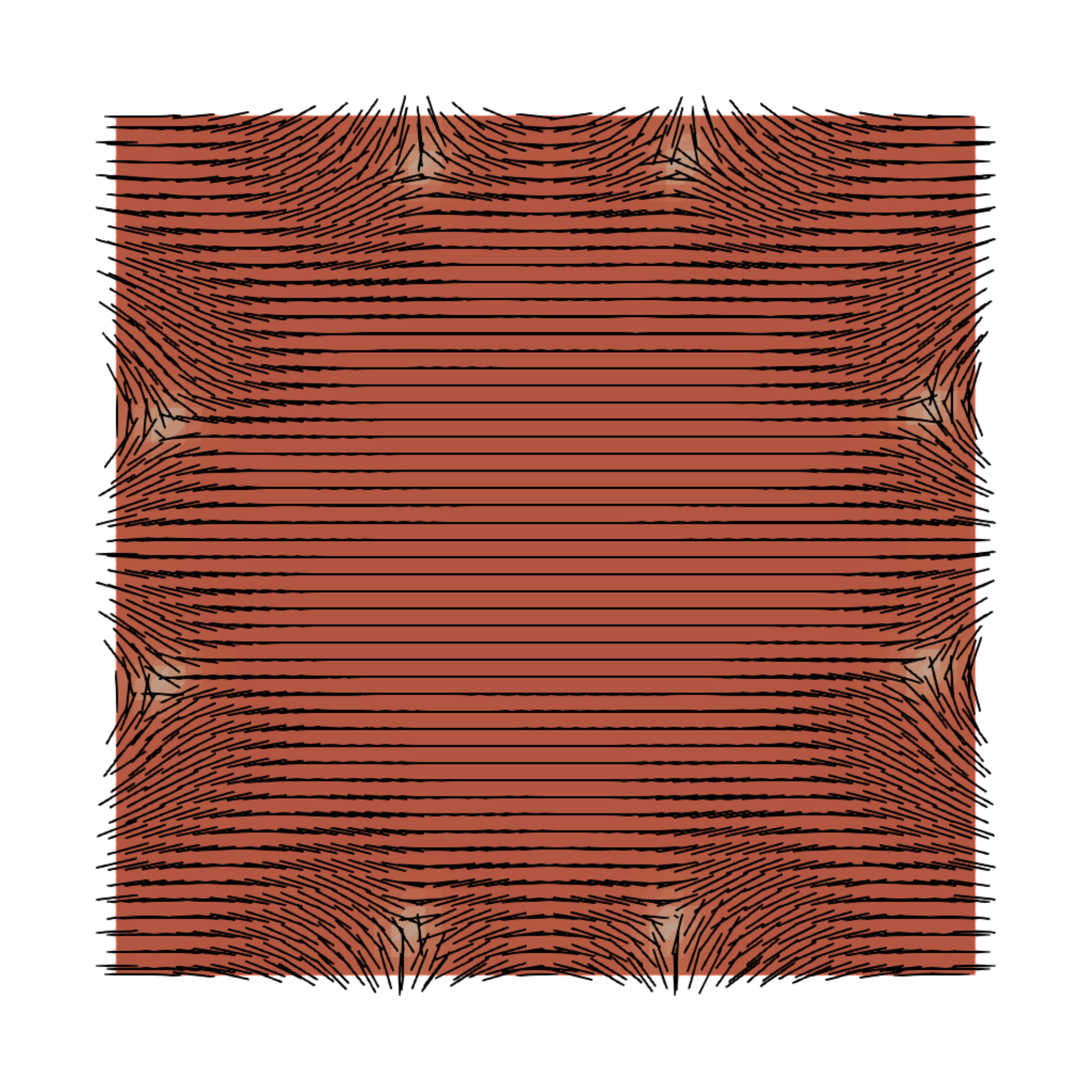}
\includegraphics[width = 0.23\textwidth]{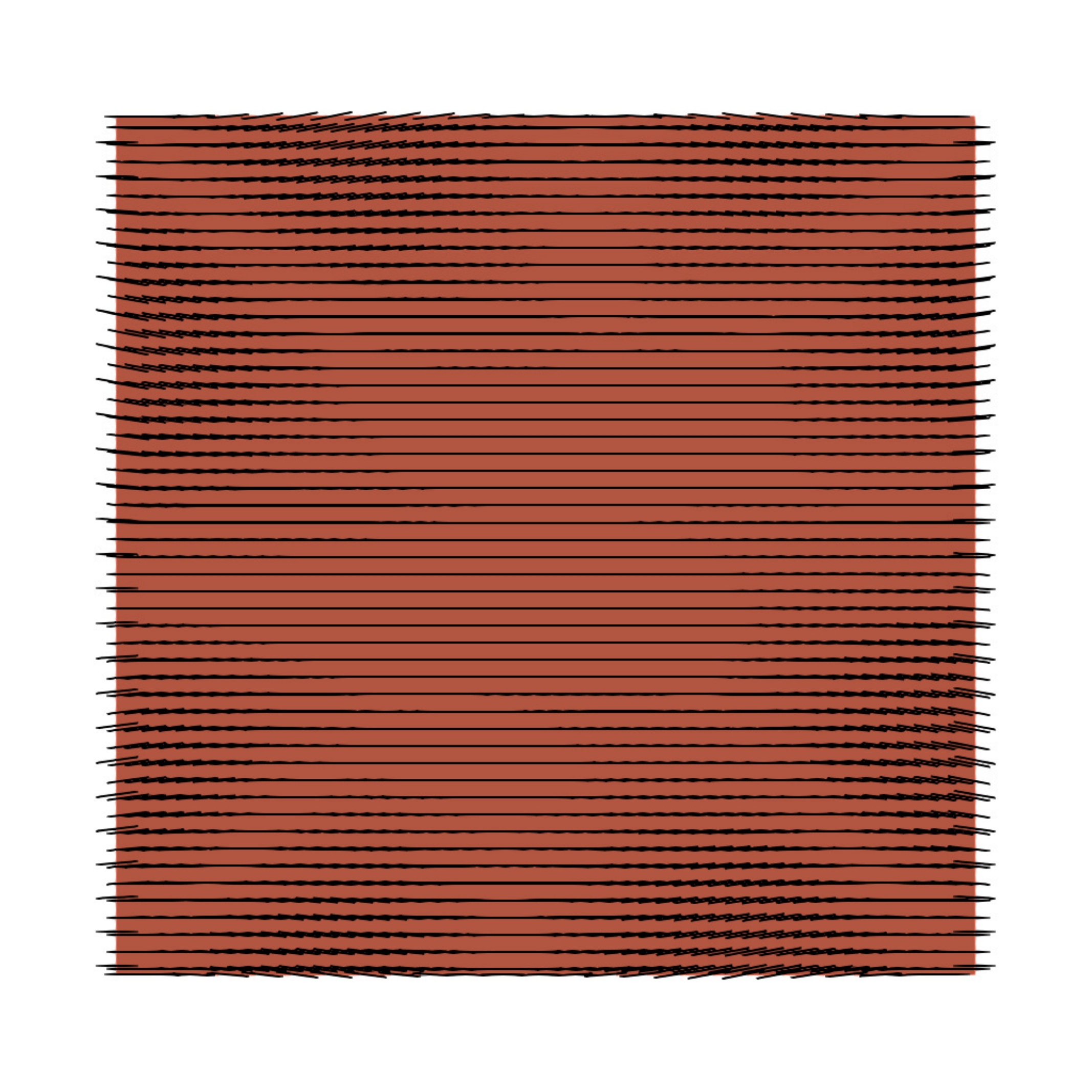}
\includegraphics[width = 0.23\textwidth]{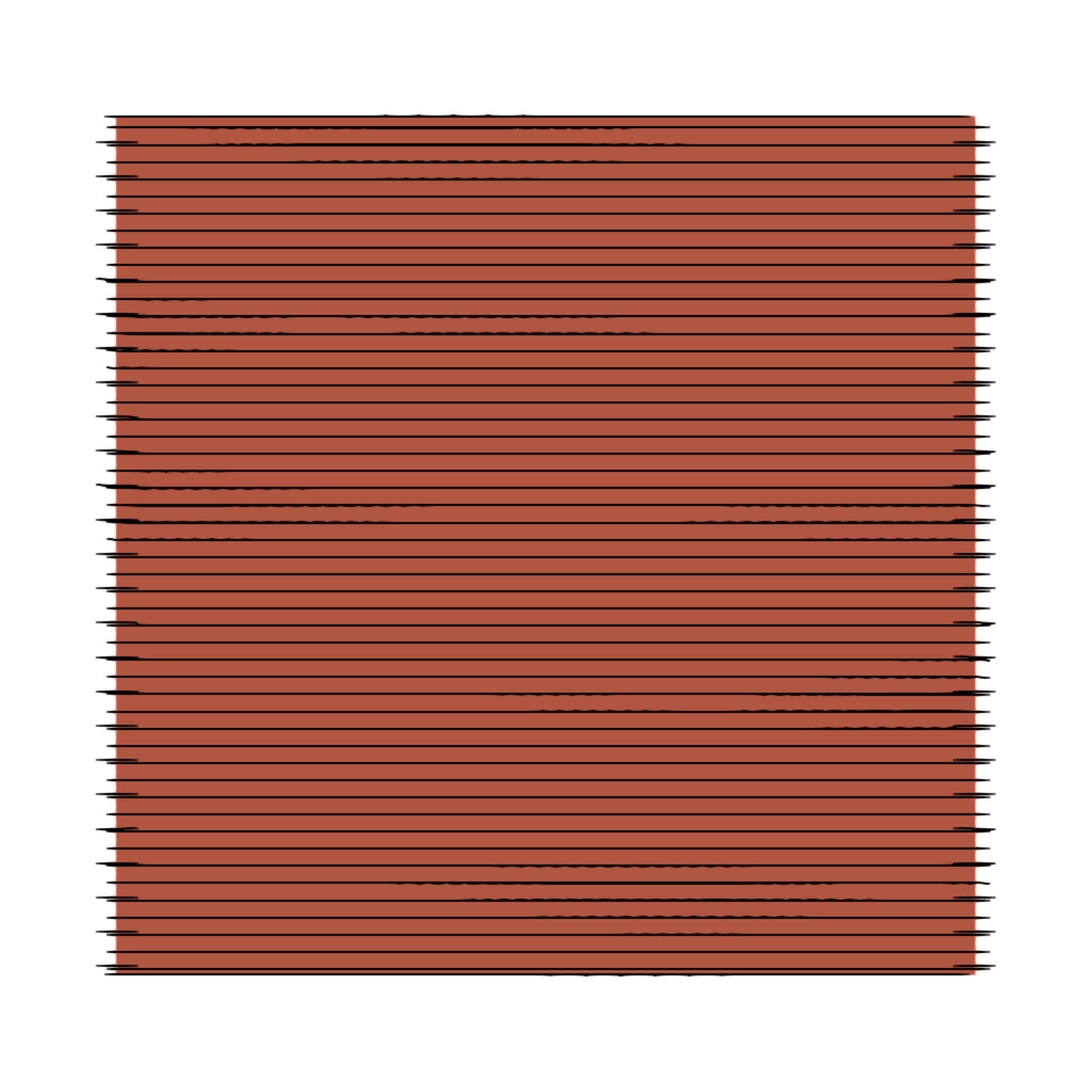}
\caption{\label{fig:defectsNeumann} Defect dynamics using scheme OD1D in 2D with Neumann boundary conditions at times $t=0.01,0.1,0.25,0.35,0.5,1.0$. Color represents the difference of the two largest eigenvalues of $\Q$ and indicates the alignment of the nematic with the dominant eigenvector shown as black lines. }
\end{center}
\end{figure}

\begin{figure}[h]
\begin{center}
\includegraphics[height = 4.75cm]{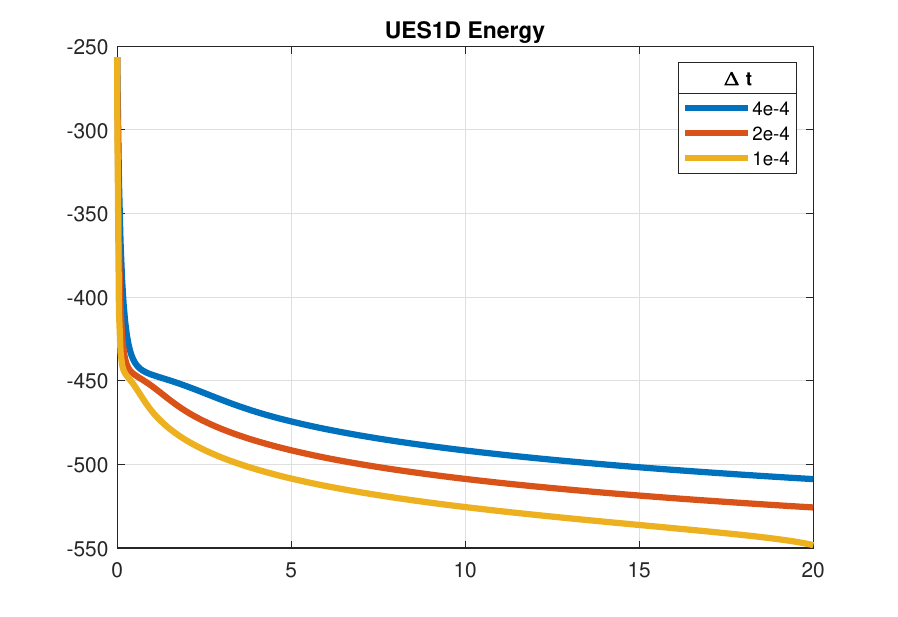}
\includegraphics[height = 4.75cm]{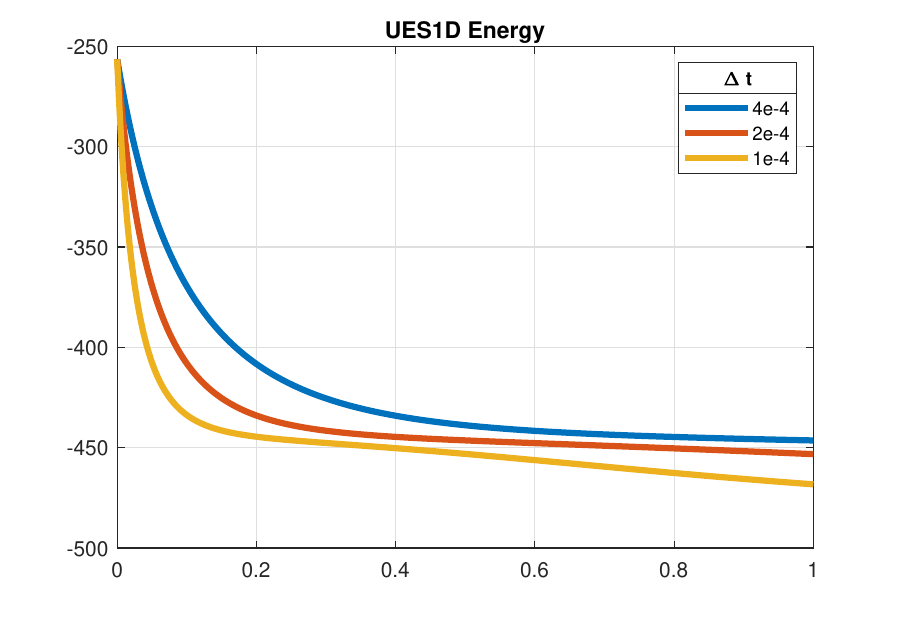} 
\includegraphics[height = 4.75cm]{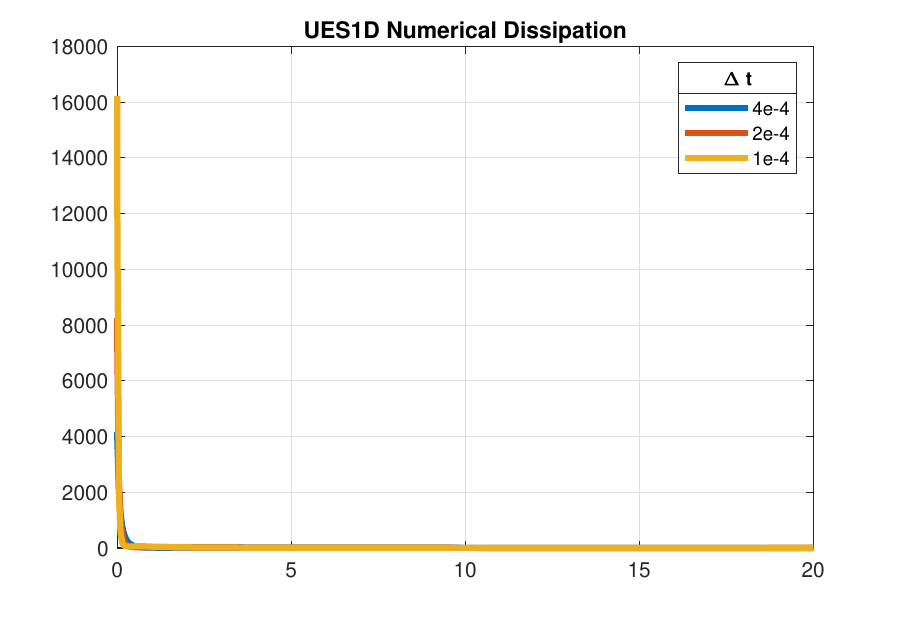}
\includegraphics[height = 4.75cm]{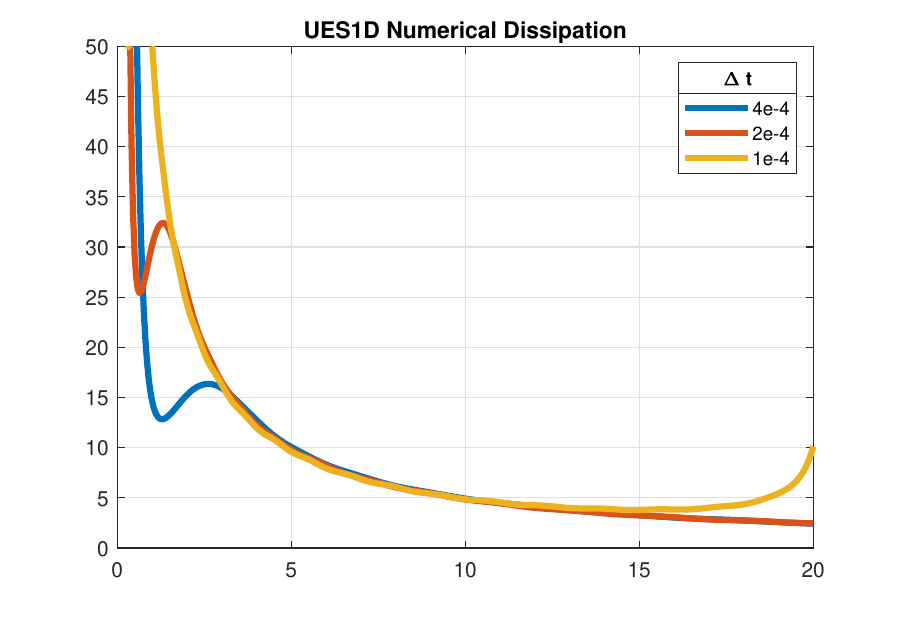}
\caption{\label{fig:dynamicsUES1D} \textit{Top row}: the energy of the system computed with scheme UES1D for different time steps. \textit{Left} shows the energy over the whole time interval, and \textit{right} a zoomed in view of the energy on the time interval $[0.3, 0.5]$. \textit{Bottom row}: the numerical dissipation for different time steps. \textit{Left} shows the numerical dissipation over the whole time interval, and \textit{right} shows a zoomed in view over the time interval $[0.3, 0.5]$.}
\end{center}
\end{figure}

\begin{figure}[h]
\begin{center}
\includegraphics[height = 4.75cm]{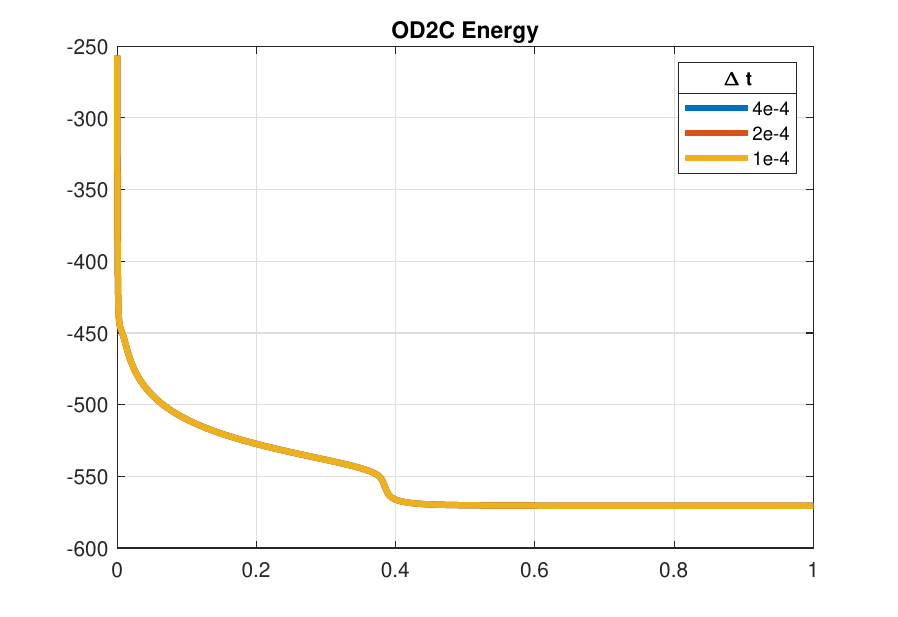}
\includegraphics[height = 4.75cm]{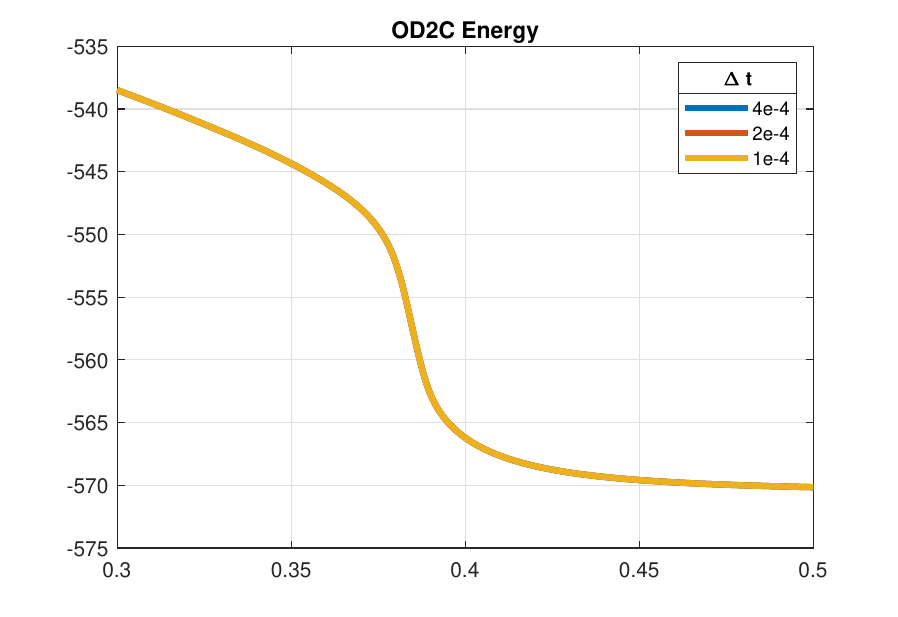}
\includegraphics[height = 4.75cm]{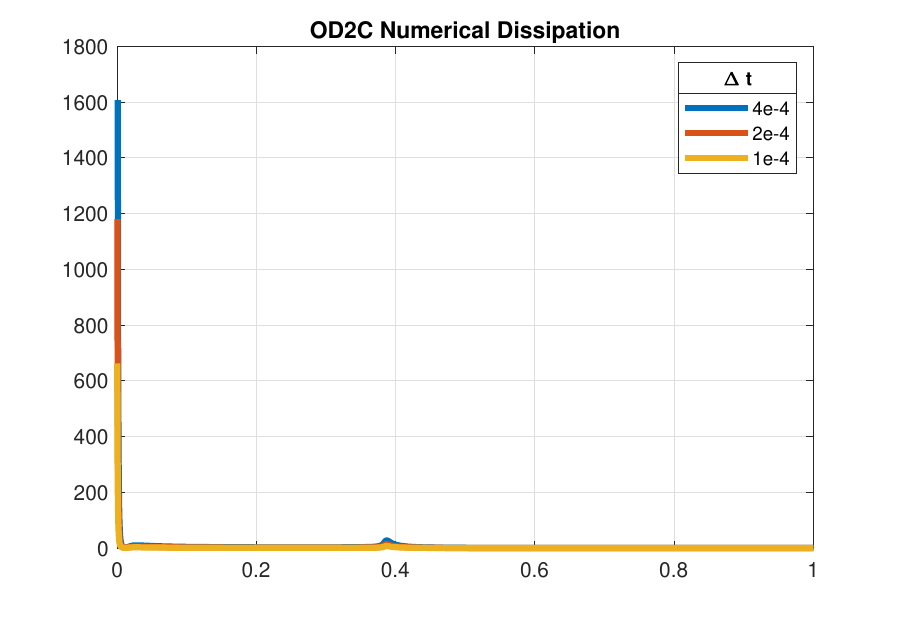}
\includegraphics[height = 4.75cm]{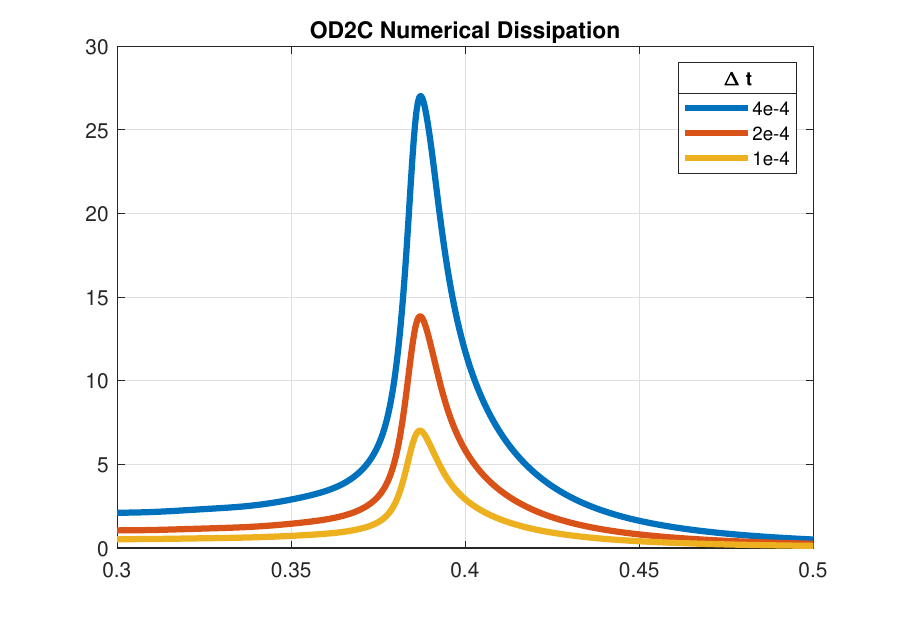}
\caption{\label{fig:dynamicsOD2C} \textit{Top row}: the energy of the system computed with scheme OD2C for different time steps. \textit{Left} shows the energy over the whole time interval, and \textit{right} a zoomed in view of the energy on the time interval $[0.3, 0.5]$. \textit{Bottom row}: the numerical dissipation for different time steps. \textit{Left} shows the numerical dissipation over the whole time interval, and \textit{right} shows a zoomed in view over the time interval $[0.3, 0.5]$.}
\end{center}
\end{figure}

\begin{figure}[h]
\begin{center}
\includegraphics[height = 4.75cm]{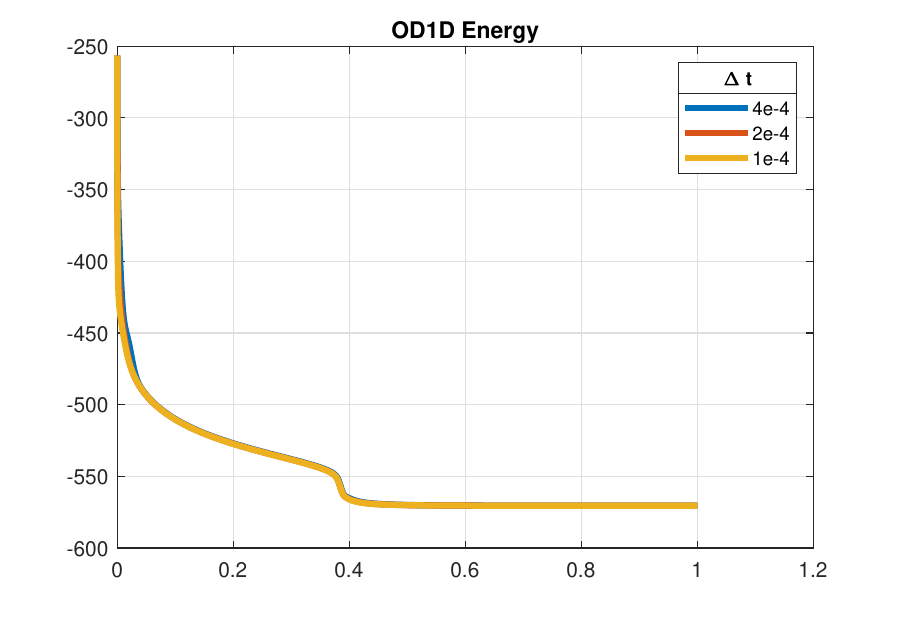}
\includegraphics[height = 4.75cm]{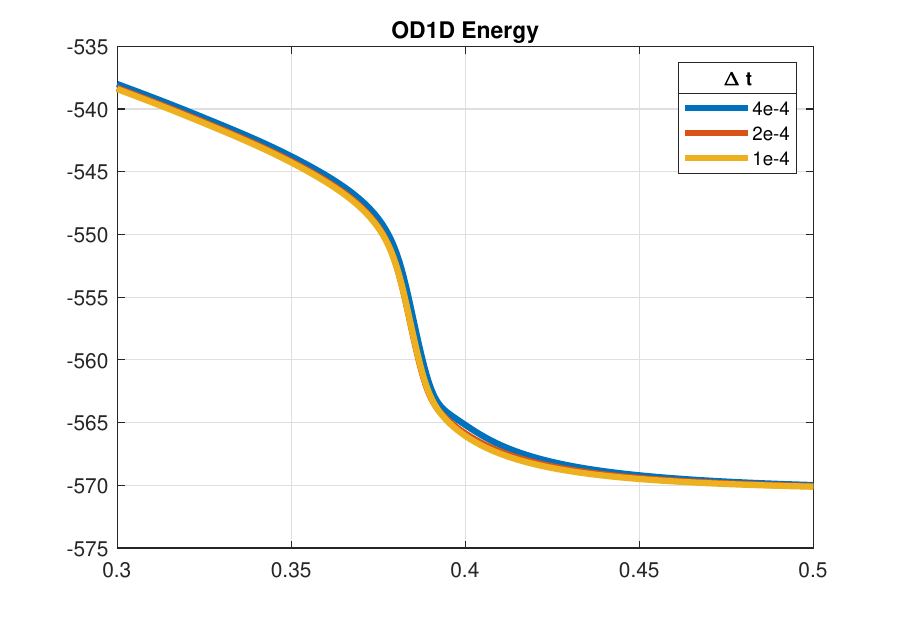}
\includegraphics[height = 4.75cm]{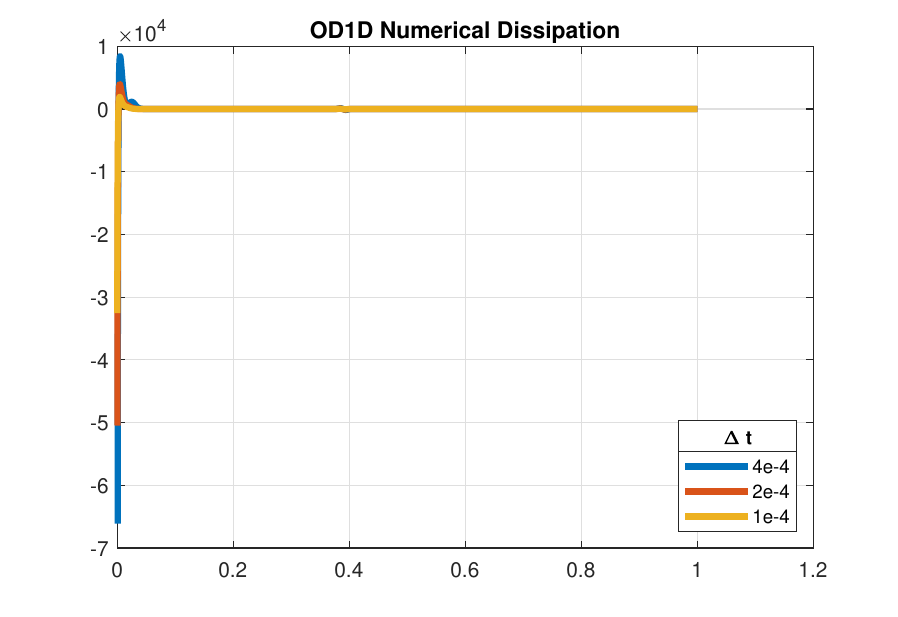}
\includegraphics[height = 4.75cm]{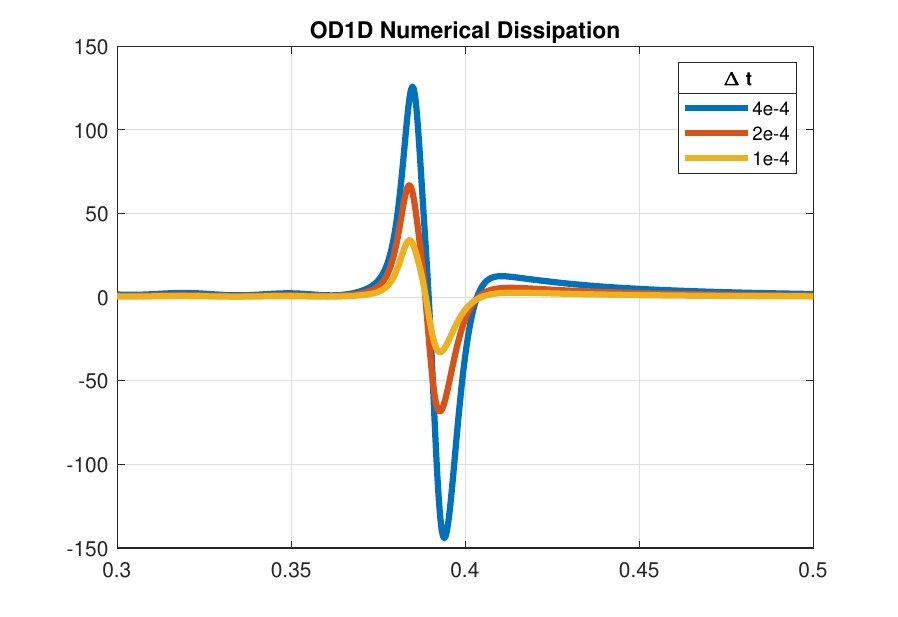}
\caption{\label{fig:dynamicsOD1D} \textit{Top row}: the energy of the system computed with scheme OD1D for different time steps. \textit{Left} shows the energy over the whole time interval, and \textit{right} a zoomed in view of the energy on the time interval $[0.3, 0.5]$. \textit{Bottom row}: the numerical dissipation for different time steps. \textit{Left} shows the numerical dissipation over the whole time interval, and \textit{right} shows a zoomed in view over the time interval $[0.3, 0.5]$.}
\end{center}
\end{figure}

\begin{figure}[h]
\begin{center}
\includegraphics[width = 0.32\textwidth]{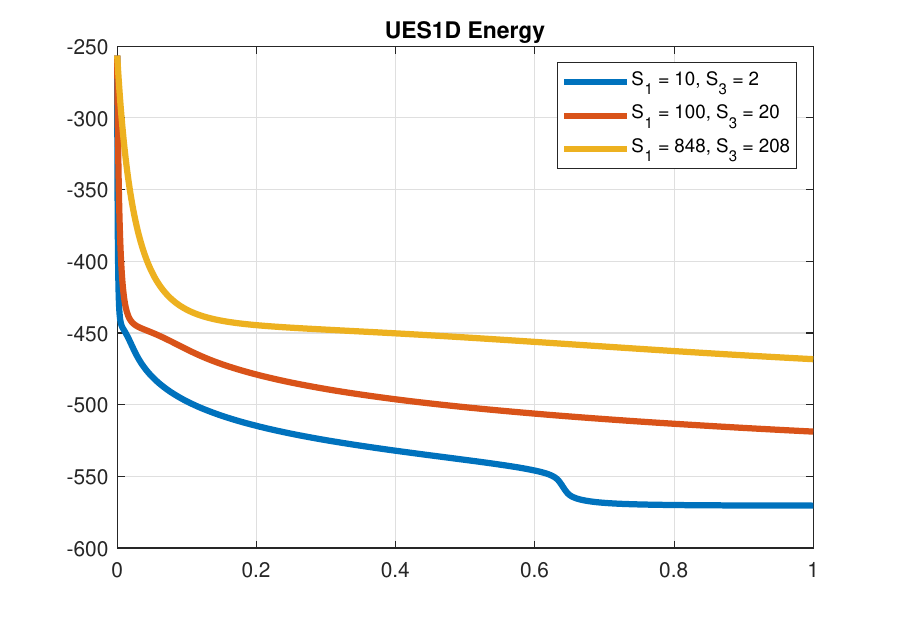}
\includegraphics[width = 0.32\textwidth]{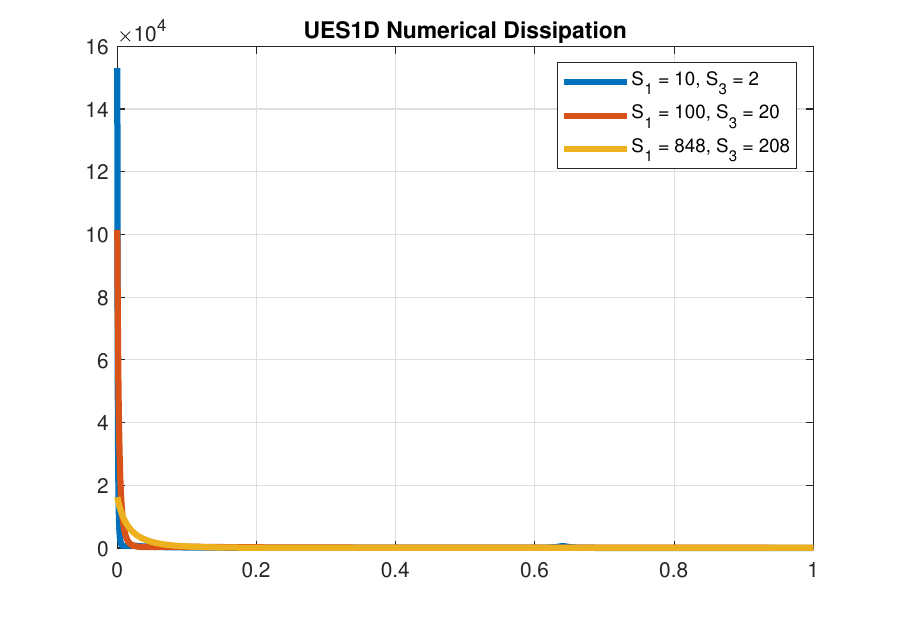}
\includegraphics[width = 0.32\textwidth]{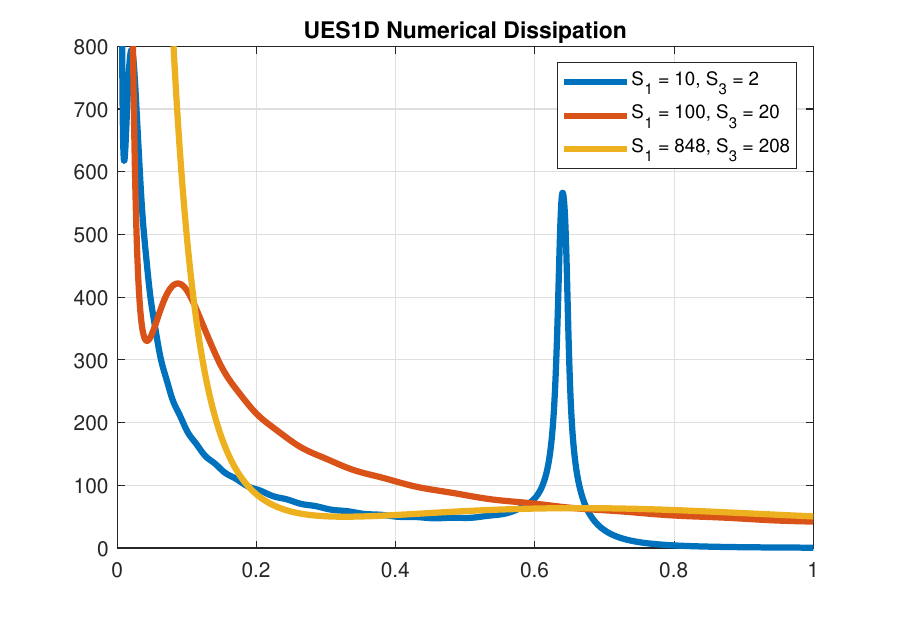}
\caption{\label{fig:dynamicsUES1D_S} Energy and numerical dissipation of a solution computed using scheme UES1D with different values for stabilization constants $S_1$ and $S_3$.}
\end{center}
\end{figure}

\begin{figure}[h]
\begin{center}
\includegraphics[height = 4.75cm]{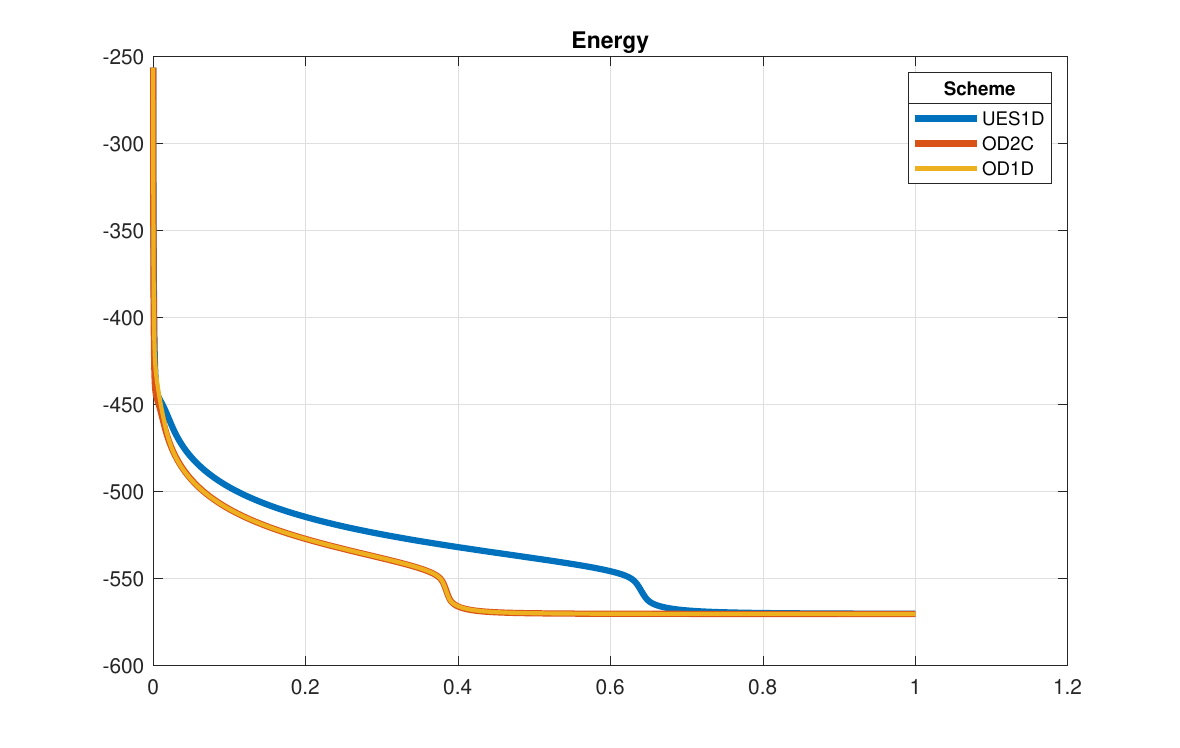}
\includegraphics[height = 4.75cm]{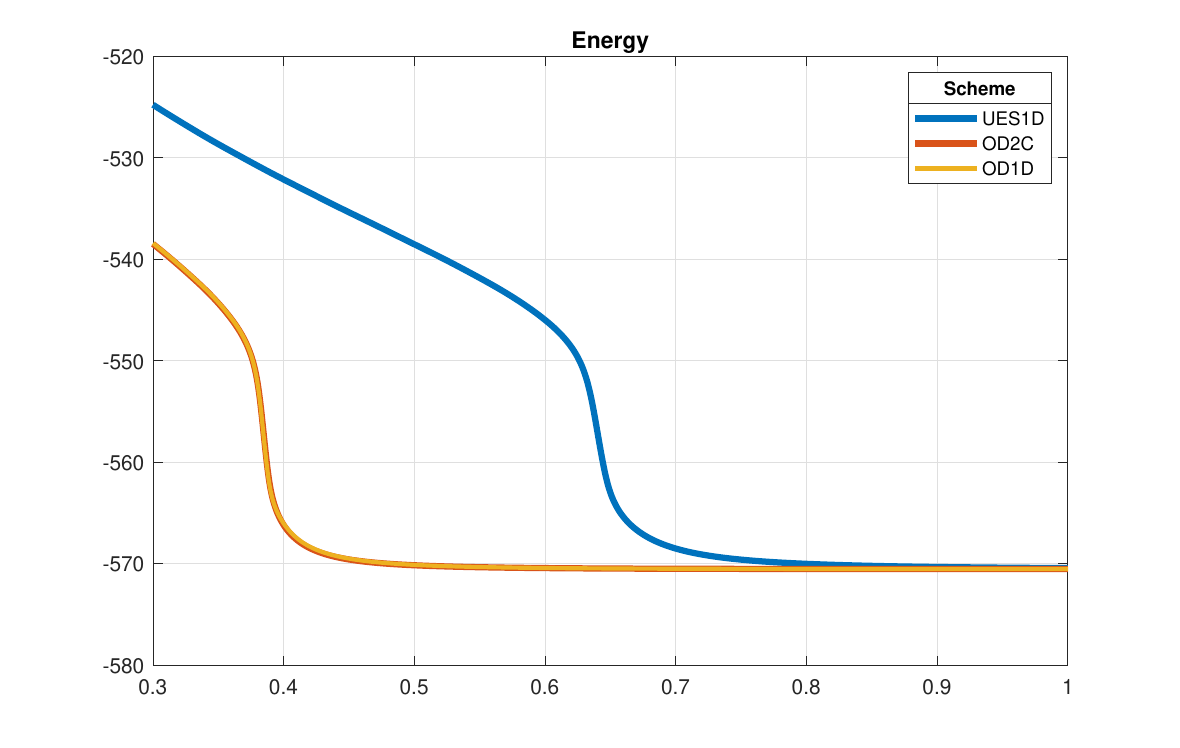} \\
\includegraphics[height = 4.75cm]{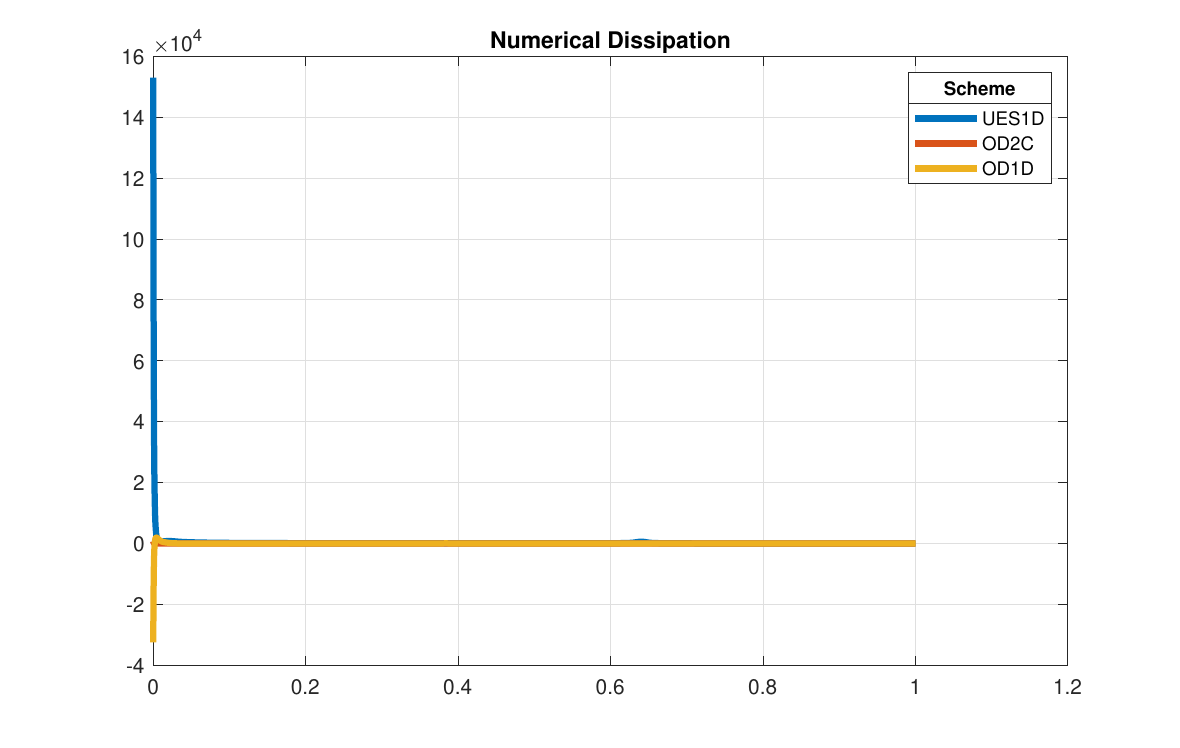}
\includegraphics[height = 4.75cm]{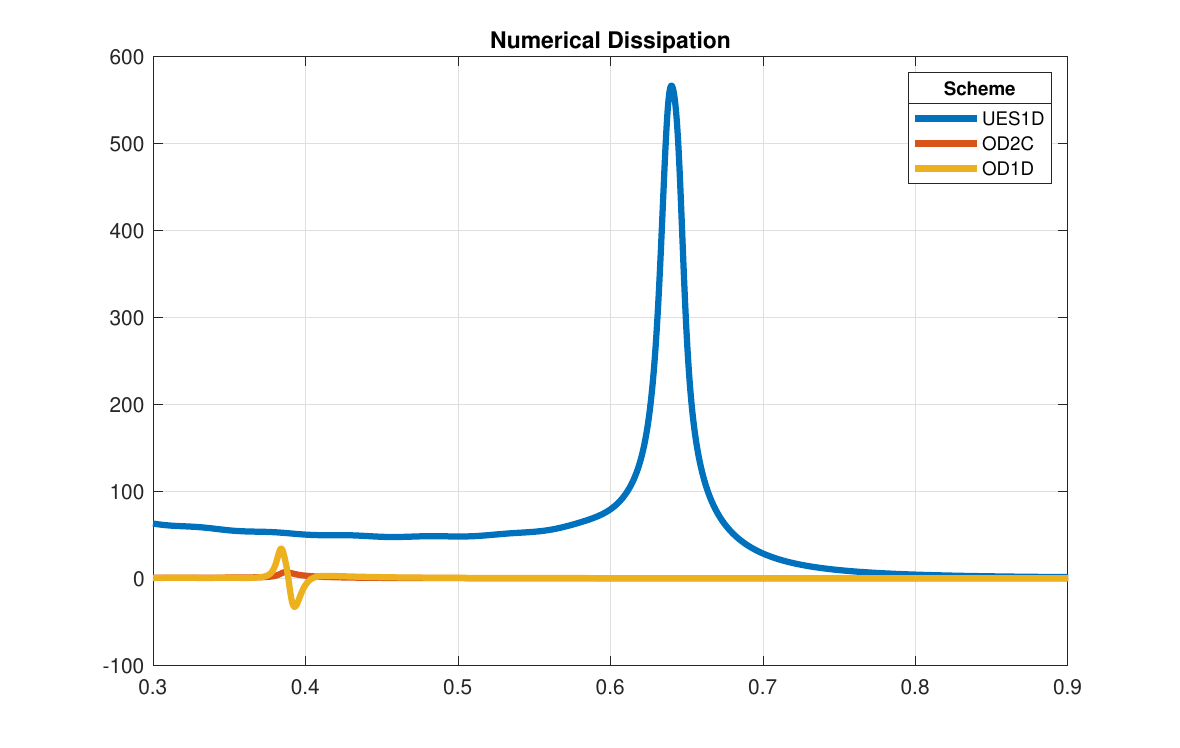}
\caption{\label{fig:dynamicsCompare} Comparing energy and numerical dissipation as a result of each scheme. Scheme UES1D corresponds to the case of the smallest stabilization constants in \Cref{fig:dynamicsUES1D_S}. \textit{Left} shows the curves on the whole time interval. \textit{Right} shows the same curves in in the time interval $[0.3, 0.5]$.}
\end{center}
\end{figure}

\begin{obs}
In the previous two examples we showed that scheme OD1D has equal or better convergence as scheme UES1D while also possessing the higher order numerical dissipation of scheme OD2C. Moreover, scheme OD1D is the most computationally efficient. Therefore, for the numerical experiments moving forward we will be using only scheme OD1D.
\end{obs}

\subsection{Defect Dynamics in 2D}
This example is to show the evolution of defects in 2D subject to different boundary conditions. We use the same initial conditions \eqref{eq:simulation2} from the previous example and change the final time $T=3.5$. The case of Neumann boundary conditions is shown in \Cref{fig:defectsNeumann}. Here, we see the presence of eight $\pm \left. 1 \right/ \left. 2 \right.$ defects which move to and escape through the boundary. \\
Next, we impose the following Dirichlet boundary conditions
\beq\label{eq:dirichlet1}
\mathbf{d}_D 
\, = \, \dis
(0,1,0)^T\, , \quad \quad
\Q_D
\, = \,
\mathbf{d}_D \mathbf{d}_D^T - \frac13\I\, , \quad \quad
\Q \mid_{\partial\Om} 
\, = \,
\Q_D\, .
\eeq
The time evolution of the solution is shown in \Cref{fig:defectsUniform}. We note for this case that the boundary conditions induce the presence of more defects, which move towards the center and annihilate with the original eight defects. The final configuration shows uniform orientation coinciding with the boundary conditions. \\
\begin{figure}[h]
\begin{center}
\includegraphics[width = 0.23\textwidth]{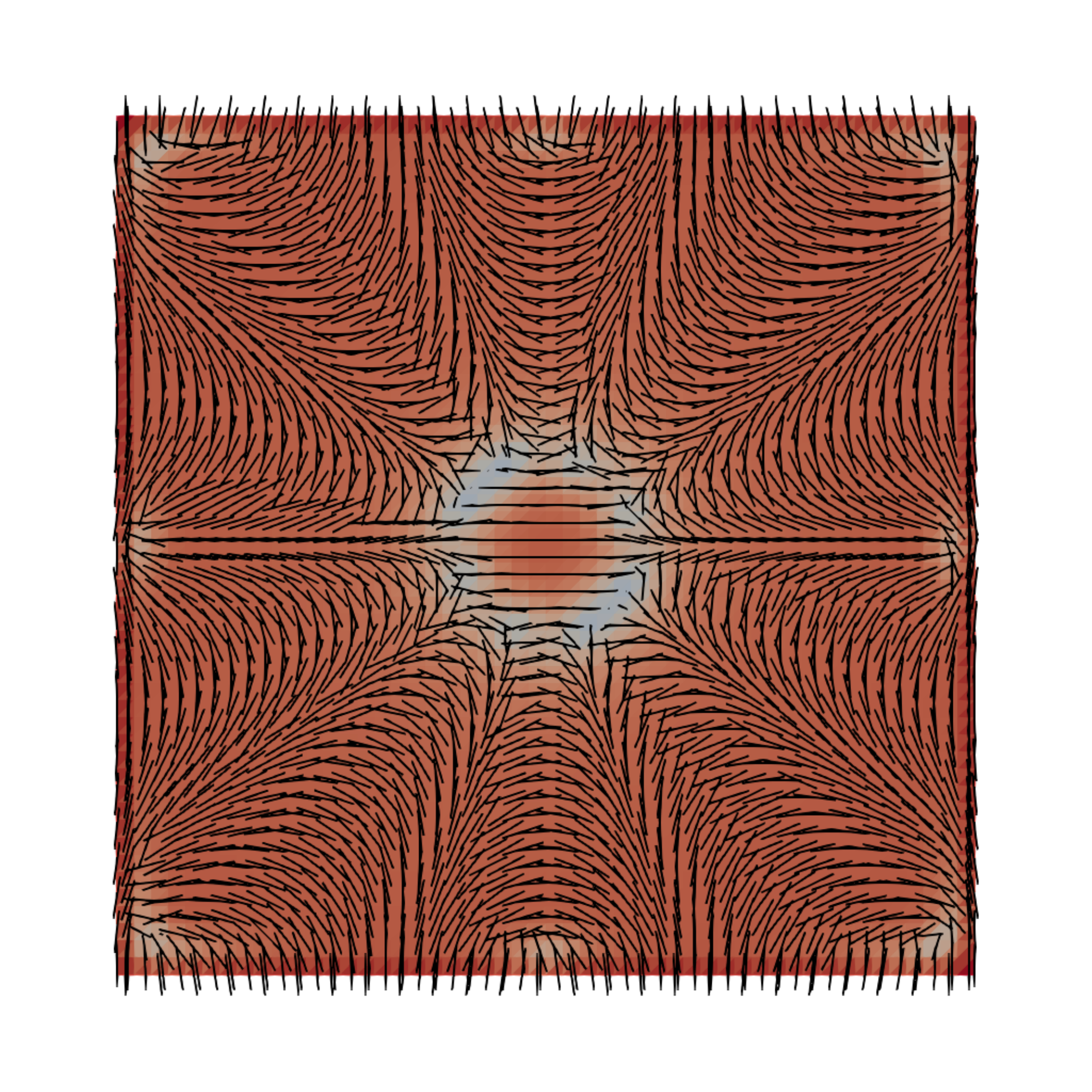}
\includegraphics[width = 0.23\textwidth]{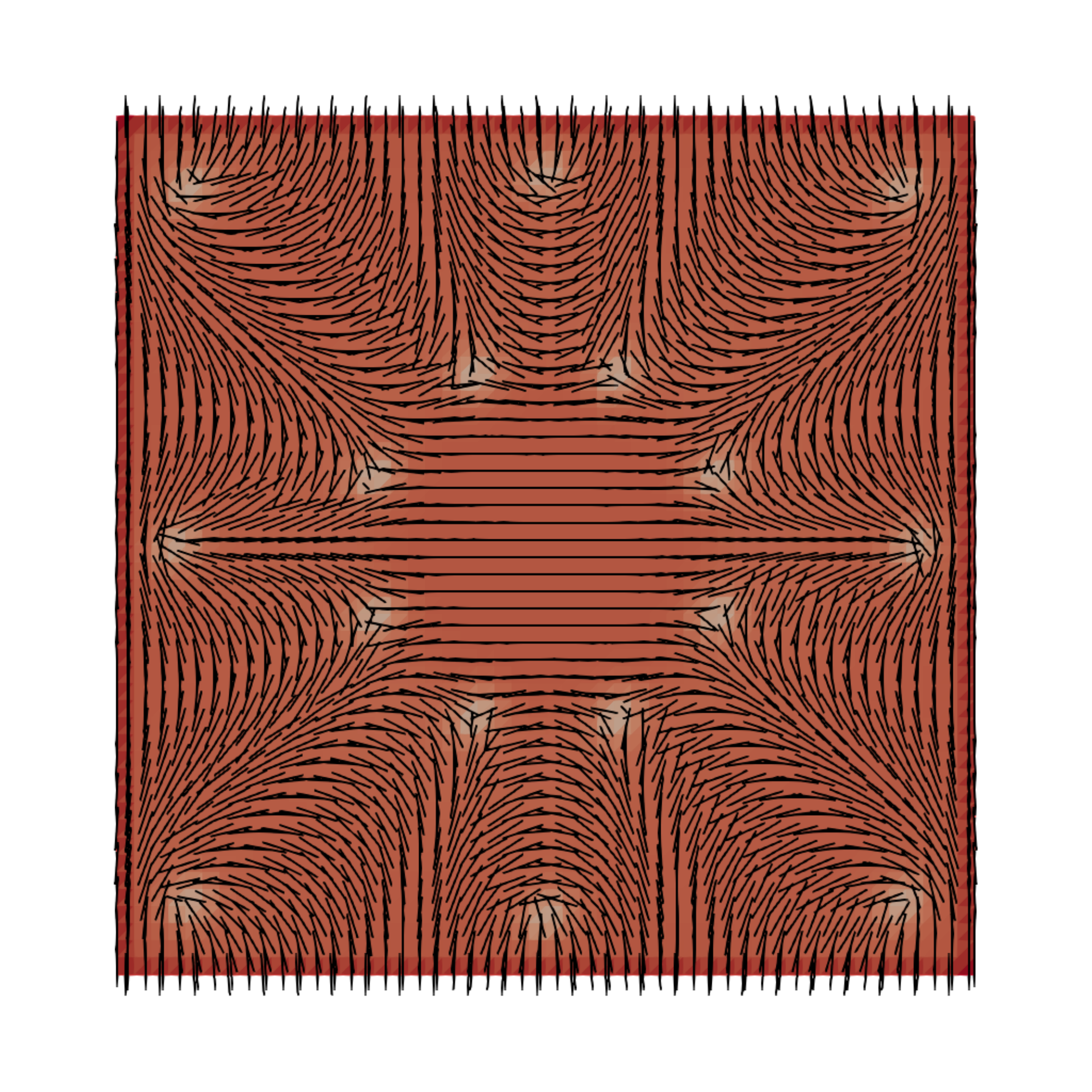}
\includegraphics[width = 0.23\textwidth]{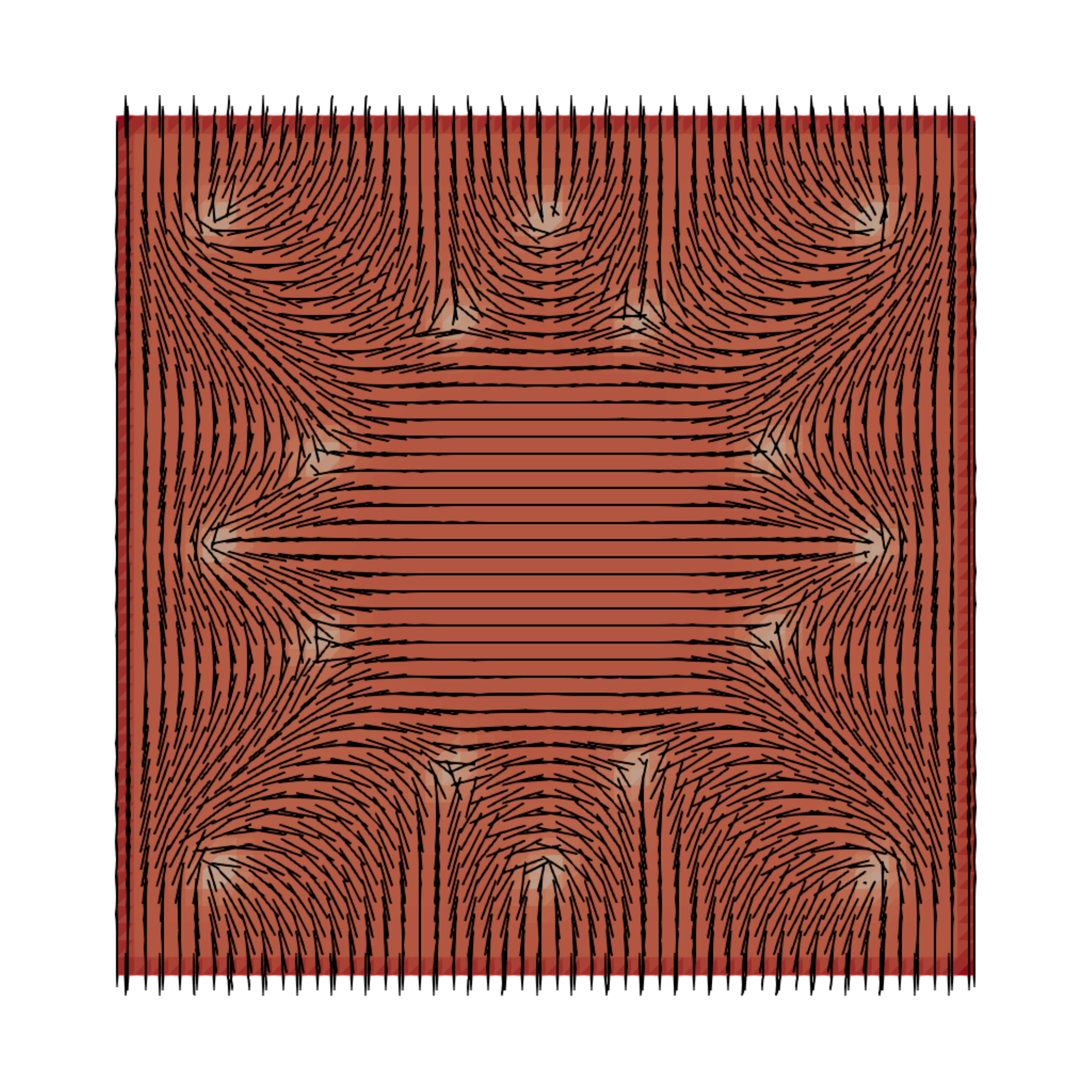}
\includegraphics[width = 0.23\textwidth]{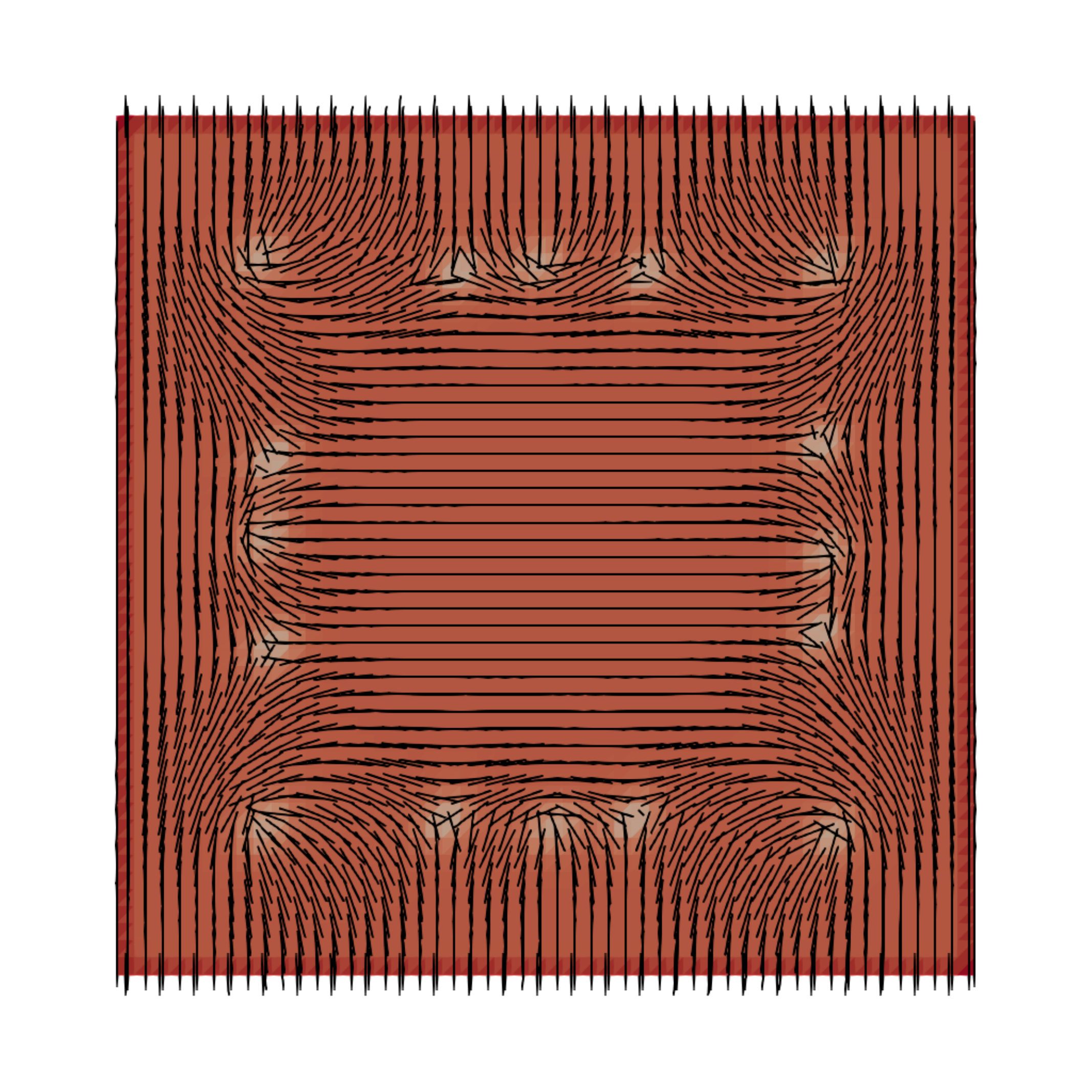}
\includegraphics[width = 0.23\textwidth]{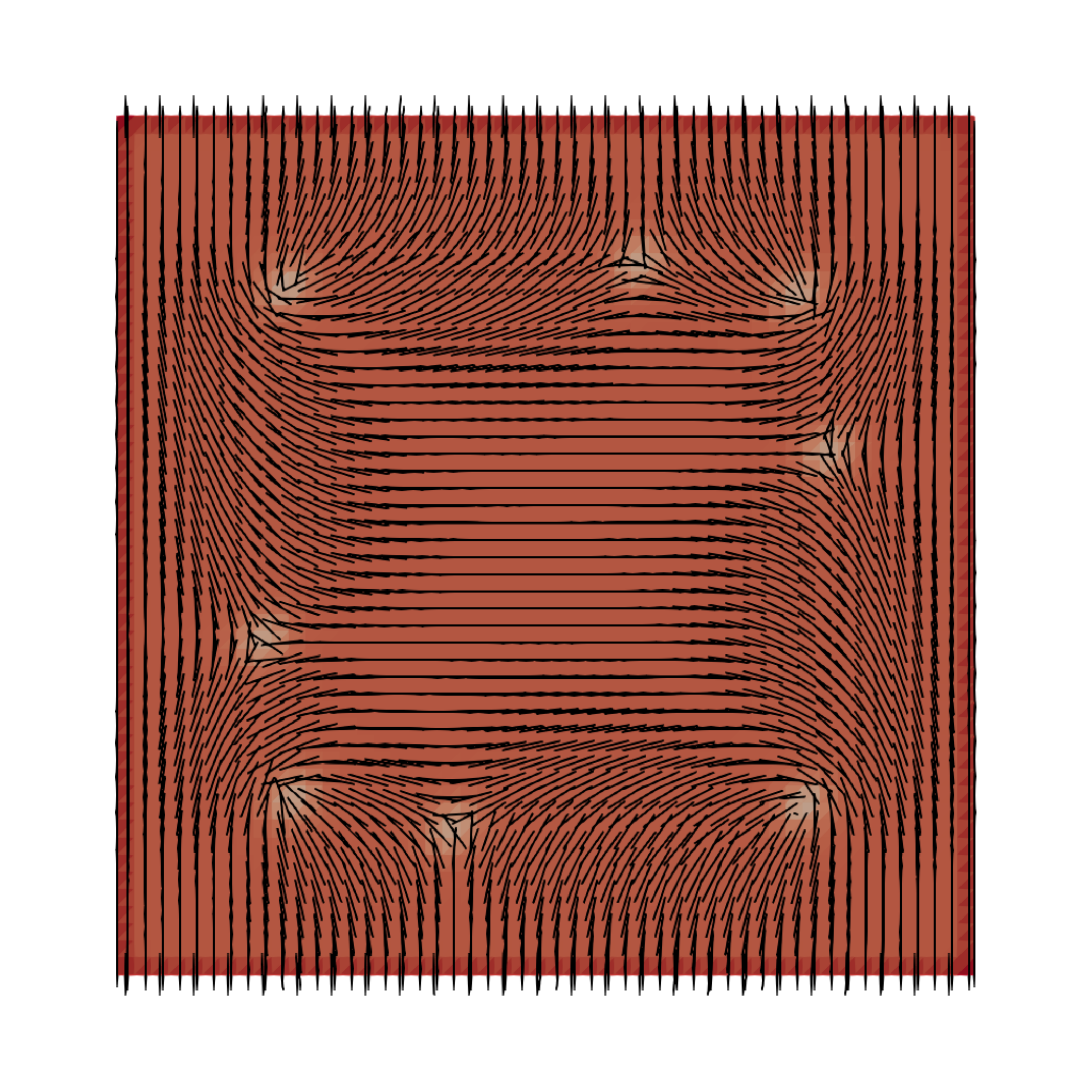}
\includegraphics[width = 0.23\textwidth]{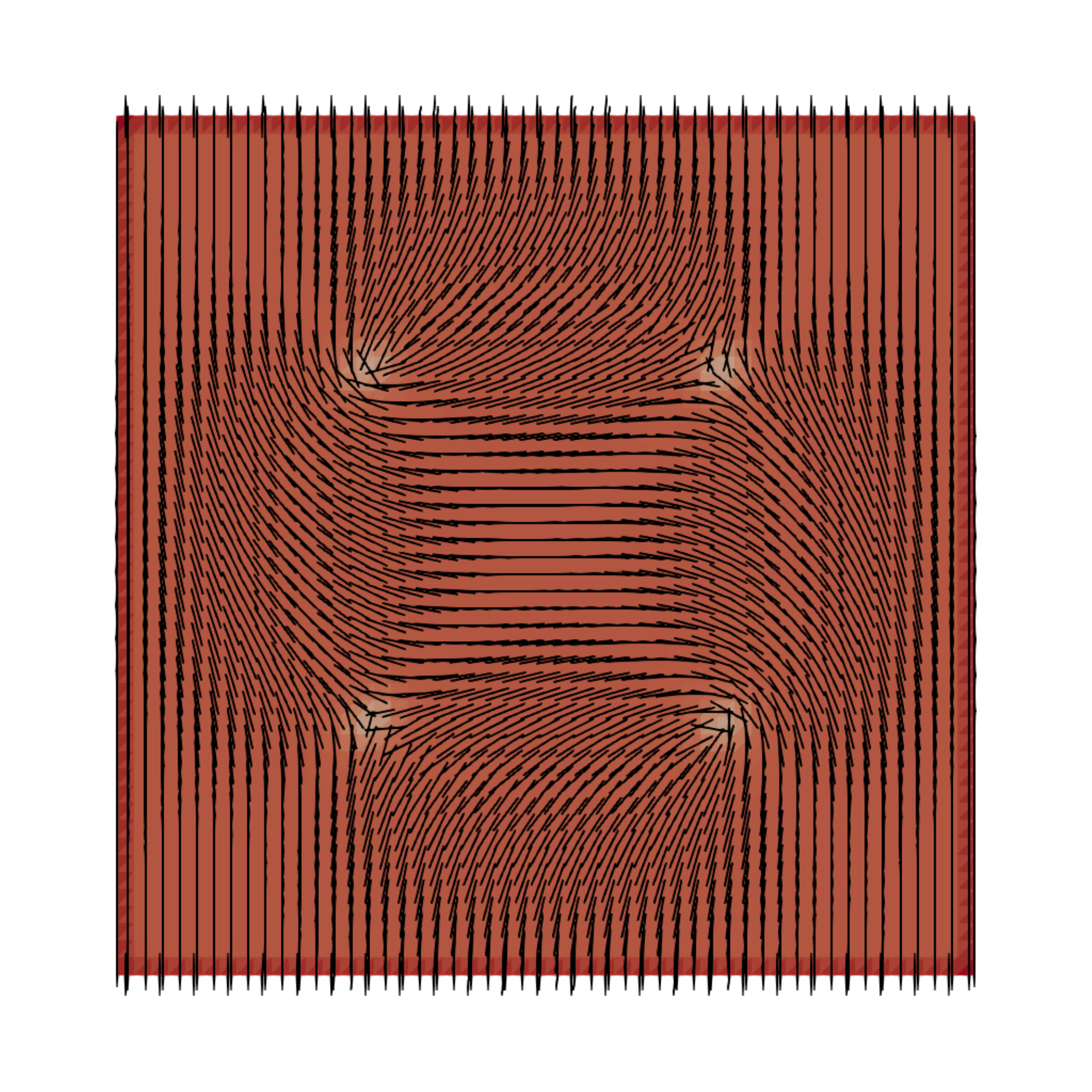}
\includegraphics[width = 0.23\textwidth]{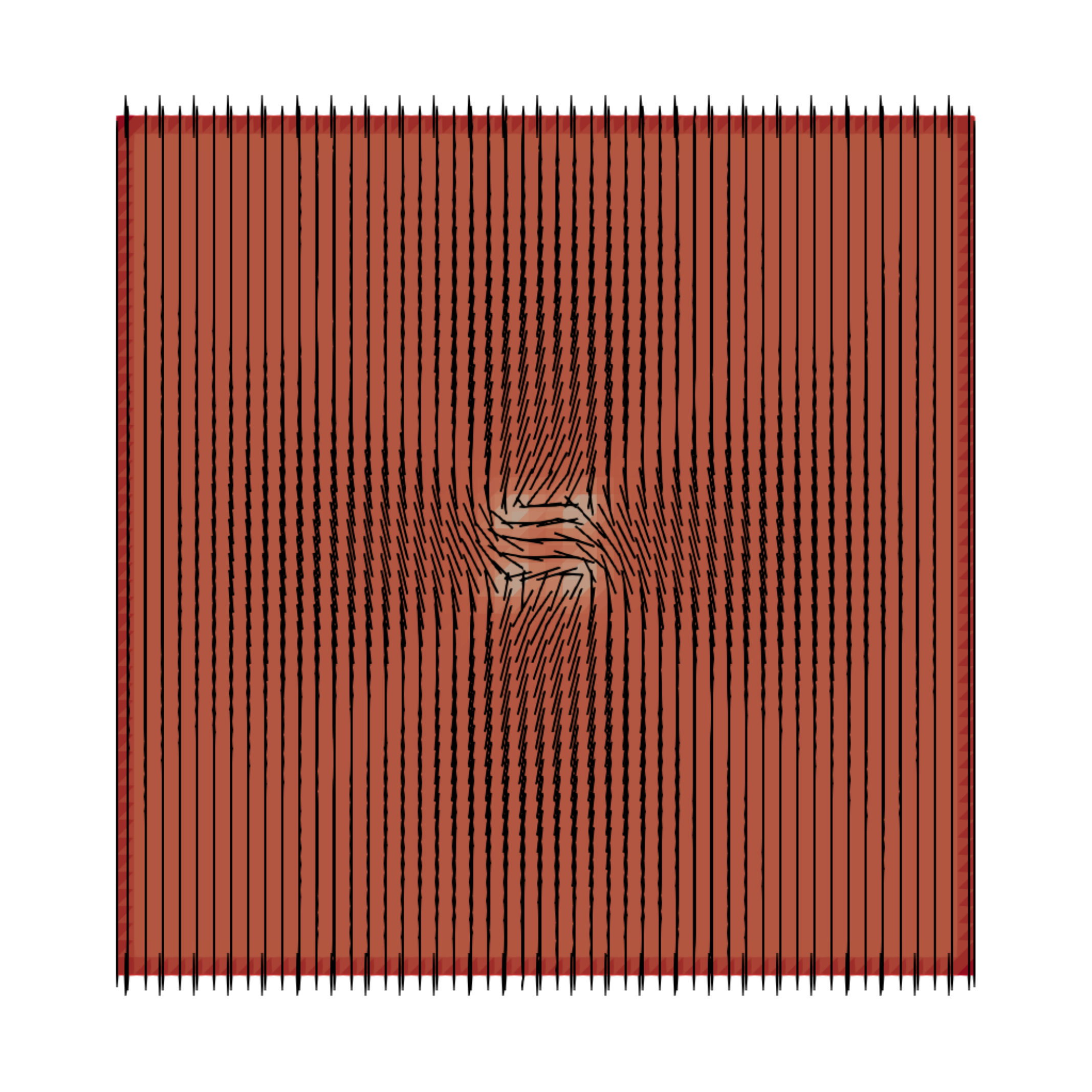}
\includegraphics[width = 0.23\textwidth]{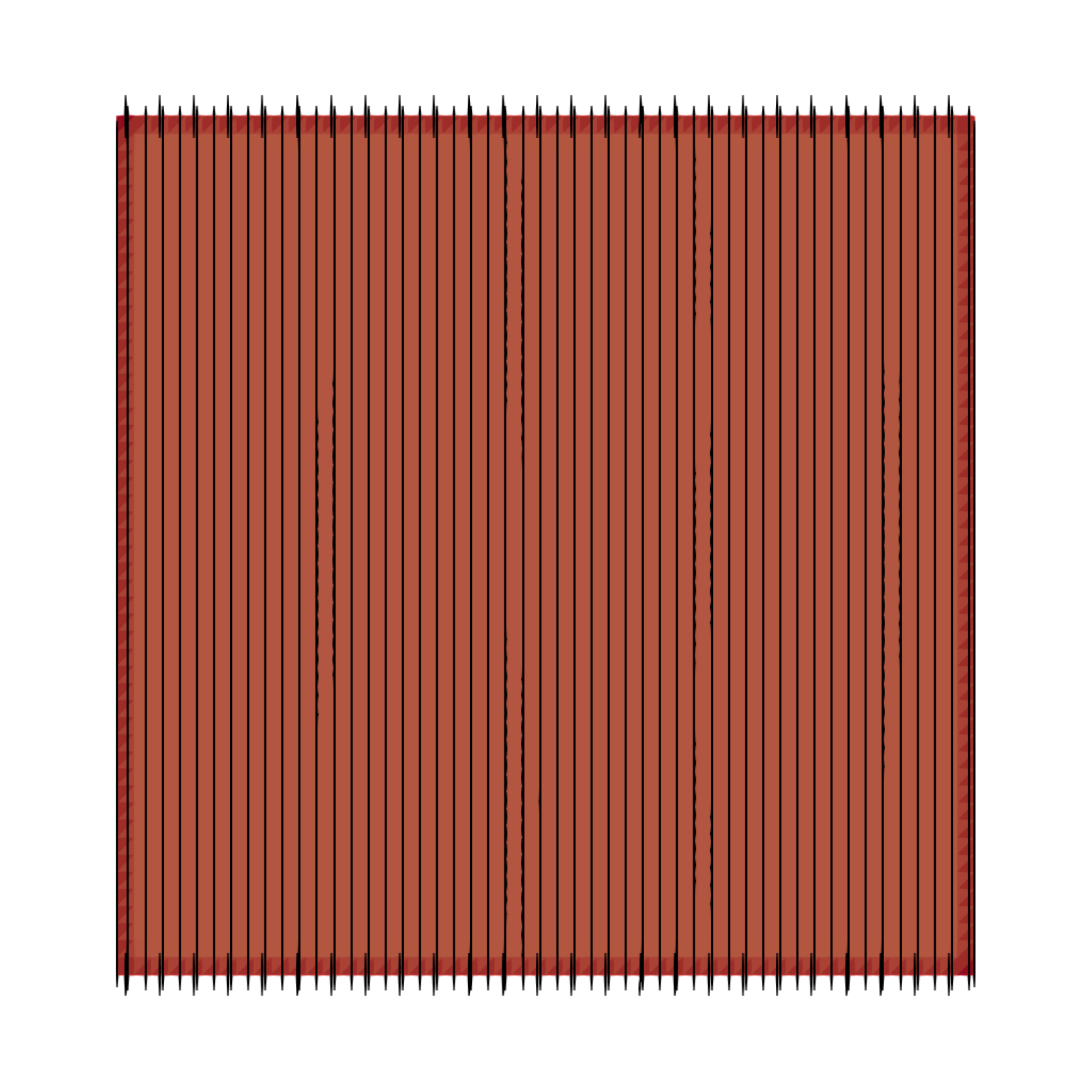} 
\caption{\label{fig:defectsUniform} Defect dynamics in 2D with Dirichlet boundary conditions at times $t=0.01,0.05,0.1,0.2,0.3,1.0,2.5,3.0$. The color represents the difference of the two largest eigenvalues of $\Q$ and indicates the alignment of the nematic with the dominant eigenvector shown as black lines.}
\end{center}
\end{figure}
Finally, we impose a different set of Dirichlet boundary conditions as follows
\beq\label{eq:dirichlet2}
\widetilde{\mathbf{d}_D}
\, = \, \dis
\left( \frac12(x - 2), \frac12 (y - 2) ,0\right)^T\, , \quad \quad
\widetilde{\Q_D}
\, = \,
\widetilde{\mathbf{d}_D} \widetilde{\mathbf{d}_D}^T - \frac{| \widetilde{\mathbf{d}_D} |^2 }3\I\, , \quad \quad
\Q \mid_{\partial\Om} 
\, = \,
\widetilde{\Q_D}\, .
\eeq
The time evolution of the solution for this simulation can be seen in \Cref{fig:defectsRadial}. Here, we also note the production of defects from the boundary conditions. However, in contrast to the last simulation, the last two remaining defects are stabilized by the boundary conditions. These defects move slowly relative the motion of defects earlier in the simulation. \\
{In \Cref{fig:dynamicsCompareBC} we compare the dynamics of the solutions obtained using the three different boundary conditions. We see that the final configuration of the case of Neumann boundary conditions and $\Q_D$ is a uniform configuration of the director vectors, and thus the energy of these solutions are similar at the final time. In the case of the boundary condition $\Q_{\tilde{D}}$ we see that the final configuration contains two defects, hence the energy of this solution is higher than the previous two. }

\begin{figure}[h]
\begin{center}
\includegraphics[width = 0.23\textwidth]{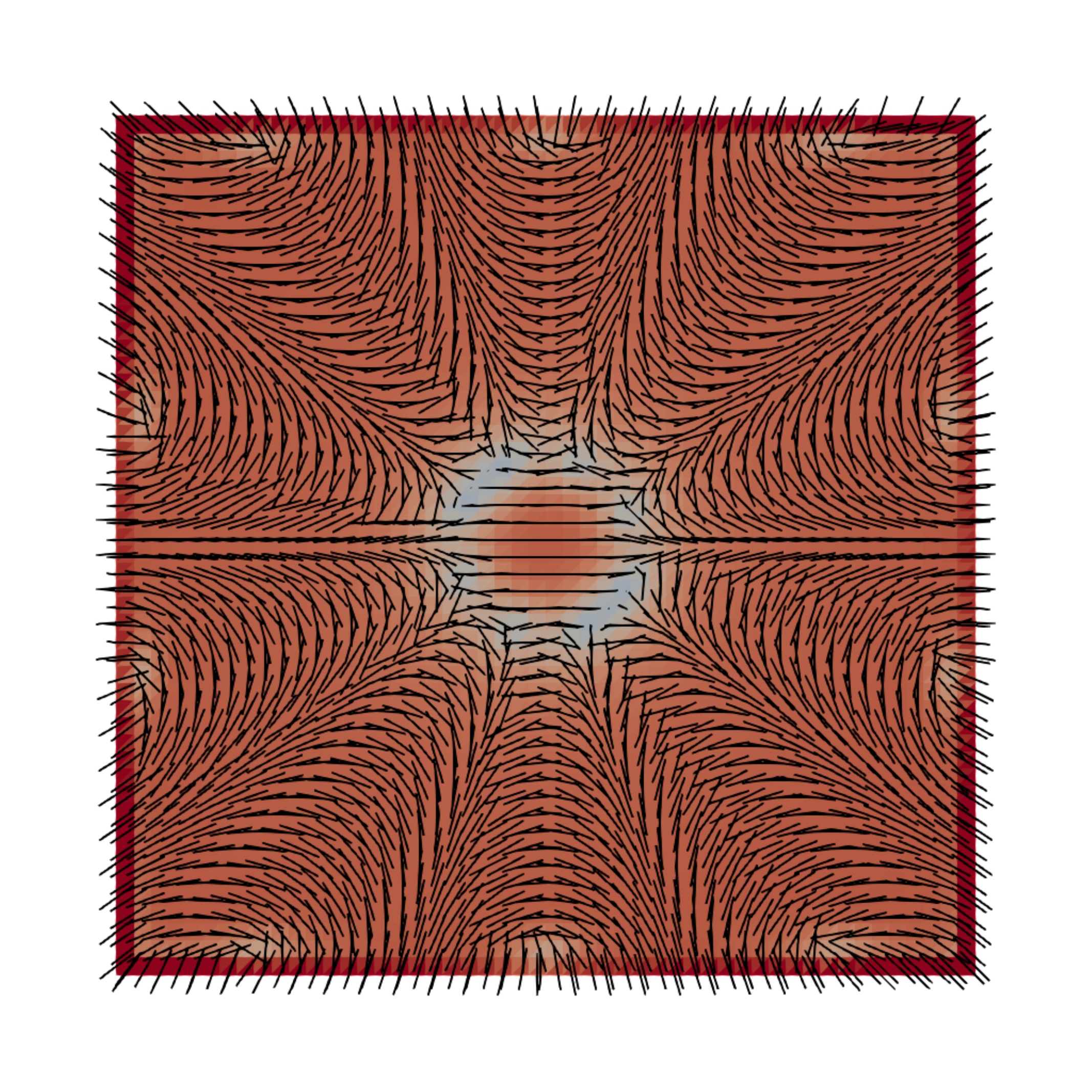}
\includegraphics[width = 0.23\textwidth]{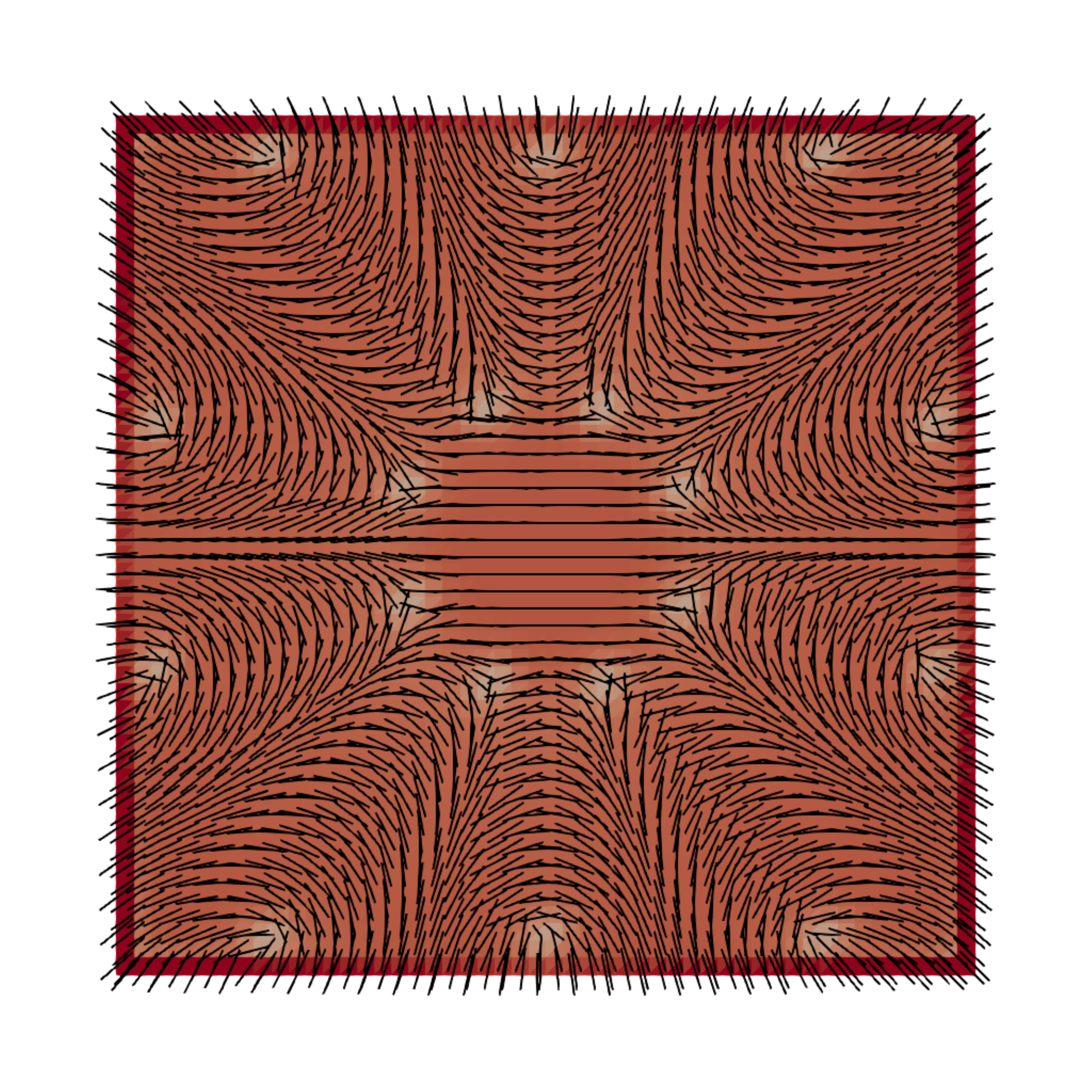}
\includegraphics[width = 0.23\textwidth]{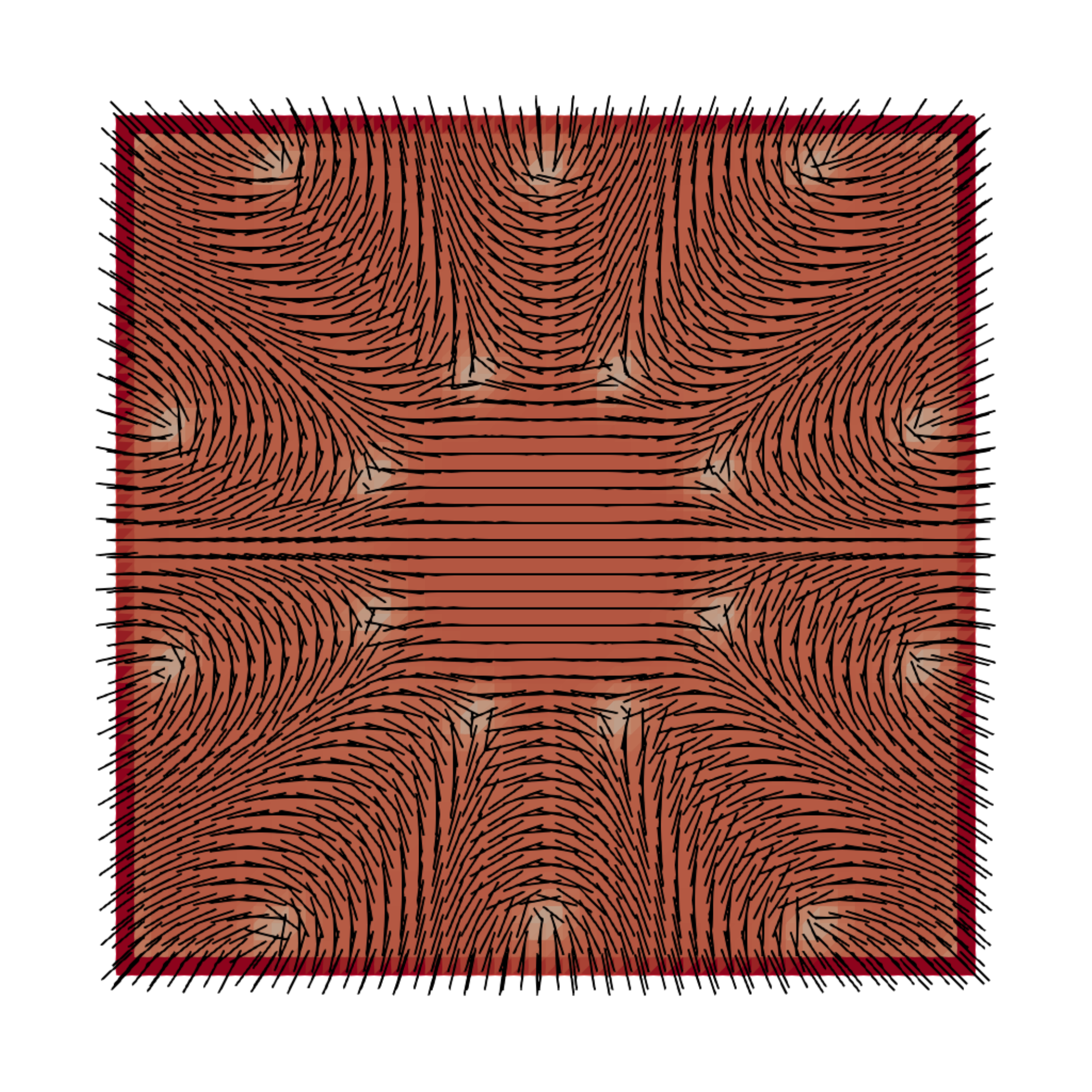} 
\includegraphics[width = 0.23\textwidth]{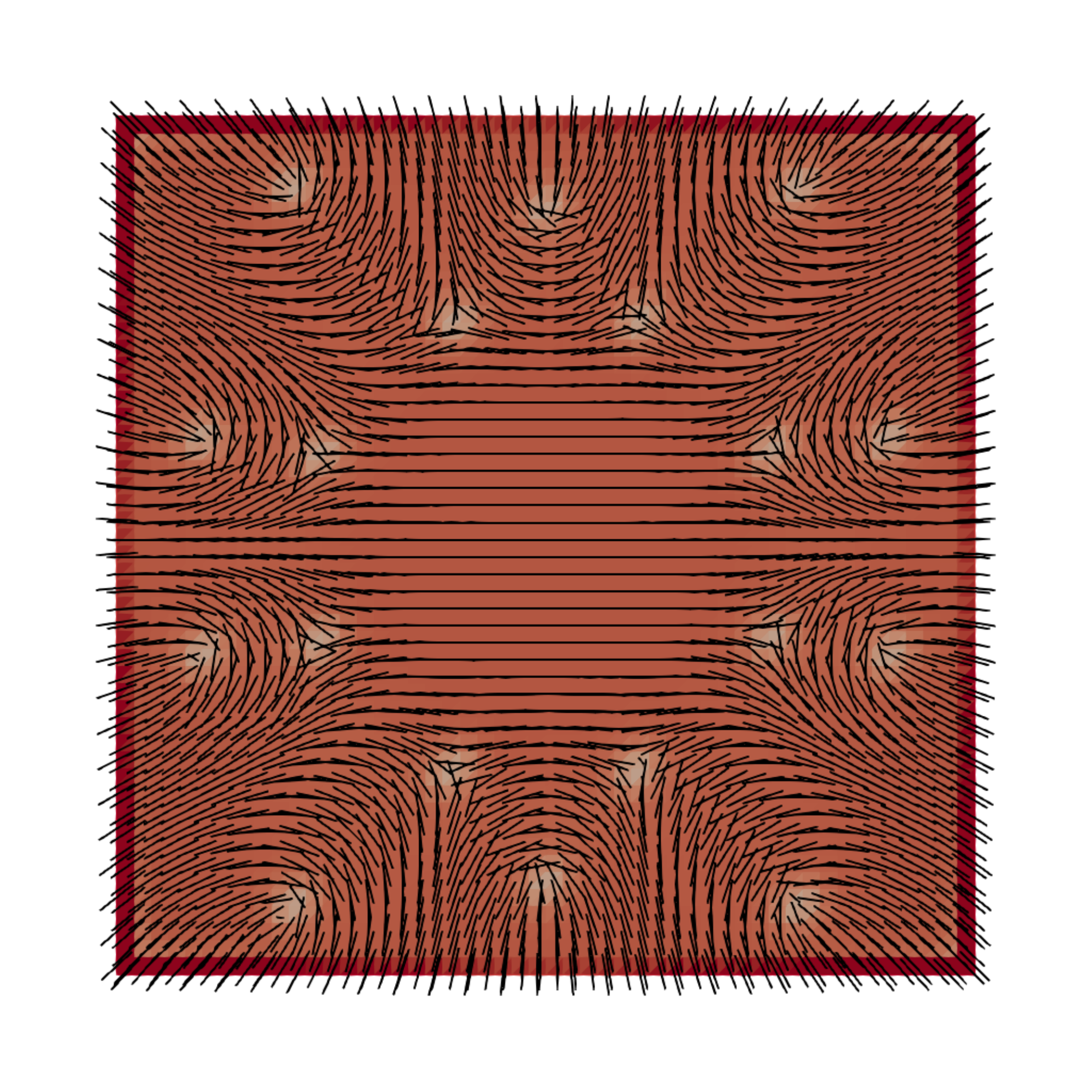}
\includegraphics[width = 0.23\textwidth]{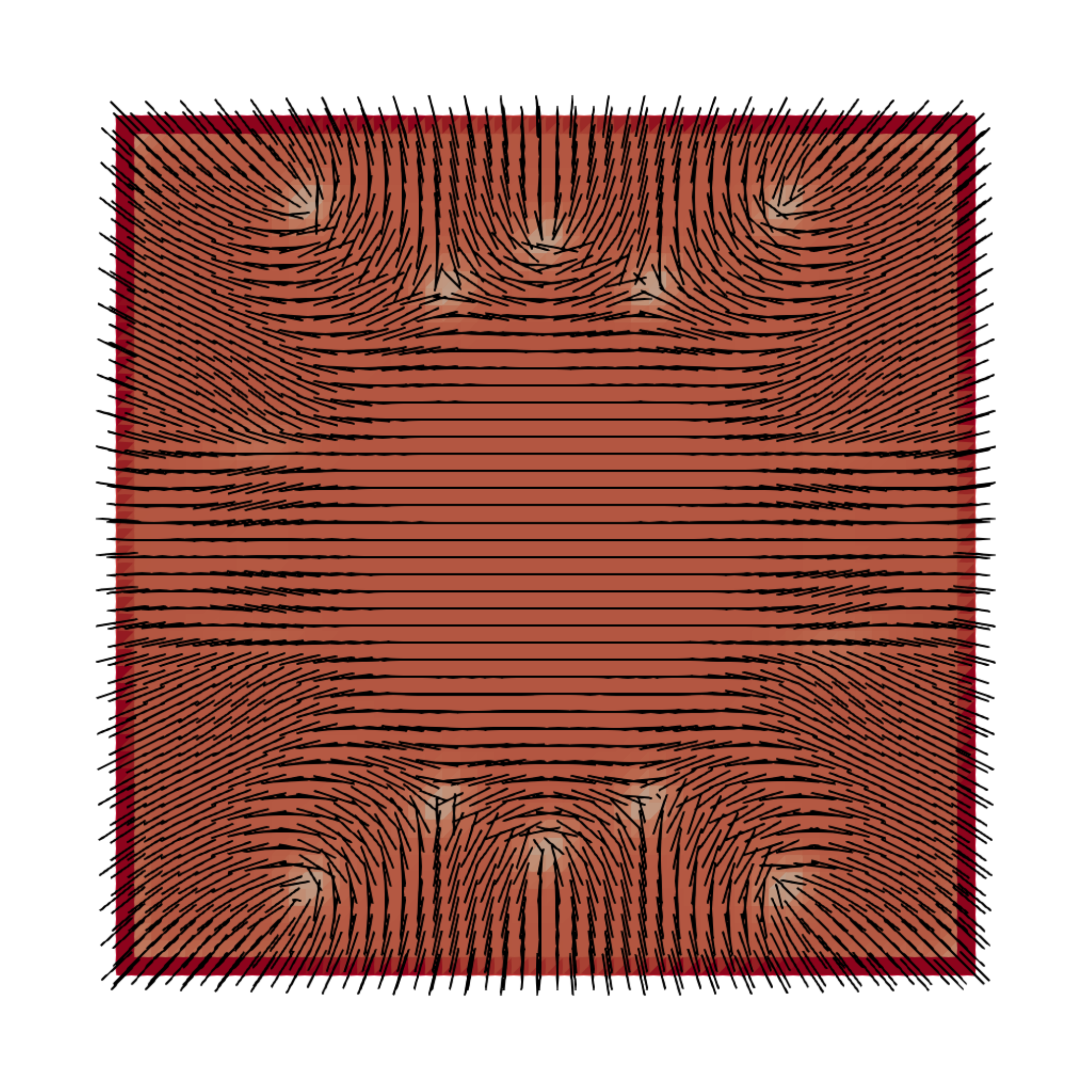}
\includegraphics[width = 0.23\textwidth]{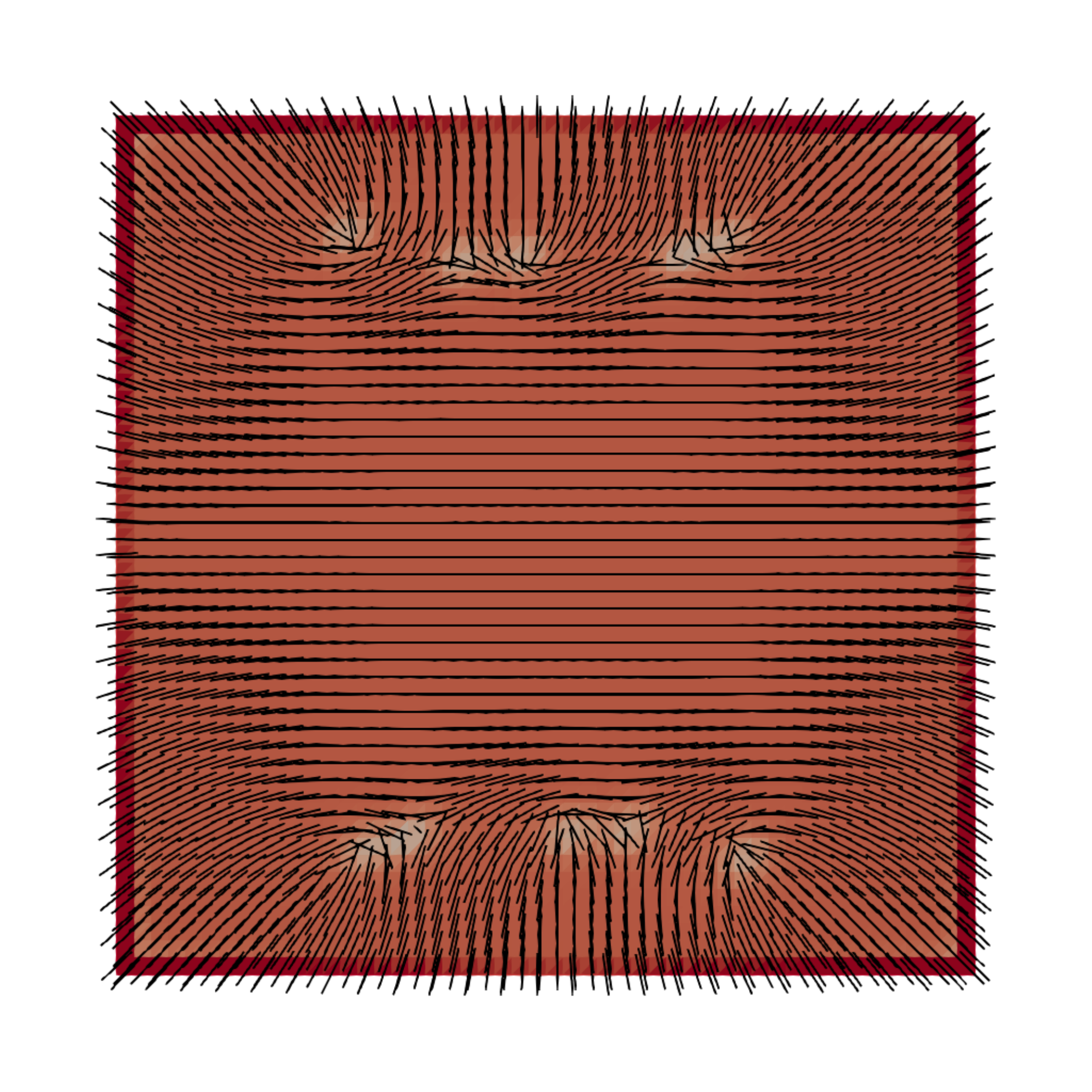}
\includegraphics[width = 0.23\textwidth]{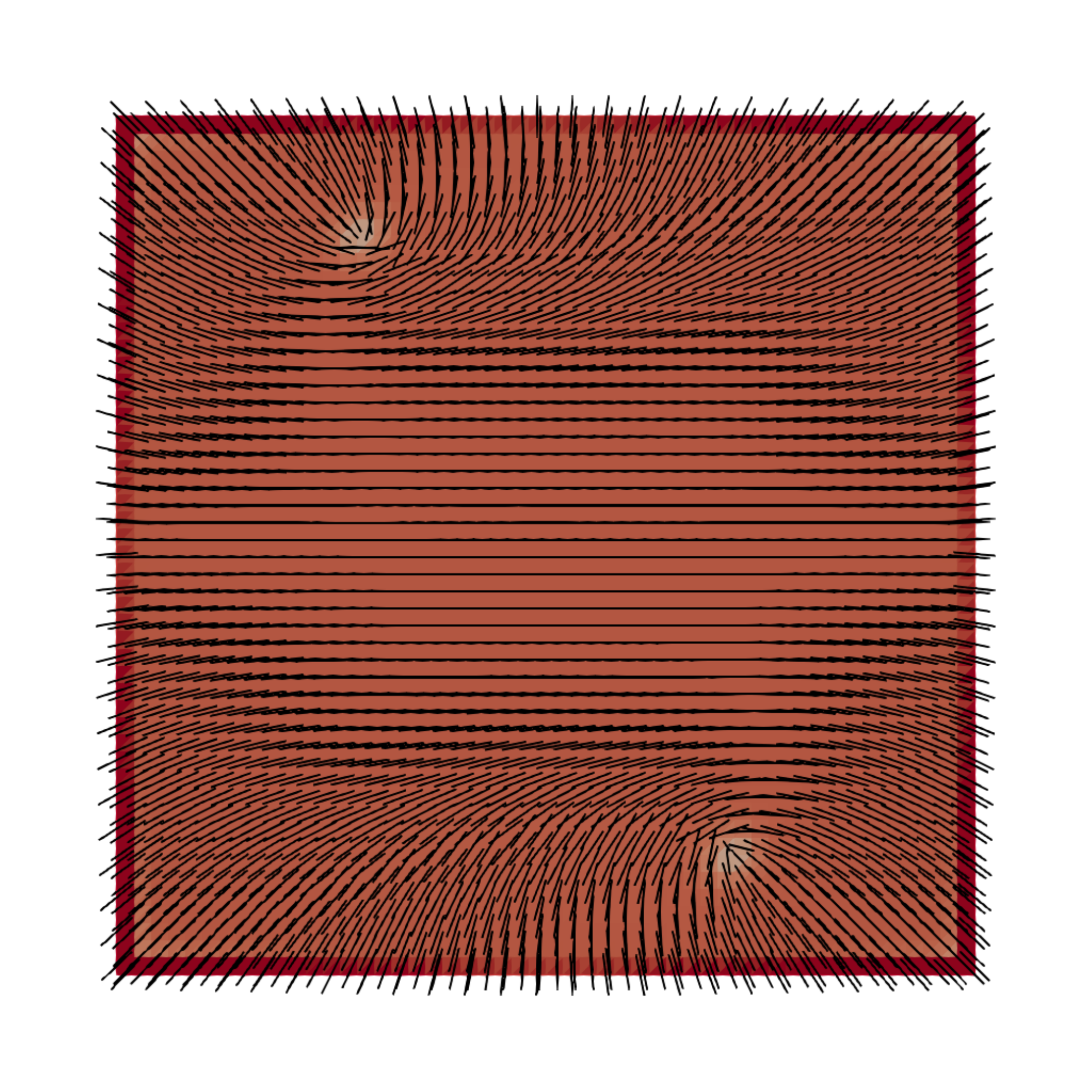}
\includegraphics[width = 0.23\textwidth]{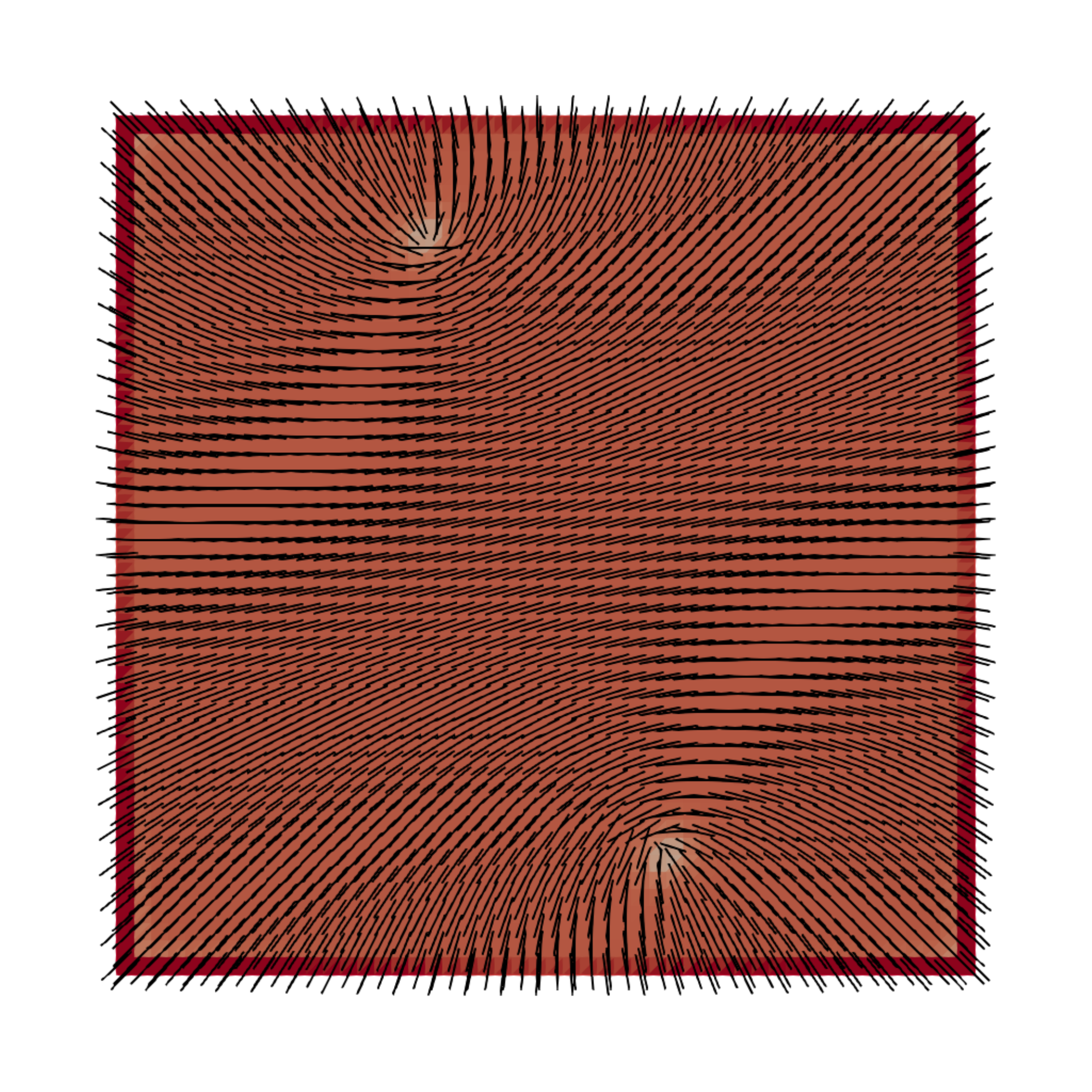} 
\caption{\label{fig:defectsRadial} Defect dynamics in 2D with Dirichlet boundary conditions at times $t=0.01,0.03,0.05,0.1,0.15,0.3,0.35,1.0,3.0$. The color represents the difference of the two largest eigenvalues of $\Q$ and indicates the alignment of the nematic with the dominant eigenvector shown as black lines.}
\end{center}
\end{figure}

\begin{figure}[h]
\begin{center}
\includegraphics[width = 0.32\textwidth]{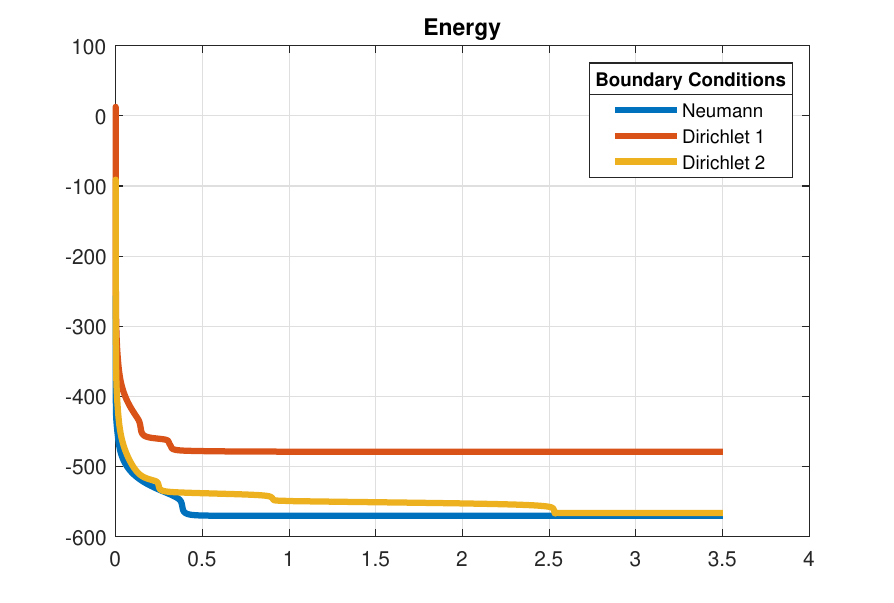}
\includegraphics[width = 0.32\textwidth]{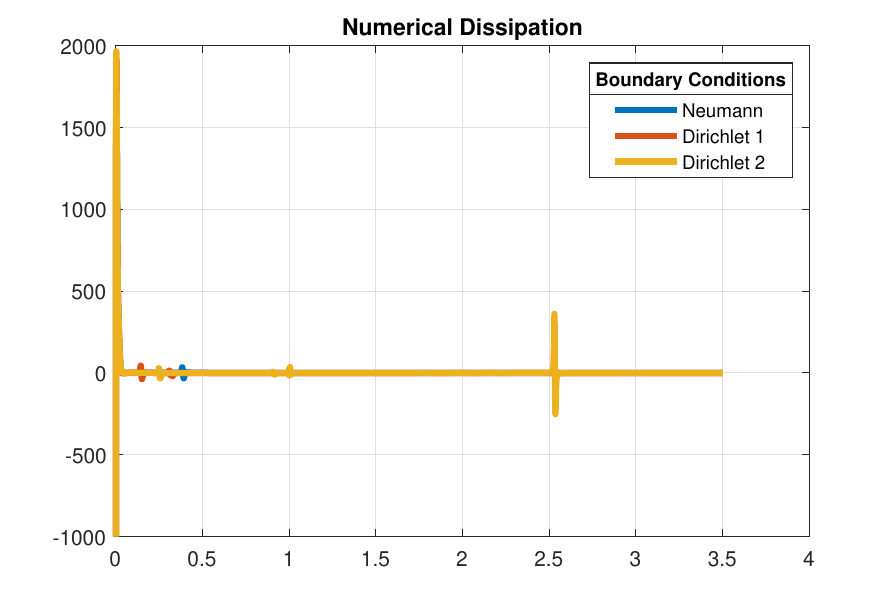}
\includegraphics[width = 0.32\textwidth]{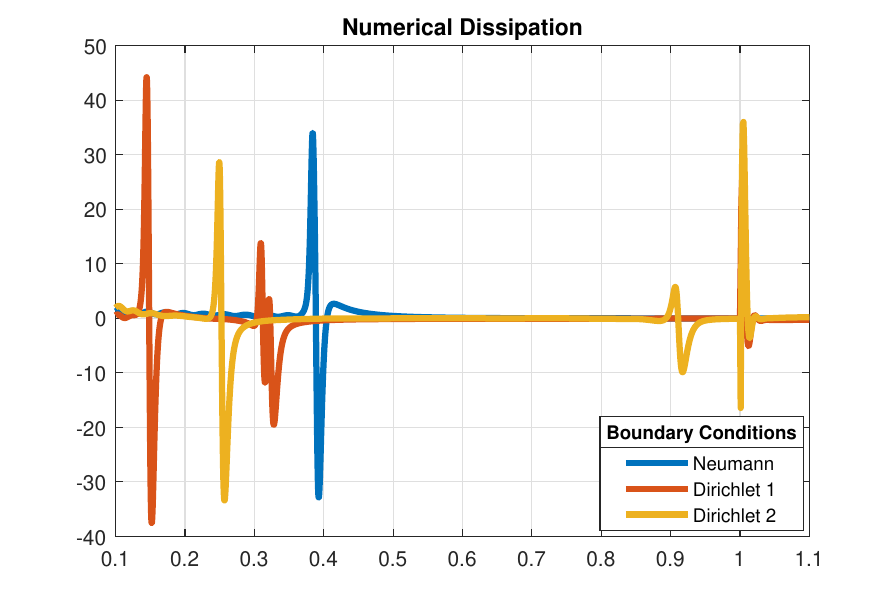}
\caption{\label{fig:dynamicsCompareBC} Comparing energy and numerical dissipation as a result of each case of boundary conditions. Here, Dirichlet 1 corresponds to the boundary conditions given by \eqref{eq:dirichlet1} and Dirichlet 2 corresponds to \eqref{eq:dirichlet2}. \textit{Right} shows a view of the curve shown in the \textit{middle} restricted to the interval $[0.1, 1.1]$.}
\end{center}
\end{figure}

\subsection{Dynamics in 3D}
The purpose of this example is to show dynamics in three dimensions to provide an example of the efficiency of scheme OD1D. We choose as computational domain $\Om = [0,2]^3$ which is discretized into a $50\times50\times50$ triangular mesh. The final time is $T= 0.2$, and the time step is $\dt = 10^{-4}$. In this case $\eps = 1.0$.\\ 
We use Neumann boundary conditions $\partial_\mathbf{n}\Q = \bm{0}$ and an initial state chosen so that each component of the vector $\mathbf{d}_0$ is randomly distributed uniformly in $[-1,1]$ with $|\mathbf{d}_0|=1$, and 
\beq
\Q_0 \, = \, \mathbf{d}_0 \mathbf{d}_0^T - \frac{| \mathbf{d}_0 |^2}3 \I \, .
\eeq
\Crefrange{fig:random3D}{fig:energy3D} show the result of this simulation using scheme OD1D. Initially, a random orientation of the nematic throughout the domain, and as time progresses the system moves towards a uniform orientation. In \Cref{fig:random3D} we note the appearance, movement, and annihilation of defects on the surface of the domain. \Cref{fig:energy3D} shows the time evolution of the energy, 
where we can observe that the energy is decreasing at all times.
\begin{figure}[h]
\begin{center}
\includegraphics[width = 0.23\textwidth]{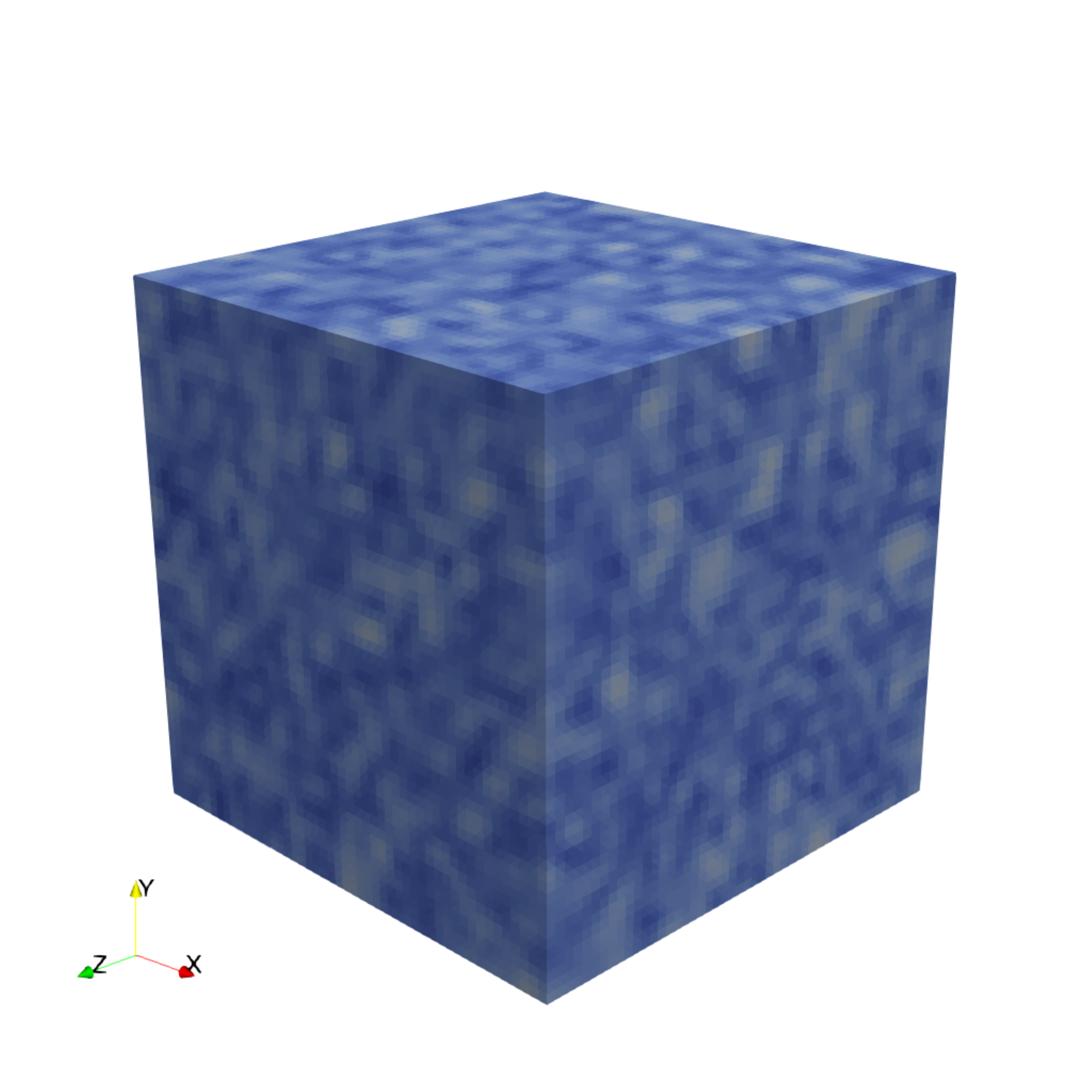}
\includegraphics[width = 0.23\textwidth]{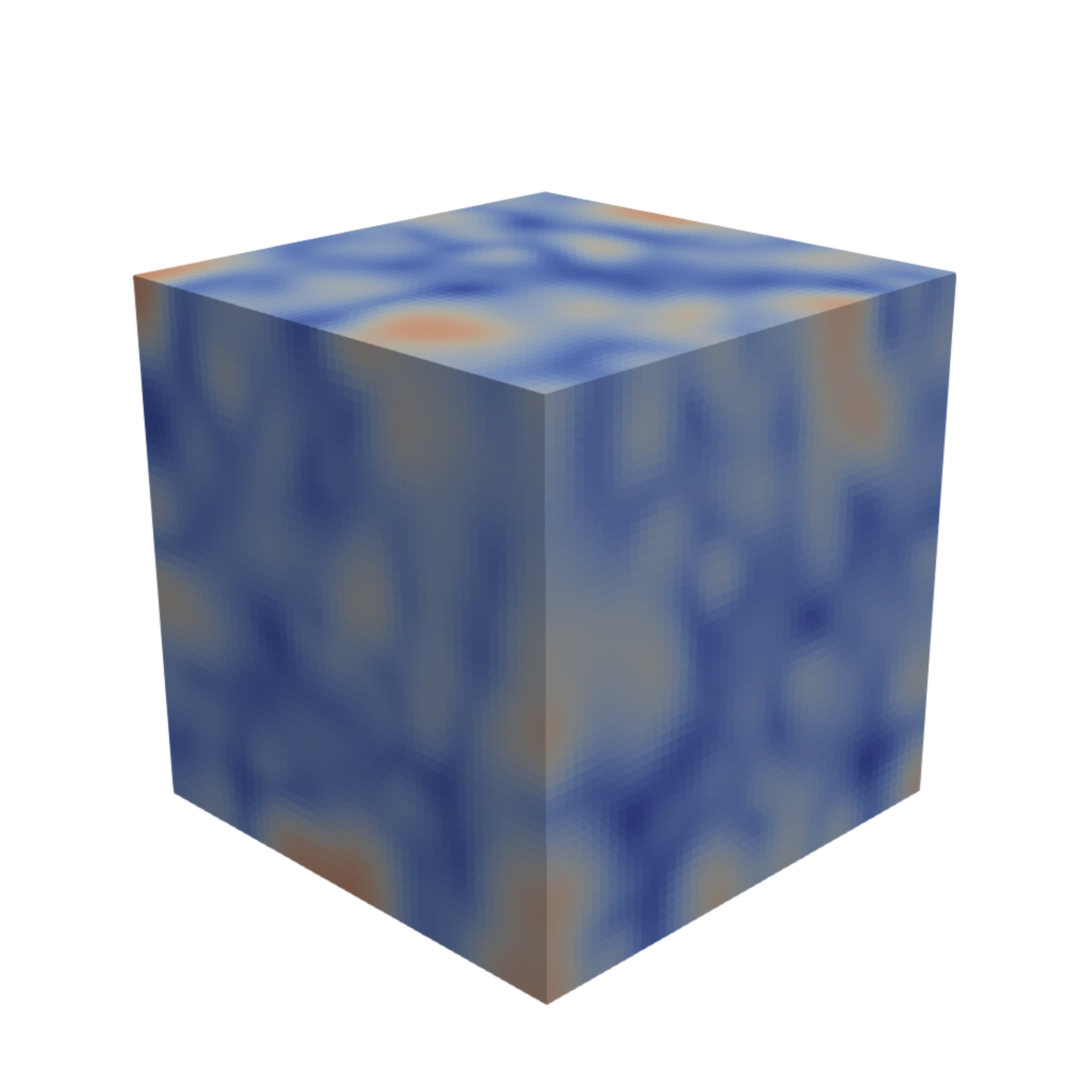}
\includegraphics[width = 0.23\textwidth]{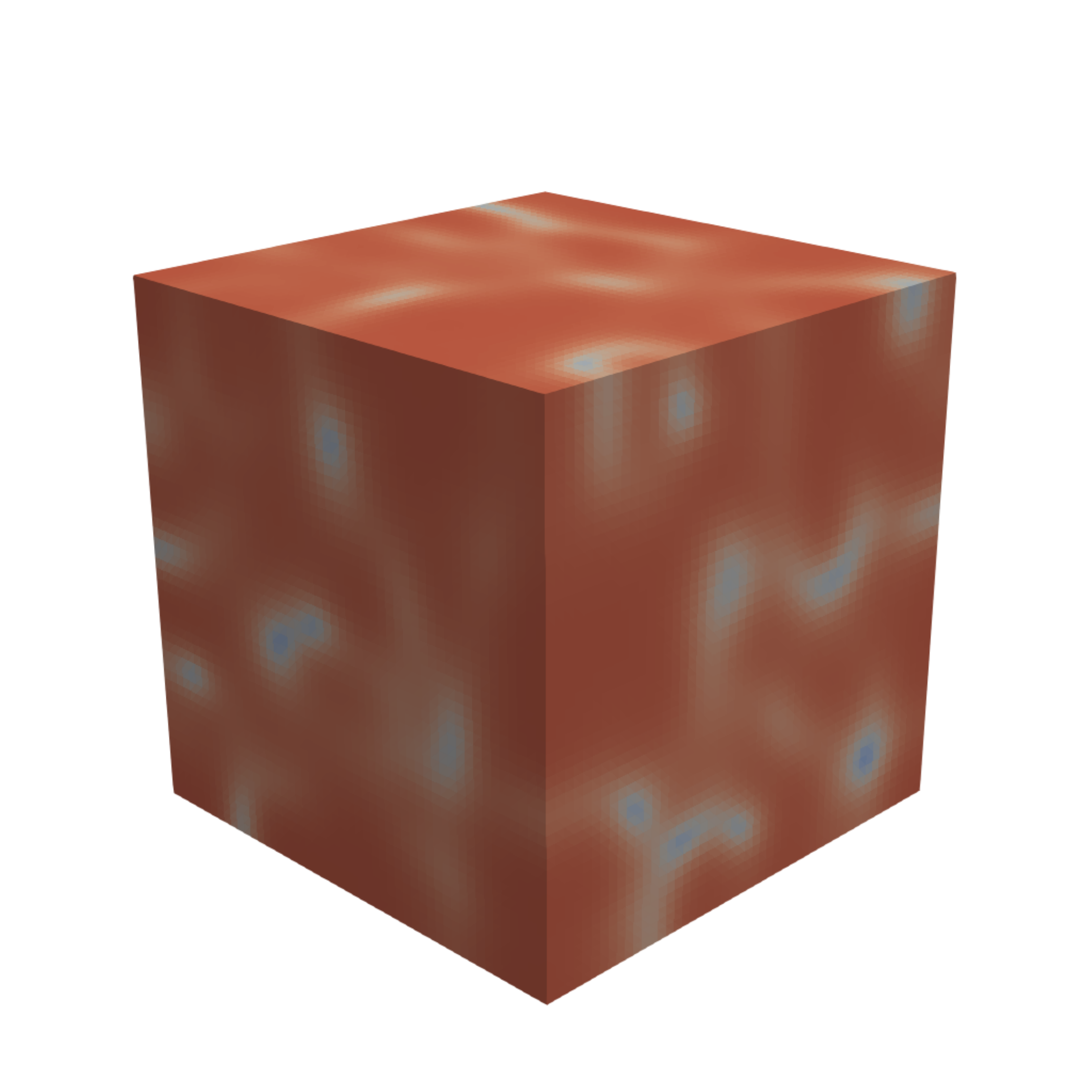}
\includegraphics[width = 0.23\textwidth]{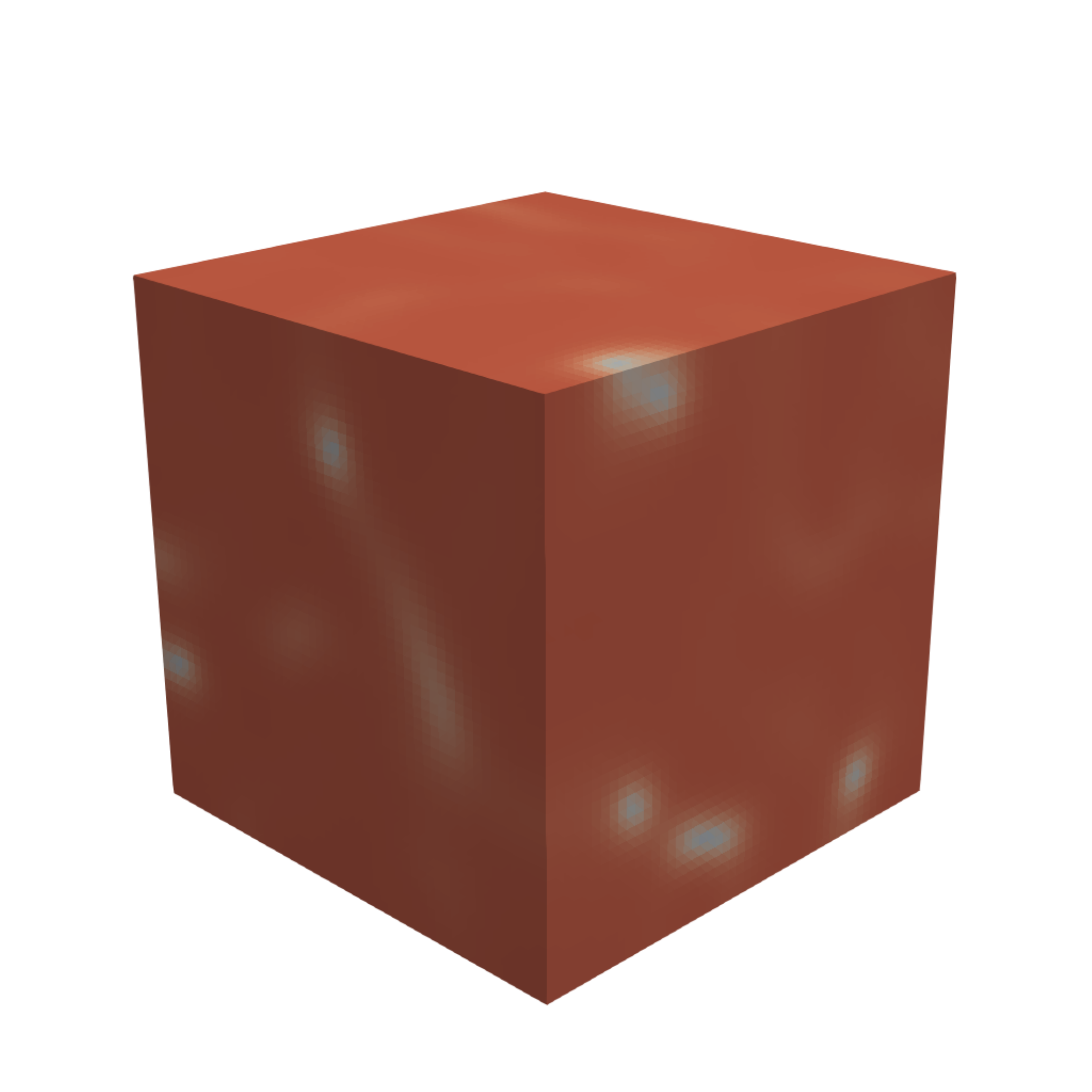}
\includegraphics[width = 0.23\textwidth]{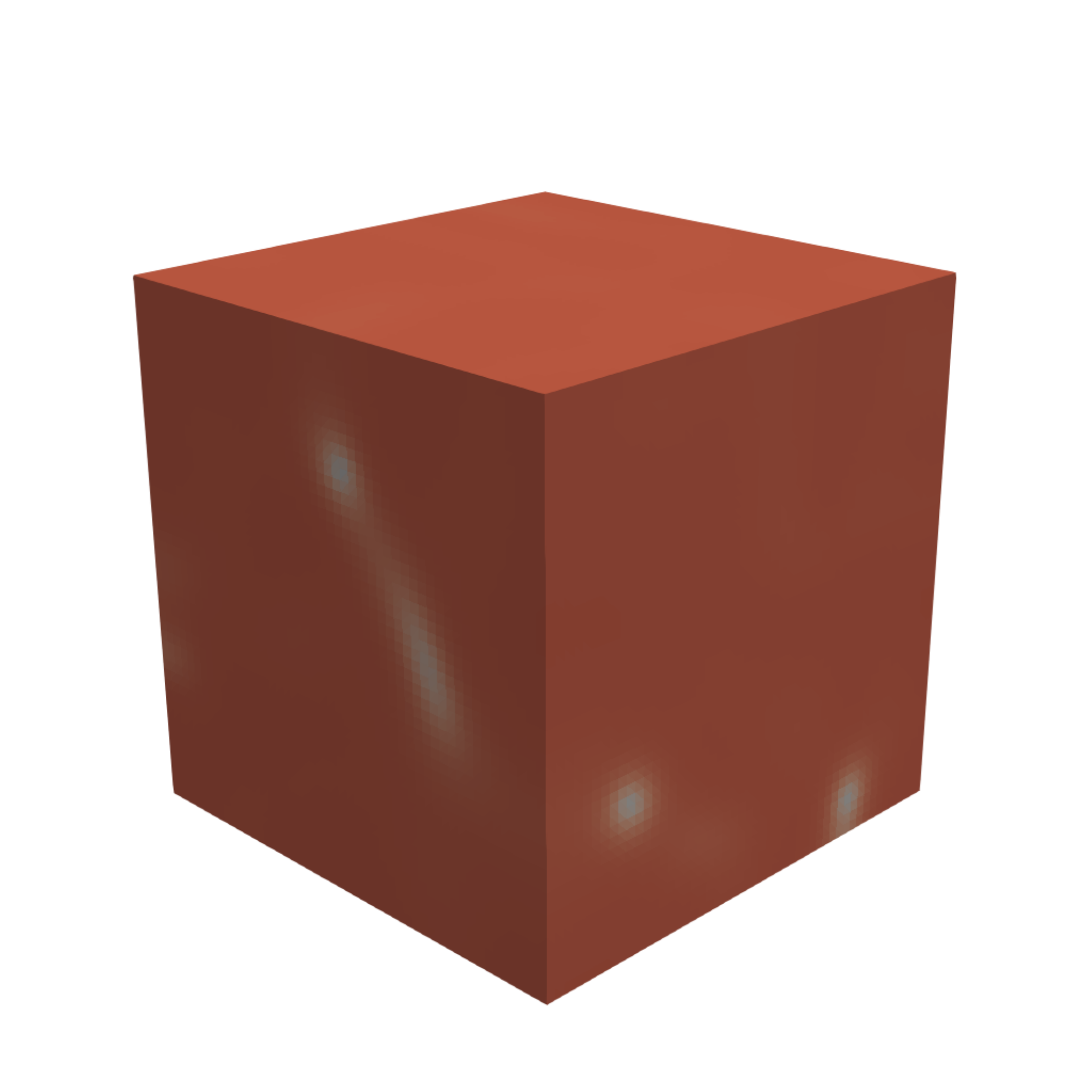}
\includegraphics[width = 0.23\textwidth]{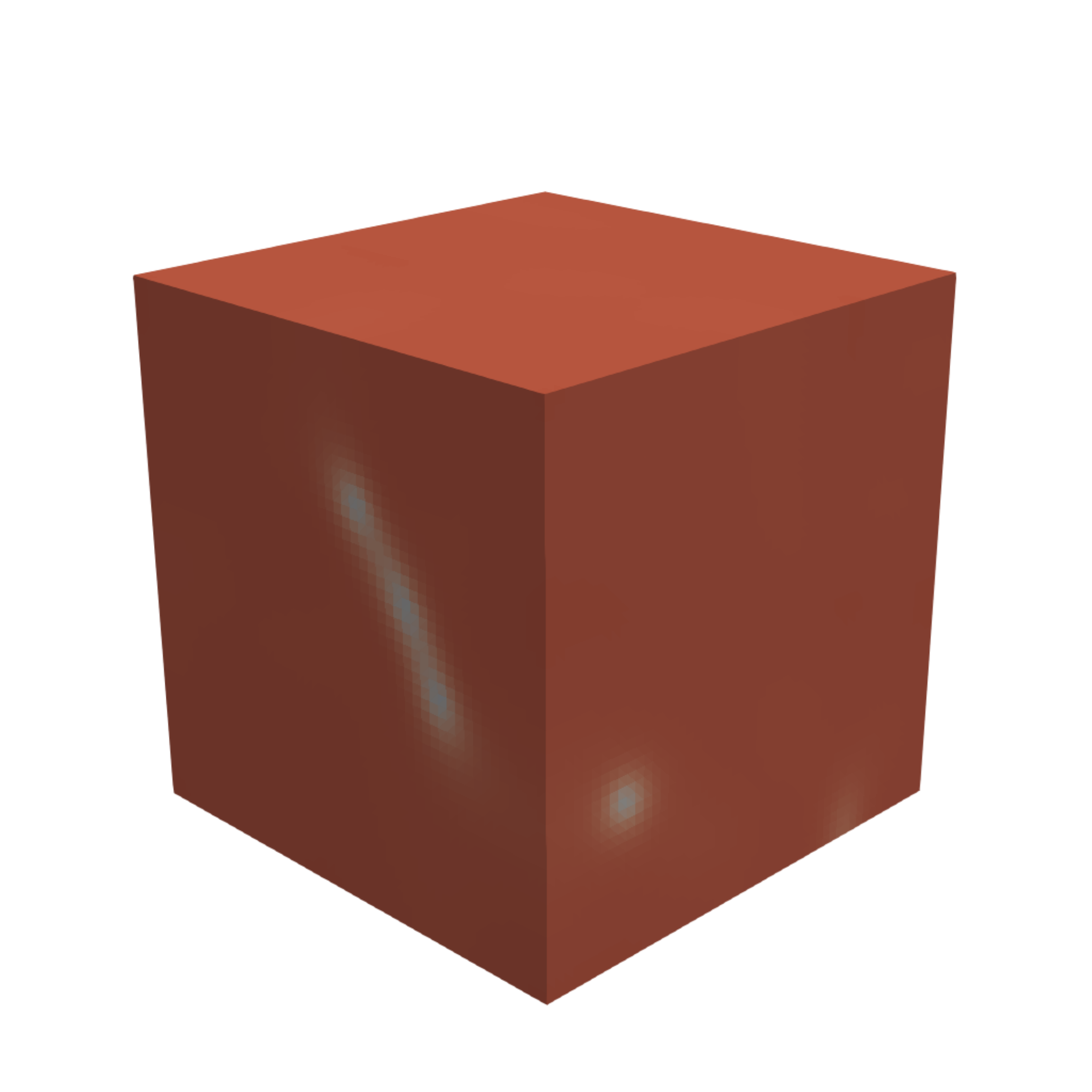}
\includegraphics[width = 0.23\textwidth]{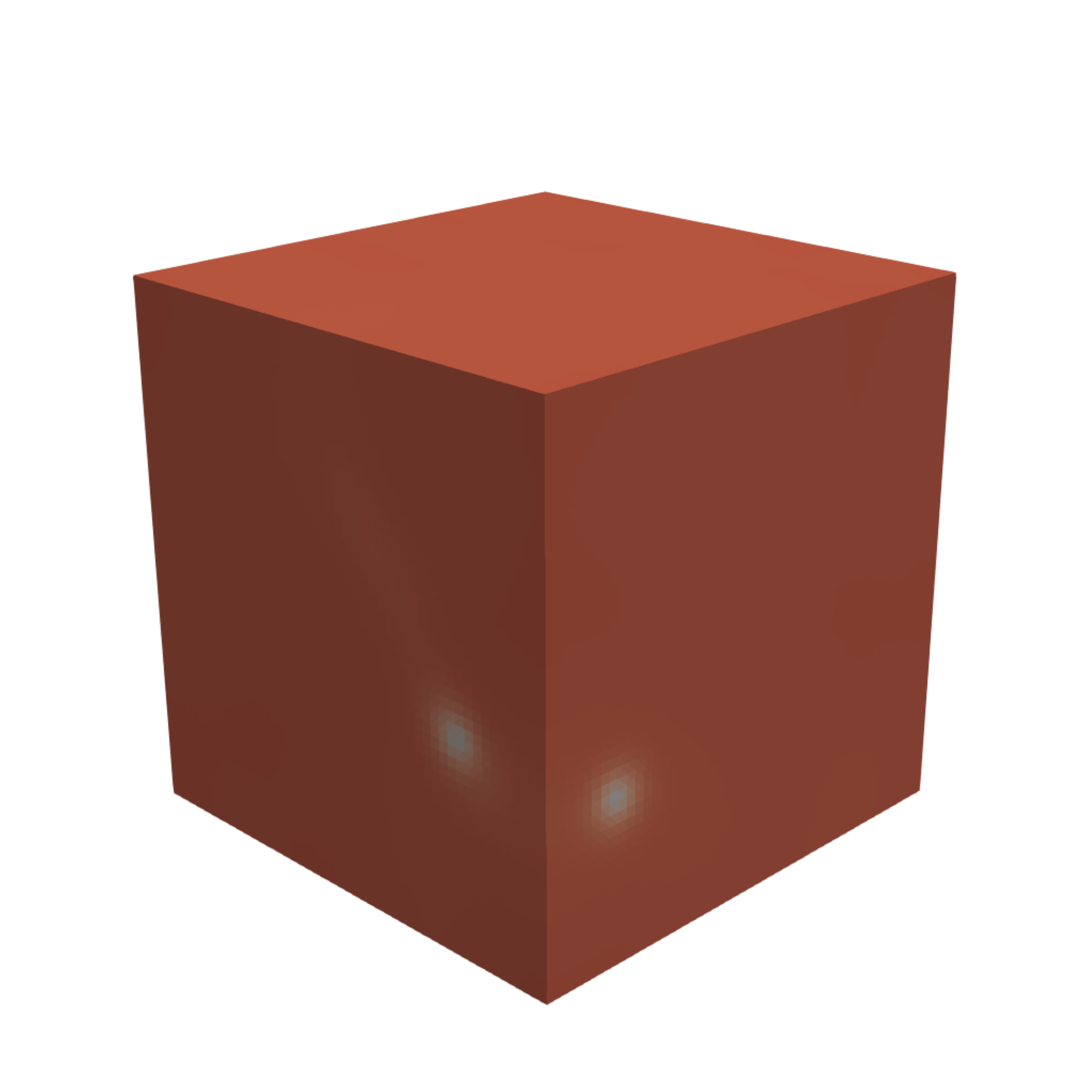}
\includegraphics[width = 0.23\textwidth]{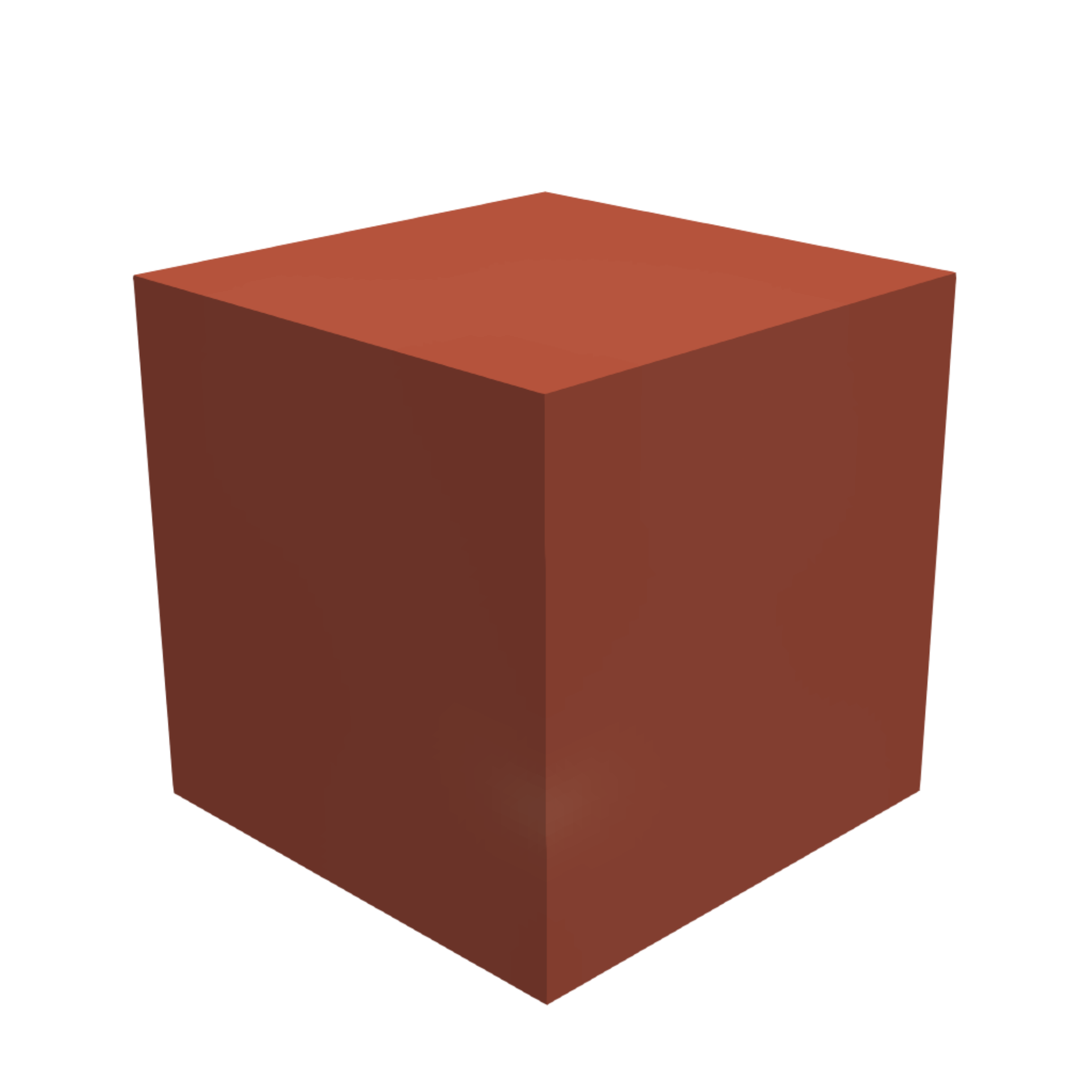}
\caption{\label{fig:random3D} Dynamics in 3D with random initial condition at $t=0.0001,0.01,0.02,0.03,0.04,0.05,0.065,0.1,$ and $0.125$. The color represents the difference of the two largest eigenvalues of $\Q$ and indicates the alignment of the nematic with the dominant eigenvector shown as black lines in \Cref{fig:randomvector3D}.}
\end{center}
\end{figure}

\begin{figure}
\begin{center}
\includegraphics[width = 0.23\textwidth]{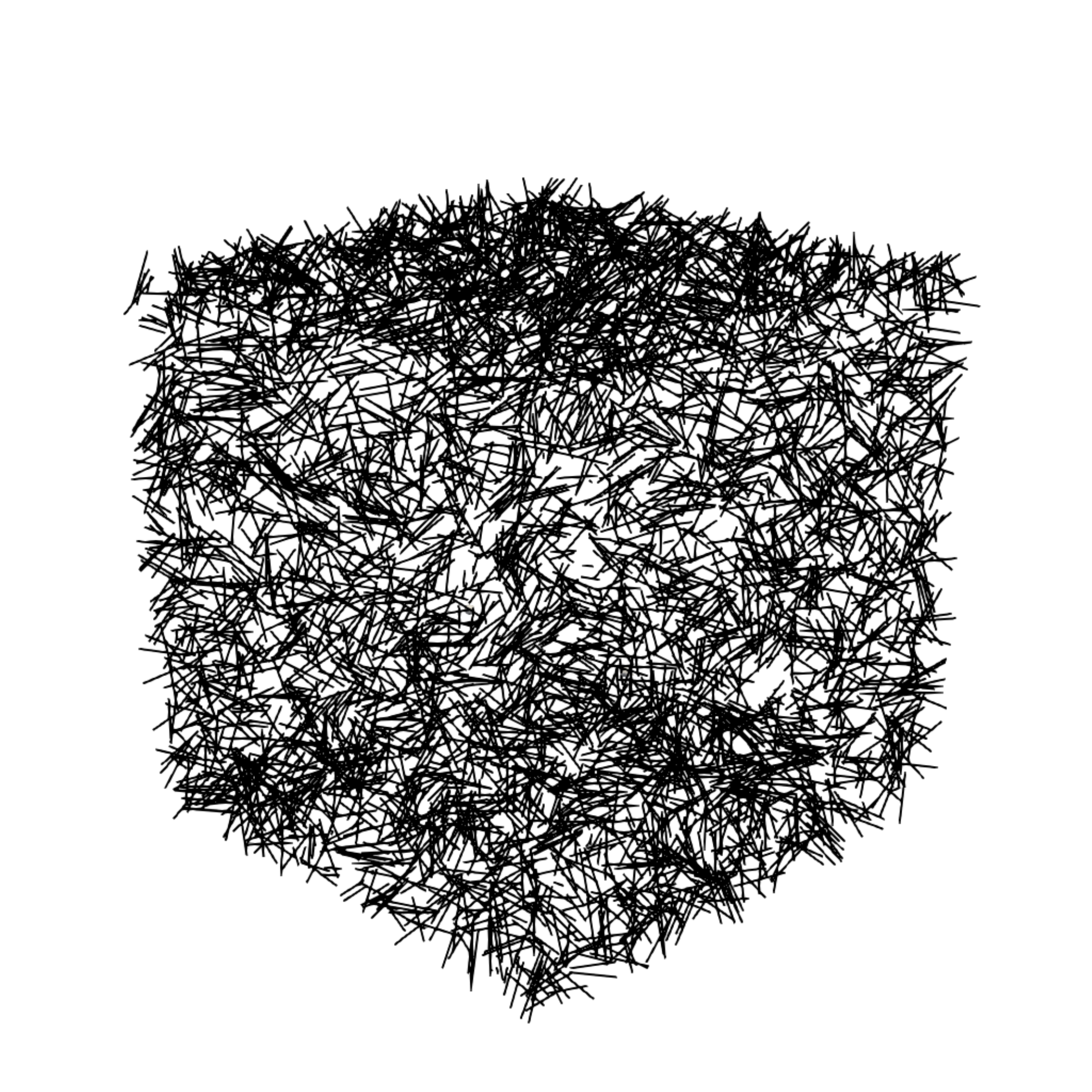}
\includegraphics[width = 0.23\textwidth]{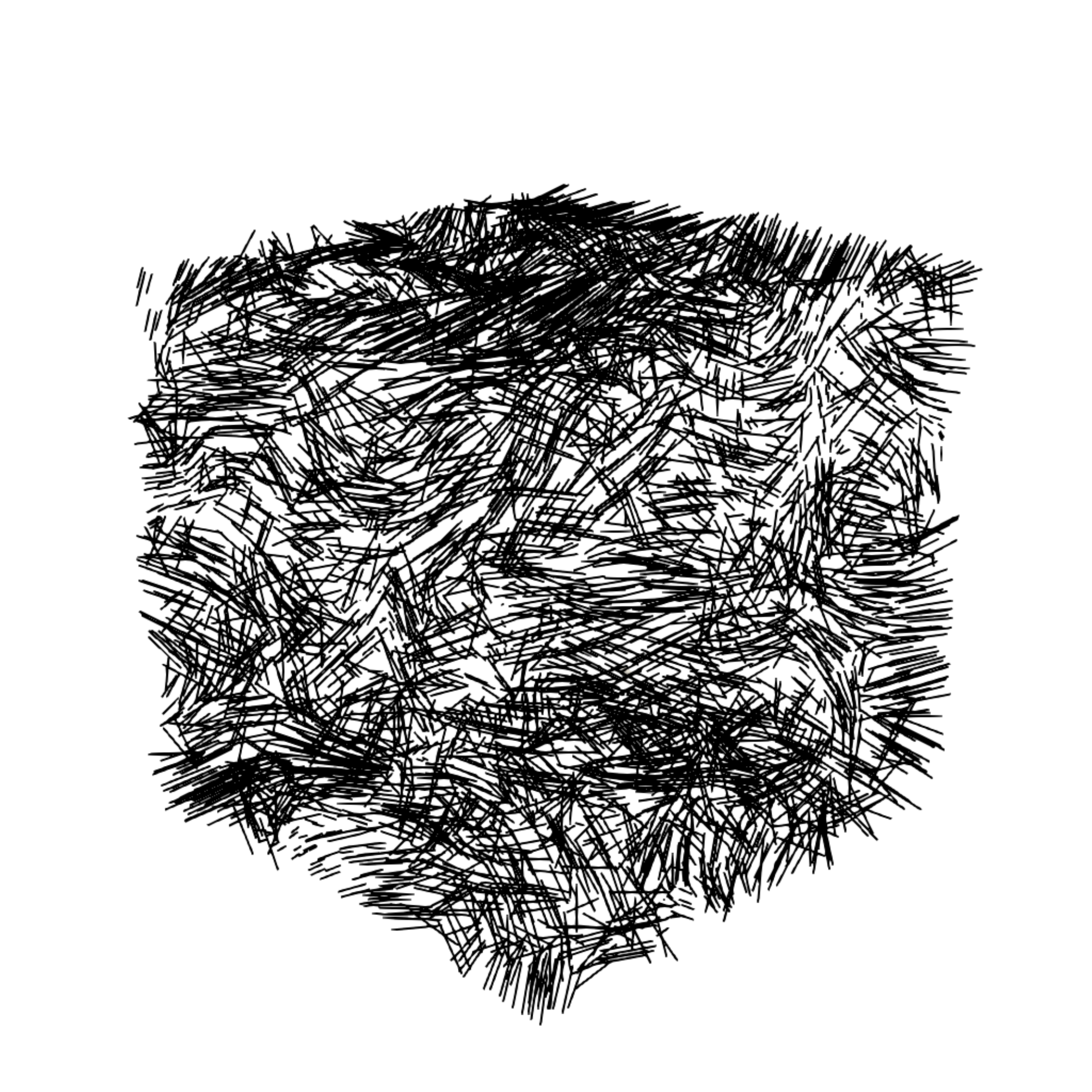}
\includegraphics[width = 0.23\textwidth]{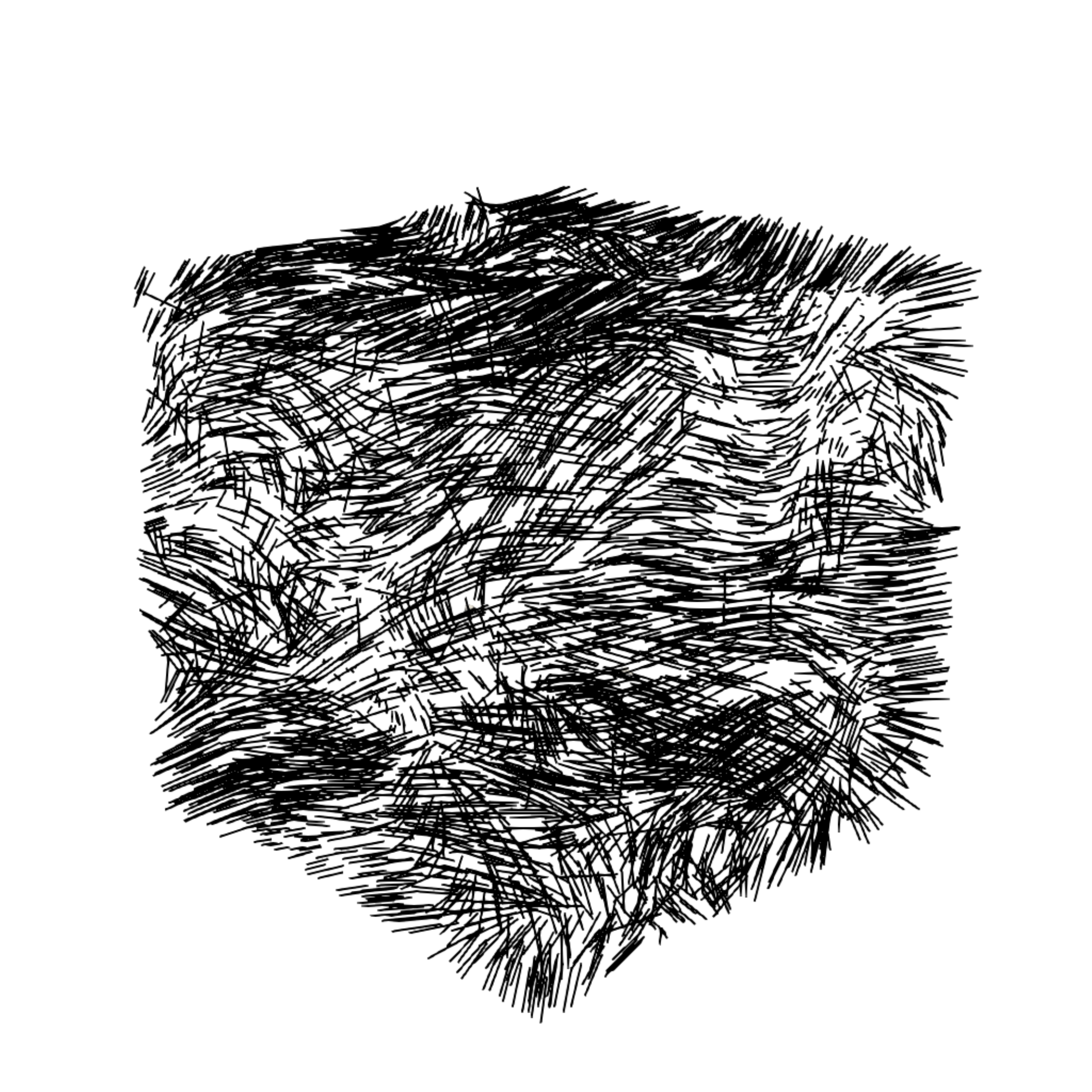}
\includegraphics[width = 0.23\textwidth]{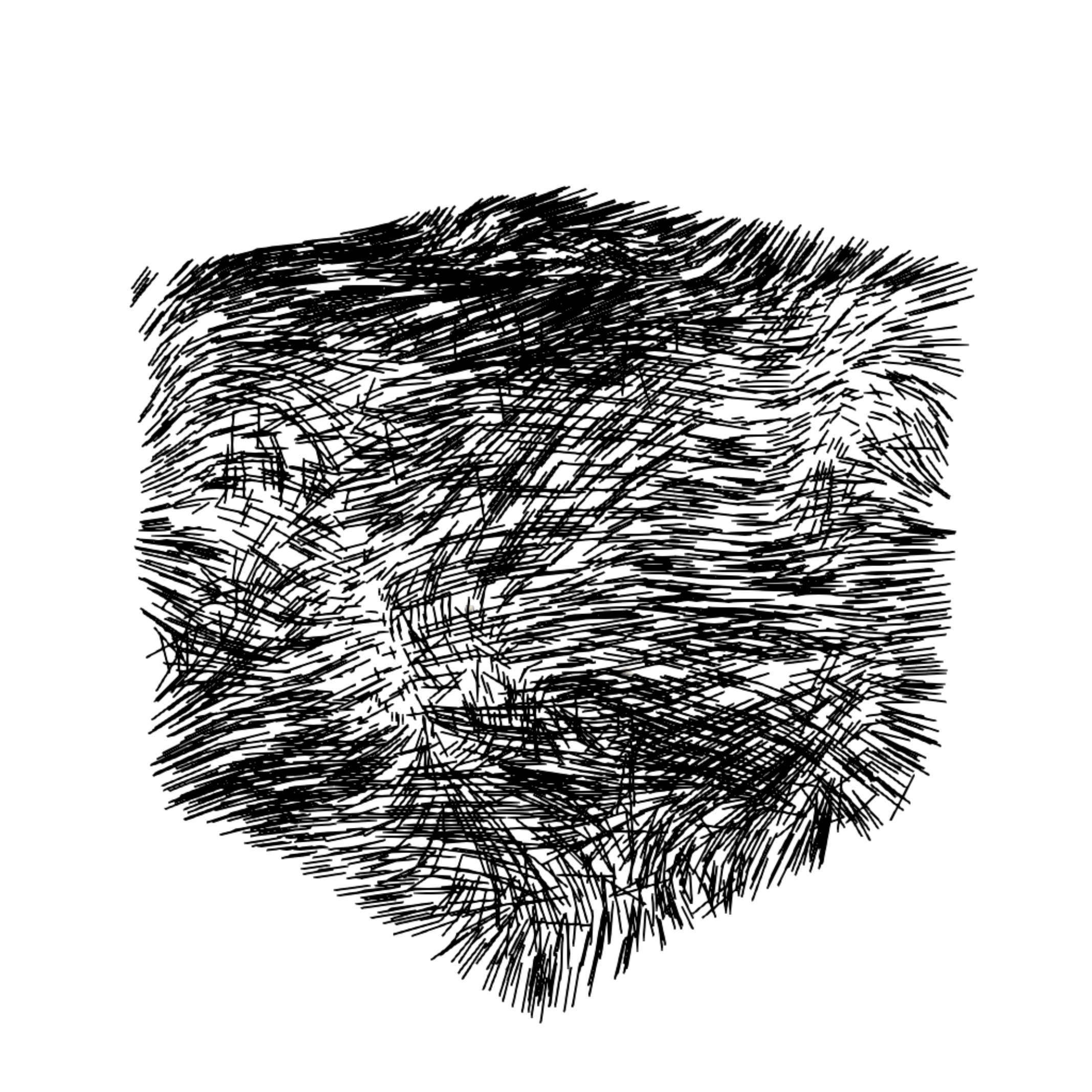}
\includegraphics[width = 0.23\textwidth]{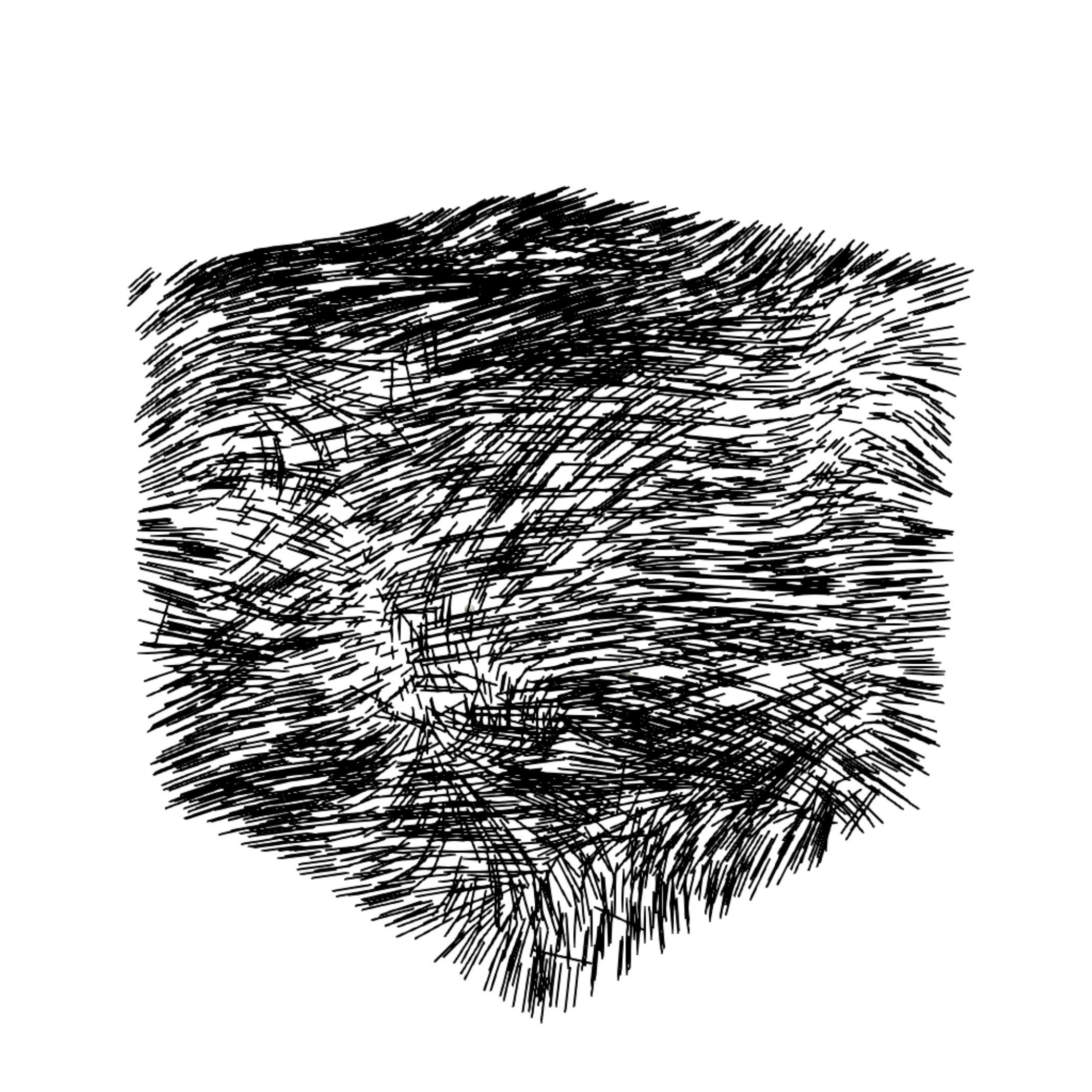}
\includegraphics[width = 0.23\textwidth]{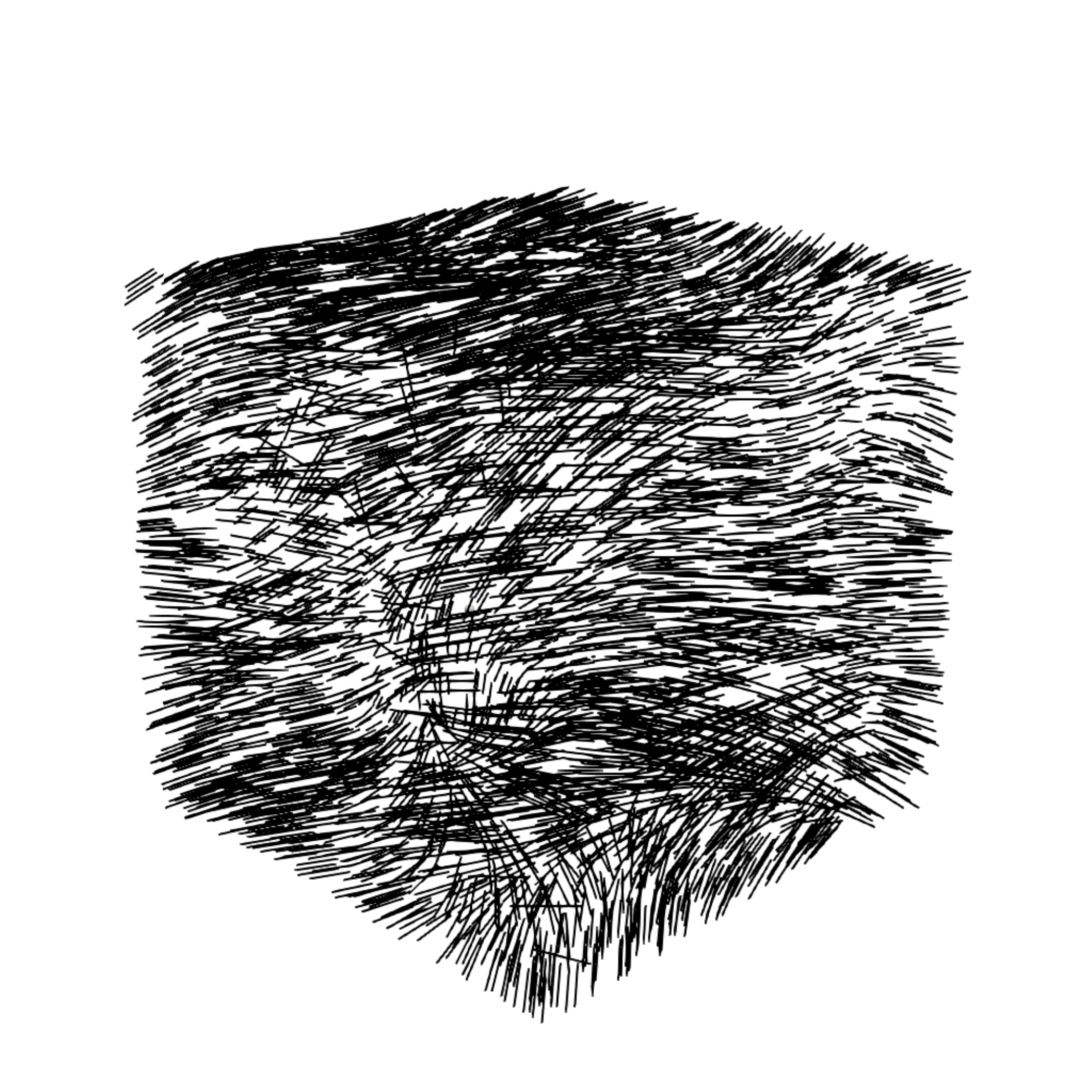}
\includegraphics[width = 0.23\textwidth]{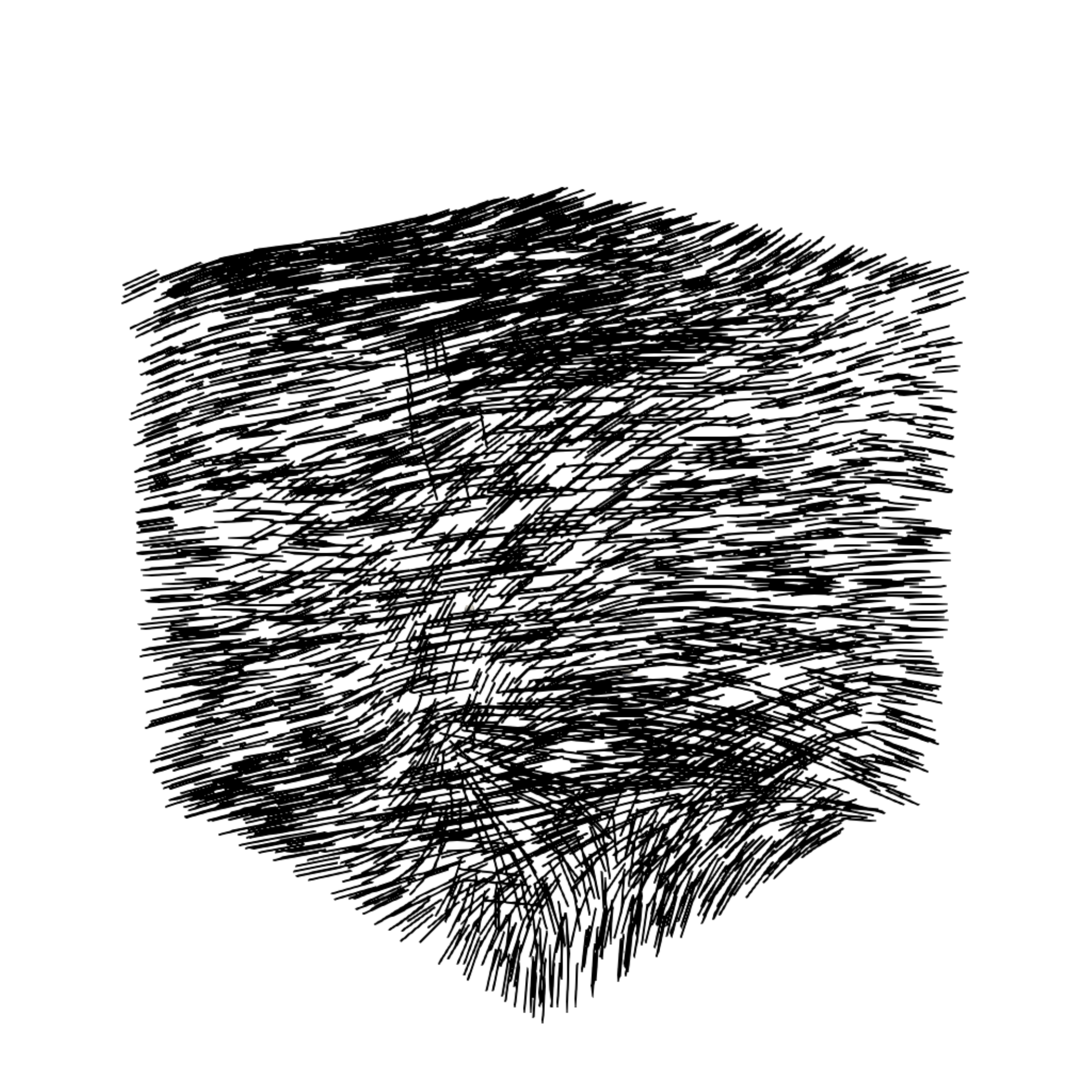}
\includegraphics[width = 0.23\textwidth]{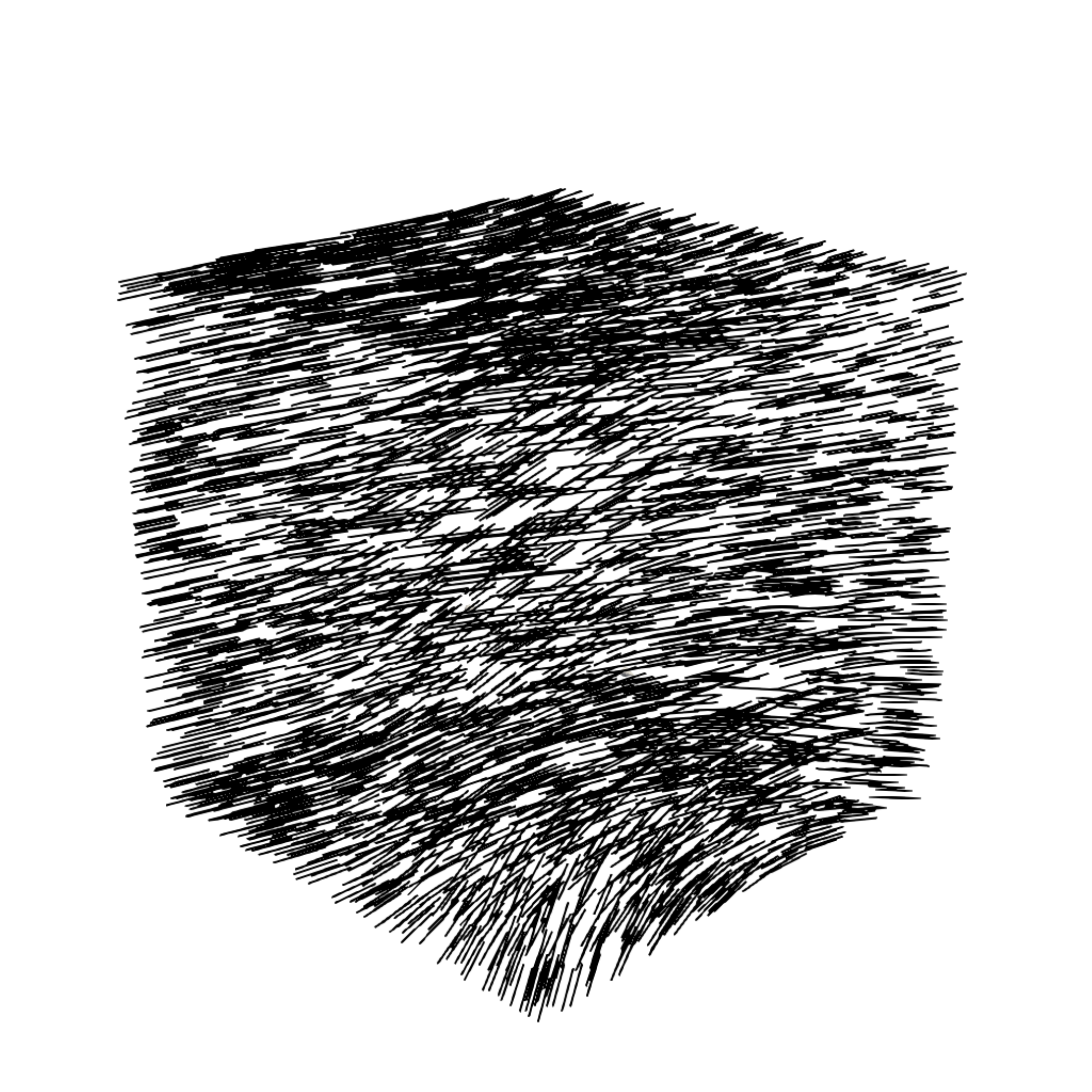}
\caption{\label{fig:randomvector3D} The eigenvector of $\Q$ corresponding to the largest eigenvalue at $t=0.0001,0.01,0.02,0.03,0.04,0.05,0.065,0.1,$ and $0.125$.}
\end{center}
\end{figure}

\begin{figure}[h]
\begin{center}
\includegraphics[height=4.75cm]{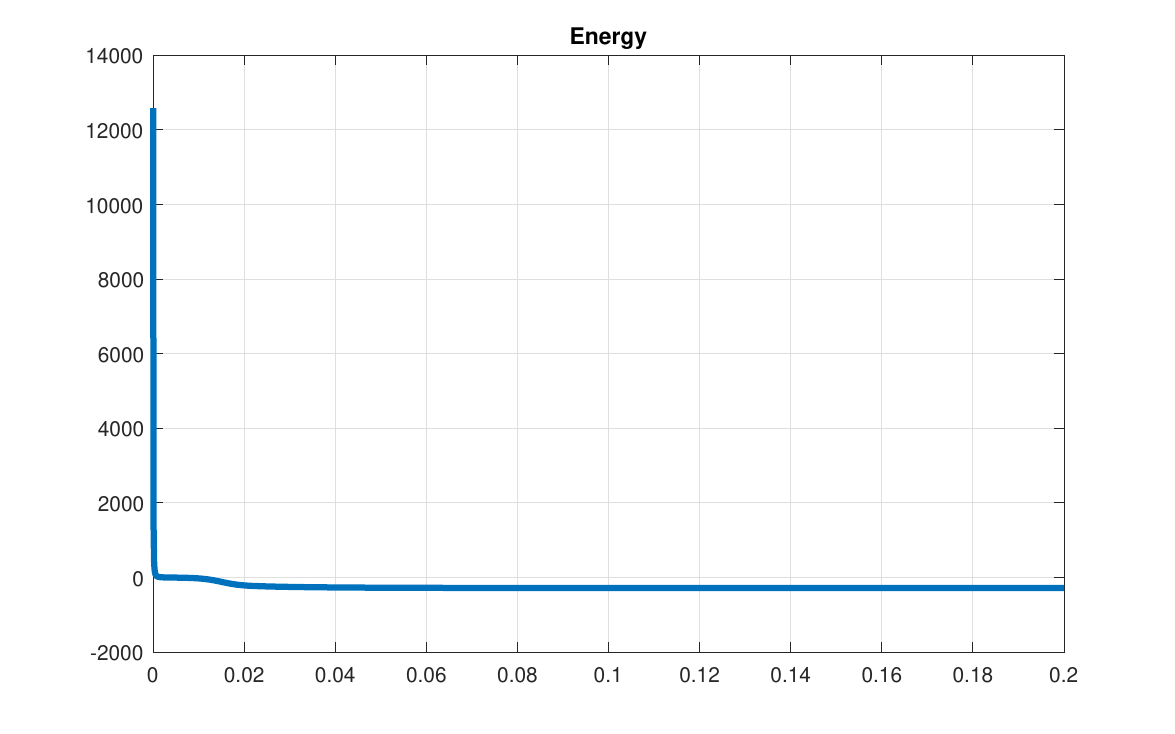}
\includegraphics[height=4.75cm]{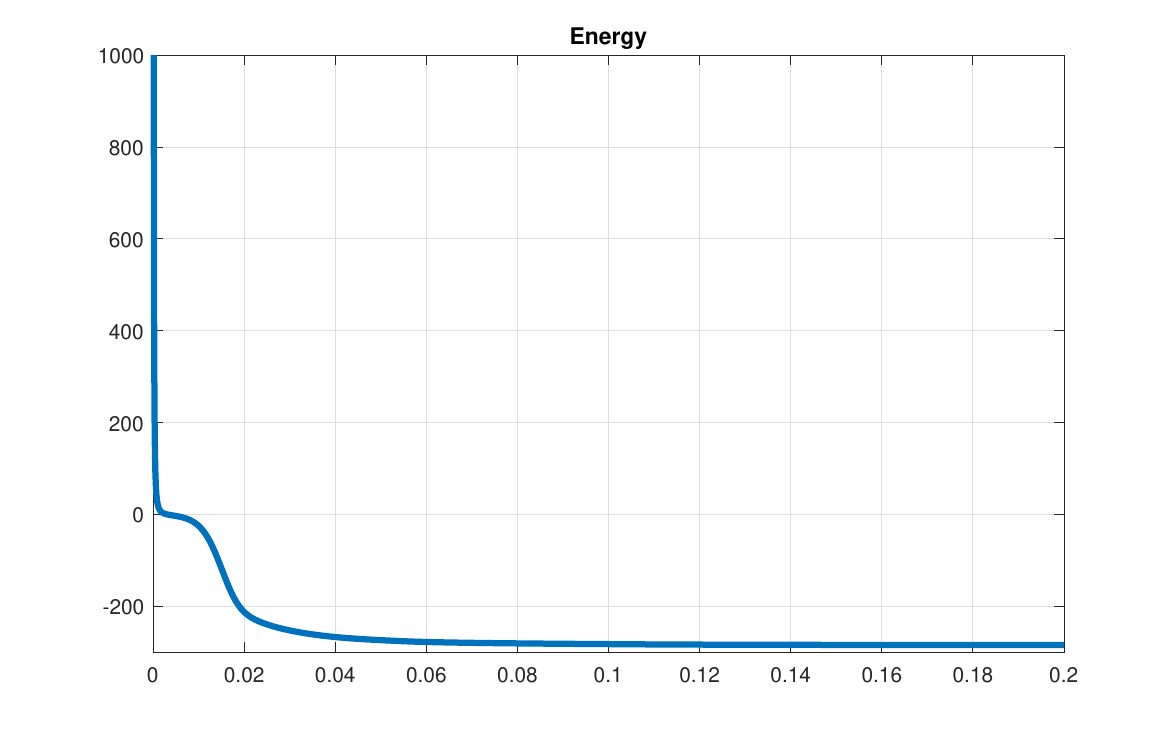}
\caption{\label{fig:energy3D}Time evolution of the energy 
for the three dimensional simulation computed using scheme OD1D. 
Right plot is the same curve in a limited y-axis window to better show the dynamics.
}
\end{center}
\end{figure}

\section{Conclusion}\label{sec:conclusion}
In this paper we have presented three different linear numerical schemes for a $\Q$-tensor problem which can be used to study the dynamics of nematic liquid crystals efficiently in two and three dimensions. To summarize (see \Cref{tab:schemeSummary}), scheme UES1D is an unconditionally energy stable decoupled scheme which is first order accurate in time. Next, we proposed scheme OD2C which is a second order accurate scheme that introduces considerably less numerical dissipation than scheme UES1D which allows for simulations more closely resembling the true dynamics of the system. Finally, scheme OD1D is a first order accurate decoupled scheme which retains the same second order numerical dissipation as in scheme OD2C while also being the most computationally efficient. \\
It is important to note that scheme UES1D is energy stable, however, from the simulations presented we see that this scheme is not ideal. The high numerical dissipation introduced by this scheme causes the simulated dynamics to be relatively slow which increases the computational cost. In addition, the truncation procedure further decreases the efficiency. Instead, it is scheme OD1D which we have shown to strike the best balance of accuracy and efficiency by presenting several simulations in two and three dimensions. This scheme provides researchers with a method for simulating nematic liquid, and a starting point for developing accurate and efficient numerical schemes for the larger Beris-Edwards system which incorporates hydrodynamic effects.
\begin{table}
\begin{center}
\begin{tabular}{|c|c|c|c|c|}
\hline
Scheme & Decoupled & Numerical Dissipation & Accuracy & Energy Stable \\
\hline 
\hline
1 & \checkmark & $\mathcal{O}(\dt)$ & $\mathcal{O}(\dt)$ & \checkmark \\
\hline
2 & & $\mathcal{O}(\dt^2)$ & $\mathcal{O}(\dt^2)$ & \\
\hline
3 & \checkmark & $\mathcal{O}(\dt^2)$ & $\mathcal{O}(\dt)$ & \\
\hline
\end{tabular}
\caption{\label{tab:schemeSummary} Summary of the properties of the numerical schemes presented in this work. }
\end{center}
\end{table}

\printbibliography

\section*{Appendix}\label{sec:appendix}
Here we write a few facts about derivatives that help to derive the terms involved in the numerical schemes. 
It is helpful to use repeated index notation so that 
$$
Q_{ij}Q_{jk} = \sum_j Q_{ij}Q_{jk}\,.
$$
In addition, we use the Kronecker delta function
$$
\delta_{ij}
=\left\{
\ba{lr}
1
& \mbox{ if }
i=j\,,
\\ \hueco
0
& \mbox{ if }
i\neq j\,.
\ea
\right.
$$
Below we make use of the following relations for second order tensors
$$
\dis 
A_{ij}\delta_{jk} = A_{ik} \, ,\quad \quad
A_{ij}\delta_{ik} = A_{kj} \quad \quad \mbox{and} \quad \quad 
\frac{\partial A_{ij}}{\partial A_{kl}} = \delta_{ik}\delta_{jl} \, .
$$
\subsection*{Derivatives of the potential function $\Psi(\Q)$}
Let $F:\R^{3\times 3}\rightarrow \R$ be a scalar valued function of a second order tensor $\Q$. Then $\bm{f}=\nabla_{\Q}F$ is a second order tensor valued function and the second derivative $\nabla \bm{f}$ is a fourth order tensor valued function with the following indexing
$$
\left[ \bm{f} \right]_{ij}
\,=\,
\dis \frac{\partial}{\partial Q_{ij}} f
\quad
\mbox{ and }
\quad
\left[ \nabla\bm{f} \right]_{ijkl}
\,=\,
\dis \frac{\partial}{\partial Q_{kl}} \left[\bm{f} \right]_{ij} \, .
$$
First, let $F(\Q)=\trace(\Q^2)$ as in the first term of the LdG potential function.
Then we can write $F$ as
$$
F(\Q) 
\,=\,
\trace(\Q^2)
\,=\,
Q_{ij}Q_{ij}
$$
and so the derivatives become
$$
\dis 
\left[\bm{f}\right]_{pq}
\,=\,
\frac{\partial}{\partial Q_{pq}} Q_{ij}Q_{ij} 
\,=\,
2 Q_{ij} \frac{\partial Q_{ij}}{\partial Q_{pq}} 
\,=\,
2 Q_{ij} \delta_{ip}\delta_{jq} 
\,=\,
2 Q_{iq}\delta_{ip} 
\,=\,
2 Q_{pq} \Rightarrow
\bm{f} \,=\, 2\Q\,.
$$
Then the second derivative is
$$
\dis 
\left[\nabla\bm{f}\right]_{ijkl}
\,=\,
2 \frac{\partial}{\partial Q_{kl}} Q_{ij}
\,=\,
2 \delta_{ik}\delta_{jl}\,. 
$$
Next, we let $F(\Q)=\trace(\Q^3)$. In repeated index notation this is
$$
F(\Q) 
\,=\,
\trace(\Q^3)
\,=\,
Q_{ij}Q_{jk}Q_{ki}\,.
$$
The first and second derivatives are
$$\dis 
\left[\bm{f}\right]_{pq}
\,=\,
\frac{\partial }{\partial Q_{pq}} Q_{ij}Q_{jk}Q_{ki}
\,=\,
Q_{pk}Q_{kq} + Q_{pi}Q_{iq} + Q_{pj}Q_{jq}
\,=\,
3 \left[\Q^2\right]_{pq} \Rightarrow
\bm{f} (\Q)
\,=\, 3\Q^2\,,
$$
$$
\left[\nabla\bm{f}\right]_{ijkl}
\,=\,
3 \frac{\partial }{\partial Q_{kl}} Q_{ip}Q_{pj}
\,=\,
3\left( \delta_{ik}\delta_{lp}Q_{pj} + Q_{ip}\delta_{pk}\delta_{jl} \right)
\,=\,
3\left( \delta_{ik}Q_{lj} + Q_{ik}\delta_{jl} \right)\,.
$$
Finally, let $F(\Q)=\left(\trace(\Q^2)\right)^2$. The first and second derivatives are
$$\dis 
\left[\bm{f} \right]_{pq}
\,=\,
\frac{\partial }{\partial Q_{pq}} \left(Q_{ij}Q_{ij}\right)^2 
\,=\,
2 \left(Q_{ij}Q_{ij}\right) \frac{d}{dQ_{pq}} \left(Q_{ij}Q_{ij}\right) 
\,=\,
4 Q_{pq}\left(Q_{ij}Q_{ij}\right) \Rightarrow 
\bm{f}(\Q) 
\,=\, 
4\Q\trace(\Q^2)\,. 
$$
$$
\left[ \nabla\bm{f}\right]_{ijkl}
\,=\,
4 \frac{\partial }{\partial Q_{kl}} \left[ Q_{ij}\left(Q_{pq}Q_{pq}\right) \right]
\,=\,
4\left( \left(Q_{pq}Q_{pq}\right)\delta_{ik}\delta_{jl} + 2Q_{ij}Q_{kl} \right)
\,=\,
4\left( \trace(\Q^2)\delta_{ik}\delta_{jl} + 2Q_{ij}Q_{kl} \right)\,.
$$
For $\bm{p}(\Q) = \frac13 \trace(\Q^2)\I$, we have the fourth order tensor $\nabla \bm{p}(\Q)$ which can be written as
$$
\ba{rcl}
\nabla \bm{p}(\Q) &=& \I \otimes \Q
\ea
$$
and
$$
\left[\nabla \bm{p}(\Q) \right]_{ijkl} =
\left\{ \ba{ll}
2Q_{kl} & \mbox{ if } i=j\,, \\
0 & \mbox{ otherwise}\,.
\ea \right.
$$
Hence when $|\Q|\leq\alpha$ we have
$$
\left|\left| \left[ \nabla\bm{p}(\Q) \right]_{ijkl} \right|\right|_{L^\infty} 
\,\leq\,
 2 \alpha \,, \quad \quad
\left|\left|\nabla\bm{p}(\Q)\right|\right|_F 
\,\leq\, 
108\alpha^2\,.
$$

\subsection*{From Cubic to Quadratic Growth}
Here we derive the bounds for the truncated potential function \eqref{eq:barPsiThree}. Let $\beta = (\alpha_2 - \alpha_1)^{-1}$. The first and second derivatives are
$$
\dis \left[\nabla \rho(\trace(\Q^2) \right]_{ij}
\,=\,
-12 \beta^2  \Q_{ij} \left(1 - \left( |Q|^2 - \alpha_1\right) \right) \left(|Q|^2 - \alpha_1 \right)
$$
and 
$$
\ba{rcl}
\dis \left[\bm{H}_\rho(\trace(\Q^2) \right]_{ijkl}
&=&
-12\beta^2\left[ 
-\beta\delta_{ik}\delta_{jl}\left(|Q|^2\right)^2
-\beta|Q|^2 Q_{ij}Q_{kl} \right.\\
\hueco
& & \quad \dis \left.
+(1+2\alpha_1\beta)\left( |Q|^2\delta_{ik}\delta_{jl} + Q_{ij}Q_{kl} \right)
-\alpha_1\beta(1+\alpha_1)\delta_{ik}\delta_{jl} \right] \,.
\ea
$$
Hence when $\alpha_1<|\Q|<\alpha_2$, we have $Q_{ij}\leq \alpha_2$ for all $i,j=1,2,3$, so
$$\dis
\left| \left[\nabla\rho(\Q) \right]_{ij} \right| \leq 12\beta^2\left( \alpha_2^5 + (1+2\alpha_1)\alpha_2^2 + \alpha_1^2 + \alpha_1\right)
$$
and
$$\dis
\left| \left[H_\rho \right]_{ijkl} \right| \leq 24\alpha_2^4\beta^3 + 3\alpha_2^2+6\alpha_1\alpha_2^2\beta+\alpha_1^2\beta+\alpha_1\beta\,.
$$
Therefore, we see that $\left| \left[H_\rho \right]_{ijkl} \right|$ grows without bound as $\alpha_2 \rightarrow \infty$ or $\beta \rightarrow \infty$.
Now we detail the derivatives of $\widehat\psi(\Q)$:
$$
\ba{rcl}\dis
\left[\bm{\widehat\psi}_3(\Q)\right]_{ij}
&=&
\left[\bm{\widehat\psi}_3(\Q)\right]_{ij} \rho(\trace(\Q^2)) + \widehat\Psi_3(\Q)\left[\nabla\rho(\trace(\Q^2))\right]_{ij} \\
\hueco \dis
&+& 2C^2Q_{ij}(1 - \rho(\trace(\Q^2)) - C^2Q_{pq}Q_{pq}\left[\nabla\rho(\trace(\Q^2))\right]_{ij}
\ea
$$
and 
$$
\ba{rcl}\dis
\left[ \nabla\bm{\widehat\psi}_3 (\Q) \right]_{ijkl}
&=&
\left[\bm{\widehat\psi}_3(\Q)\right]_{ij}\left[\nabla\rho(\trace(\Q^2))\right]_{kl} + \left[\bm{\widehat\psi}_3(\Q)\right]_{kl}\left[\nabla\rho(\trace(\Q^2))\right]_{ij} \\
\hueco \dis
&+& \rho(\trace(\Q^2)) \left[ \nabla\bm{\widehat\psi}_3(\Q) \right]_{ijkl} + \widehat\Psi_3(\Q) \left[ \bm{H}_{\rho} (\trace(\Q^2)) \right]_{ijkl} \\
\hueco \dis
&+& 2C^2\delta_{ik}\delta_{jl}(1 - \rho(\trace(\Q^2)) + 2C^2Q_{ij}\left[\bm{\widehat\psi}_3(\Q)\right]_{kl} \\
\hueco \dis
&-& 2C^2Q_{kl}\left[\bm{\widehat\psi}_3(\Q)\right]_{ij} - C^2Q_{pq}Q_{pq}\left[ \bm{H}_{\rho} (\trace(\Q^2)) \right]_{ijkl} \,.
\ea
$$
Since we have the following bounds for $\widehat\Psi_3$ when $\alpha_1 \leq |\Q|\leq \alpha_2$
$$\dis
\left| \left[\bm{\widehat\psi}_3(\Q) \right]_{ij} \right| \leq B\alpha_2^2
$$
and
$$\dis
\left| \left[ \nabla\bm{\widehat\psi}_3 \right]_{ijkl} \right| \leq 2B\alpha_2\,,
$$
we have the following bound for $\widehat\psi(\Q)$
$$\ba{rcl}\dis
\left| \left[ \nabla\bm{\widehat\psi}_3 \right]_{ijkl} \right| &\leq& 24B\alpha_2\beta \left( \alpha_2^5 + \alpha_2^2(1+2\alpha_1) + \alpha_1^2 + \alpha_1 \right) \\
\hueco
& & \dis
+ 2B\alpha_2\left( 3\beta\alpha_2^2 + 3\beta\alpha_1 + 2\alpha_2^4\beta^2 + 4\alpha_1\alpha_2^2\beta^2 + 2\alpha_1^2\beta^2 + 1\right) \\
\hueco 
& & \dis
+ 648B\alpha_2^7\beta^3 + 81B\alpha_2^5 + 162B\alpha_2^5\beta\alpha_1 + 27B\alpha_2^3\beta\alpha_1^2+27B\alpha_2^3\beta\alpha_1 \\
\hueco
& & \dis
+ 2C^2\left( 3\beta\alpha_2^2 + 3\beta\alpha_1 + 2\beta^2\alpha_2^4 + 2\beta^2\alpha_1^2+4\beta^2\alpha_2^2\alpha_1+2\right) \\
\hueco
& & \dis
+ C^2\alpha_2^2\left( 4B\alpha_2 + 24\alpha_2^4 + 3\alpha_2^2+6\alpha_2^2\alpha_1\beta+\alpha_1^2\beta+\beta\alpha_1\right)\,.
\ea
$$

In particular, if $\alpha = 1.18 < \alpha_1=1.19 < \alpha_2 = 1.2$, then 
$$
\left\| \left[ \nabla\bm{\widehat\psi}_3 \right]_{ijkl} \right\|_{L^\infty} \leq 3619 
\quad
\mbox{ and }
\quad
\| \nabla\bm{\widehat\psi}_3 \|_F \leq 32571\,.
$$

\end{document}